\documentclass[11pt,a4paper]{article}

\usepackage[utf8]{inputenc}
\usepackage{geometry}
\usepackage{amsmath}
\usepackage{amssymb}
\usepackage{mathtools}
\usepackage{hyperref}
\usepackage{amsthm}
\usepackage{enumitem}
\usepackage{tikz}
\usepackage{nccmath}
\usepackage{bm}

\usepackage{tikz-cd}

\usepackage{dsfont}

\usepackage{mathrsfs}
\usepackage{braket}
\usepackage[style=alphabetic]{biblatex}
\renewbibmacro{in:}{%
  \ifentrytype{article}
    {}
    {\bibstring{in}%
     \printunit{\intitlepunct}}}
\mathtoolsset{showonlyrefs}
\usetikzlibrary{decorations.pathreplacing}

\numberwithin{equation}{section}
\newcommand{\numberset}{\mathbb}
\newcommand{\N}{\numberset{N}}
\newcommand{\R}{\numberset{R}}

\newcommand{\baru}{u} 
\newtheorem{lemma}{Lemma}[section]
\newtheorem{theorem}{Theorem}

\newtheorem{prp}[lemma]{Proposition}
\newtheorem{cor}[lemma]{Corollary}
\newtheorem{defn}[lemma]{Definition}

\newtheorem{rmk}[lemma]{\textbf{Remark}}
\theoremstyle{remark}
\newtheorem*{remark}{\textbf{Remark}}

\title{Profile of a Touch-Down solution to a nonlocal MEMS model with critical parameters}
\author{Maissâ Boughrara\vspace{1em} \\
Fakult\"at f\"ur Mathematik, Universit\"at Bielefeld,\\
D-33501 Bielefeld, Germany.} 
\date{}

\begin{document}
\maketitle

\begin{abstract}
This work investigates a generalisation of a mathematical model arising in the study of MEMS devices, described by the following parabolic equation on $[0,T)\times\Omega$:
$$\partial_t v = \Delta v + \frac{\lambda}{(1-v)^{\bar p}\left( 1 + \gamma \int_{\Omega} \frac{1}{(1-v)^r}\, dx \right)^{q}} 
,
\qquad 0 \leq v < 1,$$
where $\Omega \subset \mathbb{R}^N$ is a bounded domain and $\lambda, \gamma, \bar p, q, r > 0$, in some critical regime for the parameters. We construct a solution with a prescribed profile, which quenches in finite time $T$ at exactly one interior point $a \in \Omega$. Moreover, we are able to provide an asymptotic description of the quenching profile.

We reformulate the problem as a blow-up problem to utilize the techniques employed in \cite{MZ1997DMJ}, \cite{DZ2019} and \cite{DGKZ2022}. 
The proof proceeds through two principal steps: a reduction to a finite-dimensional dynamical system and a classical topological argument employing index theory. 
The main challenge lies in managing the nonlocal integral term, which generates an additional gradient term when the problem is transformed into the blow-up framework. This situation is more delicate than the subcritical regime considered in \cite{DZ2019}. The critical setting in \cite{DGKZ2022}, is useful for us, though not enough. Indeed, we introduce new idea in this paper, in particular, for the dynamics in similarity variables.
\end{abstract}
\textbf{Keywords:} {MEMS, touch-down phenomenon, quenching solution, asymptotic profile, nonlocal partial differential equations.

\section{Introduction}

Micro-electro-mechanical systems (MEMS) are electronic devices that enable the translation of a mechanical signal, such as a sound wave, into an electrical signal (see the book by Kavallaris and Suzuki \cite{KS2018}). This system is usually very small, between 1 and 100 micrometers. MEMS devices are essential in modern technology such as telecommunications, biomedical engineering, space exploration, etc.

Roughly speaking, a MEMS device consists in a rigid plate and a flexible membrane, both connected to a current source and a capacitor.

One model that describes the behavior of such device is a hyperbolic PDE (see Guo and Souplet \cite{GS2015} as well as Kavallaris and Suzuki \cite{KS2018}) given by the following:

\begin{equation}\label{eqH}
\begin{cases}
    &\varepsilon^2 \partial_{tt} v+\partial_t v = \Delta v    +
\displaystyle \frac{ f(x,t) }{  (1 - v)^2 \left( 1 + \gamma
    \displaystyle \int_{\Omega} \frac{1}{1-v} dx \right)^2}
,\ x\in\Omega, t>0,\\
&v(x,t)=0, x\in \partial \Omega, t>0,\\
&v(x,0)=v_0(x), x\in \overline \Omega,
\end{cases}
\end{equation}
where  $\varepsilon$ is the ratio of the interaction due to the inertial and damping terms, $v\in [0,1]$ is the deflection of the membrane, $f(x,t)>$0 is the varying dielectric permittivity of the elastic plate, and $\Omega$ is a $C^\infty$ domain of $\R^2$ and is the shape of the elastic plate. Note that in this model, the rigid plate is located at position $v\equiv 0$. In fact, assuming the dumping term is high, we may take $\varepsilon=0$, and reduce to  standing for the following parabolic equation:
\begin{equation}\label{eq2}
\begin{cases}
    &\partial_t v = \Delta v    +
\displaystyle \frac{ \lambda}{  (1 - v)^2 \left( 1 + \gamma
    \displaystyle \int_{\Omega} \frac{1}{1-v} dx \right)^2}
,\ x\in\Omega, t>0,\\
&v(x,t)=0, x\in \partial \Omega, t>0,\\
&v(x,0)=v_0(x), x\in \overline \Omega.
\end{cases}
\end{equation}
Note that we take here $f\equiv \lambda>0$. Up to replacing $v(x,t)$ by $v(\lambda^{\frac{1}{2}}x,\lambda t)$, we may assume that $\lambda=1$. In fact, we will consider the following generalization of the problem \eqref{eq2}.
\begin{equation}\label{eq}
\begin{cases}
    &\partial_t v = \Delta v    +
\displaystyle \frac{ 1}{  (1 - v)^{\bar p} \left( 1 + \gamma
    \displaystyle \int_{\Omega} \frac{1}{(1-v)^r} dx \right)^q}
,\ x\in\Omega, t>0,\\
&v(x,t)=0, x\in \partial \Omega, t>0,\\
&v(x,0)=v_0(x), x\in \overline \Omega,
\end{cases}
\end{equation}
where $\bar p,q,r>0$.

This problem arises in the study of the dynamic deflection of the elastic membrane. During normal operation of the device, the membrane does not touch the rigid plate, i.e. $v(x,t)<1$ for all $x\in \Omega, t\geq 0$. If ever $v(x_0,t_0)=1$ for some $x_0\in \Omega$ and $t_0\geq 0$, this means that the membrane touches the rigid plate and the MEMS device is damaged for good. Such phenomenon is called "Touch-down". As mathematicians, we would like to understand this phenomenon to avoid it and to get more reliable devices. From a mathematical point of view, Touch-down happens if there exists $T<+\infty$ and
\begin{equation}\label{touch down}
   \underset{t\rightarrow T}{\lim \inf} [\min_{x\in \overline\Omega}1-v(x,t) ]=0.
\end{equation}
Moreover, we say that $a\in \Omega$ is a Touch-down point if and only if there exist a sequence $(a_n,t_n)\in \Omega\times (0,T)$ such that

$$(a_n, t_n) \rightarrow (a,T) \text{ and } v(a_n, t_n)\rightarrow 1, \text{ as } n\rightarrow +\infty.$$

This Touch-down phenomenon can be interpreted as a form of finite-time quenching, where a physical quantity reaches zero in finite time. In this case, that quantity is the gap between the membrane and the rigid plate, given by $1-v$.

In this work, we construct a solution which quenches in finite time called $T$ only and only at one point $a\in \Omega$ and give its quenching profile. Many have studies the quenching problem for this model. We may cite \cite{G2014}, \cite{GH2009}, \cite{GK2012}. In the work of Duong and Zaag \cite{DZ2019}, the authors constructed a solution that quenches in the subcritical regime, i.e. when $\frac{N}{2}>\frac{1}{\bar p+1}, r=1, q>0$, with the following profile:
\begin{equation}
1-v(x,t)
\sim  (T-t)^\frac{1}{\bar p+1}\left(\bar p+1+\frac{(\bar p+1)^2}{4\bar p}\frac{|x|^2}{(T-t)|\ln(T-t)|} \right)^{\frac{1}{\bar p+1}}, \text{ as } t\rightarrow T.
\end{equation}
However, the existence of a quenching solution specifically in the critical regime has not yet been established and remains an open question. 
\subsection{Statement of the result}
 
Here, we are interested in proving a general exitance quenching result, in the critical regime with no restriction on any $\gamma > 0$ and for a $C^\infty$ bounded domain $\Omega$. More precisely, we consider the problem \eqref{eq} in the case where
\begin{equation}\label{critical condition}
    \frac{N}{2}=\frac{r}{\bar p+1} \text{ and } q<\frac{2}{N}.
\end{equation}
In this regard, we introduce
\begin{equation}\label{def theta infty and beta}
\left\{
\begin{aligned}
    &\theta_\infty=\left(\frac{(\frac{N}{2}+1)2b^\frac{N}{2}}{\gamma (1-q\frac{N}{2})}\right)^{\frac{q}{1-q\frac{N}{2}}},\\
& \beta=\frac{q(\frac{N}{2}+1)}{1-q\frac{N}{2}},
\end{aligned}
\right.
\end{equation}
and
\begin{equation}
b=\frac{(\bar p+1)^2(1+\beta)}{4\bar p}.
\end{equation}
The following is our main result.
\begin{theorem}\label{thm profile}

Let $\Omega$ be a smooth and bounded domain in $\mathbb{R}^N$ containing the origin and consider equation \eqref{eq} in the critical regime \eqref{critical condition}. Then, there exists $\varepsilon_0$ small such that for all $rq\gamma\in [0,\varepsilon_0]$, there exist a non-negative initial data $v_0 \in C^\infty(\bar \Omega)$ and $T\in \R^+_*$ such that the solution $v$ of \eqref{eq} belongs to $ C^{2,1}(\Omega \times [0, T)) \cap C(\bar{\Omega} \times [0, T))$ and quenches  in finite time at $T$, only at the origin. Moreover, we have the following quenching asymptotics:

\begin{enumerate}[label=(\roman*)]

\item \textbf{The intermediate blowup profile.} For all $t \in (0,T)$, we have
\begin{align}\label{profile result}
\Bigg\| \left(\frac{\theta_\infty(T-t)}{|\ln(T-t)|^{\beta} }\right)^\frac{1}{\bar p+1}\frac{1}{1-v(.,t) }
- & \left(\bar p+1+b \frac{|\cdot|^2}{(T-t)|\ln(T-t)|} \right)^{-\frac{1}{\bar p+1}}
\Bigg\|_{L^\infty(\Omega)}\\
&\leq \frac{C}{1+\sqrt{|\ln(T-t)|}}.
\end{align}

\item \textbf{The final blowup profile.}\label{item final profile result} It holds that $v(x,t) \to v^*(x) \in C^2(\Omega \setminus \{0\})$ 
as $t \to T$, uniformly on compact sets of $\Omega \setminus \{0\}$. In particular, we have
\begin{equation}\label{final profile result}
1-v^*(x) \sim 
\left[\theta_\infty b \frac{|x|^2}{(2|\ln |x||)^{1+\beta}} \right]^{\frac{1}{\bar p+1}}.
\end{equation}
\item \textbf{Behavior of} $\theta(t) =\left( 1 + \gamma \int_{\Omega} \frac{1}{(1- v(x,t))^r} dx \right)^{-q}$.
It holds that
\begin{equation}\label{asymp theta result}
\theta(t) = \theta_\infty |\ln(T-t)|^{-\beta} \left( 1 + O\!\left( \frac{1}{\sqrt{|\ln(T-t)|}} \right)\right), 
\quad \text{as } t \to T.
\end{equation}
\end{enumerate}
\end{theorem}

\begin{remark}
When $q = 0$, our problem reduces to the setting considered in the works of Filippas and Guo \cite{FG1993}, as well as Merle and Zaag \cite{MZ1997Nonlinearity}. Therefore, the case $q \ne 0$ represents the novel and meaningful contribution of our study, where the main challenge lies in the treatment of the nonlocal term.
\end{remark}

The approach is based on certain links between quenching and blow-up phenomena. Specifically, we reformulate the problem in terms of constructing a blow-up solution for a related PDE and then analyze its behavior. This strategy draws inspiration from the works of Bricmont and Kupiainen \cite{BK1994}, Merle and Zaag \cite{MZ1997Nonlinearity}, Duong and Zaag \cite{DZ2019} for the same equation but in the subcritical regime, and Duong, Zaag, Kavallaris and ghoul in \cite{DGKZ2022}, for the Gierer-Meinhardt system in the critical regime with more restrictive condition on $\bar p$. Indeed, this latter problem shares with ours the presence of a non-local term, which is in our case:
\begin{equation}
    \theta(t) =\left( 1 + \gamma
    \displaystyle \int_{\Omega} \frac{1}{(1- v(x,t))^r} dx \right)^{-q}.
\end{equation} 
Thus, we observe that once $v$ quenches, $\theta$ will affect the solution's asymptotics.

Let us discuss the influence of the non-local term in the literature. In \cite{DZ2019}, the authors constructed a quenching solution. However, in the that work, the authors only treated the case where the non-local term stays away from $0$ and infinity, i.e.
$$\theta(t) \rightarrow \theta^*> 0 \text{ as } t \rightarrow T.$$
Thus, $v$ can be seen as solution to
$$\partial_t v\sim\Delta v+\frac{\theta^*}{(1-v)^{\bar p}}.$$
However, in our work, as in \cite{DGKZ2022}, we treat the case where the non-local term vanishes due to the critical regime, i.e.
\begin{equation}\label{bar theta to 0}
\theta(t) \rightarrow 0 \text{ as } t \rightarrow T.
\end{equation}
which clearly has an affect on the nonlinear term of equation \eqref{eq}, in the sense that
$$\partial_t v\sim\Delta v+o\left( \frac{ 1}{  (1 - v)^{\bar p}}\right),$$
which makes the description of solution's behavior more complicated.

\subsection{Setting of the problem and strategy}\label{Setting problem}
We recall our goal in this work. We want to construct a solution for \eqref{eq}, defined for all $(x,t)\in \Omega\times [0,T)$, for some $T>0$ with $0\leq u(x,t)<1$, and
\begin{equation}\label{conv u to 1}
v(x,t)\rightarrow 1 \text{ as } (x,t)\rightarrow (0,T).
\end{equation}
We introduce
\begin{equation}\label{def u}
    \baru=\frac{1}{1-v}-1=\frac{v}{1-v}.
\end{equation}
This way, constructing a solution $v$ such that \eqref{conv u to 1} holds is equivalent to constructing a solution $\baru $ such that 
\begin{equation*}
    \baru(x,t)\rightarrow +\infty \text{ as } (x,t)\rightarrow (0,T).
\end{equation*}
Using \eqref{eq}, we write the following equation for $u$:
\begin{equation}\label{eq u}
\begin{cases}
    &\partial_t \baru = \Delta \baru - 2\frac{|\nabla \baru|^2}{\baru+1}    +
\displaystyle  \theta(1 +\baru)^{p},\\
&\baru(x,t)=0, x\in \partial \Omega, t>0,\\
&\baru(x,0)=\baru_0(x), x\in \overline \Omega,
\end{cases}
\end{equation}
where $p=\bar p+2$, $u_0$ is given from $v_0$ in \eqref{eq} with \eqref{def u} and 
\begin{equation}\label{def theta}
    \theta(t) =\left\{
    \begin{aligned}
    &\left( 1 + \gamma
    \displaystyle \int_{\Omega} (1+ \baru)^r dx \right)^{-q}, &\text{ for } t\geq 0,\\
    &\frac{\theta(0)|\log T|^\beta}{|\log (T-t)|^\beta}, &\text{ for } t< 0.
    \end{aligned}
    \right.
\end{equation}
\begin{remark}
Note that this choice of definition for $\theta$ for negative times ensures its continuity for all times.
\end{remark}
We emphasize that equation \eqref{eq u} can be viewed as a special case of the more general equation given below.

\begin{equation}
\partial_t \baru = \Delta \baru -a\frac{|\nabla \baru|^2}{\baru+1}    +
\displaystyle  \theta(1 +\baru)^{p}.
\end{equation}
This equation coincides with the one studied in \cite{DGKZ2022} when $a=0$. In contrast, our setting involves $a=2$, which creates additional challenges in controlling the gradient term. To overcome this difficulty, and following the strategy in \cite{DZ2019}, we rely on the general method developed in \cite{MZ1997DMJ} for the construction of blow-up solutions to the semilinear heat equation. In a rough sense, this corresponds to the case where the last term on the right-hand side satisfies $\theta \equiv 1$. That construction relies on two stages:
\begin{itemize}
\item A formal approach, where we understand the dynamics of equation \eqref{eq u} in different regions of space, leading to the determination of a good candidate of a blow-up profile.
\item A rigorous approach, where we linearize the equation around the candidate for the profile and show the existence of a solution converging to $0$, after some appropriate scaling. Using the spectral informations about  the linearized operator, we first handle the projection of the negative part of the spectral, which happens to be infinite dimensional, with finite codimension. This way, we reduce the question to a finite dimensional problem, which is solved thanks to a topological argument from the degree theory.
\end{itemize}

\begin{remark}
Note that in our work, we impose the condition $\bar p>0$ which means $ p > 2 $, which is a relaxed assumption compared to \cite{DGKZ2022}, where the authors assumed $p>3$.
\end{remark}

This paper is organized as follows:
\begin{itemize}
\renewcommand\labelitemi{-}
\item \textbf{Section \ref{Section formal proof}:} We reduce our problem to the blow-up case in order to apply well-known formal arguments for this situation. This allows us to derive the behavior of $\theta$ formally.
\item \textbf{Section \ref{Formulation of the problem}:} We rigorously formulate the problem and introduce a shrinking set to zero.
\item \textbf{Section \ref{Section existence without technic}:} We proof the existence of a solution in the shrinking set, without technical computatations.
\item \textbf{Section \ref{Section proof thm}:} We give the proof of Theorem \ref{thm profile}.
\item \textbf{Section \ref{Section proof existence of u in S}:} We give the proofs of the technical results used in Section \ref{Section existence without technic}.
\end{itemize}
\textbf{Acknowledgement :} The author would like to express her sincere gratitude to Hatem Zaag for his valuable guidance, insightful comments, and constructive suggestions, which greatly contributed to the development and improvement of this work. This paper was written as part of the PHC Sakura project number 53291ZB and the Labex MME-DII DA\_PROJECT\_2024\_5. This research was also funded in whole or in part by the Austrian Science Fund (FWF) 10.55776/PAT5825523.
\section{Formal approach}\label{Section formal proof}
In this section, we adopt a formal approach to find a good candidate for the blow-up profile for \eqref{eq u}. For this, we partition $\Omega$ into three regions, each characterized by a distinct dynamical behavior of \eqref{eq u}, and devote separate subsections to the analysis of each of these behaviors. Then, we derive in the last subsection the behaviour of $\theta$ (see \eqref{def theta}) using the formal results obtained in each region. For this, we suggest one more scaling by introducing 

\begin{equation}\label{def U}
U(x,t)=\theta(t)^{{\frac{1}{p-1}}}\baru(x,t),
 \end{equation}
where $\theta$ is given in \eqref{def theta}. Using equation \eqref{eq u}, we find that $U$ satisfies the following equation
\begin{equation}\label{eq U}
\begin{cases}
    &\partial_t U = \Delta U - 2\frac{|\nabla U|^2}{U+\theta ^{{\frac{1}{p-1}}}}   +\left(U+\theta^{{\frac{1}{p-1}}}\right)^{p}+\frac{1}{p-1}\frac{\theta'}{\theta}U,\\
&U(x,t)=0, x\in \partial \Omega, t>0,\\
&U(x,0)=U_0(x), x\in \overline \Omega.
\end{cases}
\end{equation}
Our aim is to find a blow-up solution $U$, such that \eqref{bar theta to 0} holds. Our idea is to add the following assumption:
\begin{equation}\label{theta/theta' small u^p-1}
\frac{\theta'}{\theta }\sim o(U^{p-1}) \text{ and } \theta\sim 0,
\end{equation}
obtaining
\begin{equation}\label{theta/theta'U small U^p}
\frac{\theta'}{\theta }U\ll U^{p},
\end{equation}
so that the last term of the right-hand side of \eqref{eq U} is neglegeable with the third term. Equation \eqref{eq U} can be considered as a small perturbation of 
\begin{align}\label{formal equation}
\partial_t U = \Delta U - 2\frac{|\nabla U|^2}{U }   +U^{p}.
\end{align}
We know, from \cite{MZ1997Nonlinearity} and \cite{HZ2024}, how to construct a solution for this type of equation such that 
\begin{equation}\label{U equiv phi (T-t)}
U(x,t)\sim \left(p-1+b\frac{|x|^2}{(T-t)|\log (T-t)|}\right)^{-\frac{1}{p-1}} (T-t)^{-\frac{1}{p-1}}\text{ as } t\to T,
\end{equation}
for some $b>0$. Therefore, it is reasonable to replace \eqref{theta/theta'U small U^p} by the following: 
\begin{equation}\label{theta'/theta to (T-t)**-1}
\frac{\theta'(t)}{\theta(t)}\ll (T-t)^{-1} \text{ as } t\rightarrow T.
\end{equation}
Because of \eqref{U equiv phi (T-t)} and in order to handle the boundary effect at $\partial\Omega$, we shall decompose our domain $\Omega \subset \R^N$ as it follows: For $ K_0 > 0 $,  $\epsilon_0 > 0$ and $t \in [0, T) $ with $ T > 0$, we define
\begin{equation}\label{def P}
\begin{aligned}
&P_1(t) = \left\{ x \in \mathbb{R}^n \, \middle| \, |x| \leq K_0 \sqrt{(T-t) |\log(T-t)|} \right\}, \\
&P_2(t) = \left\{ x \in \mathbb{R}^n \, \middle| \, \frac{K_0}{4} \sqrt{(T-t) |\log(T-t)|} \leq |x| \leq \epsilon_0 \right\},\\
&P_3 = \left\{ x \in \Omega \, \middle| \, |x| \geq \frac{\epsilon_0}{4} \right\}.
\end{aligned}
\end{equation}
Notice that we have
$$
\Omega = P_1(t) \cup P_2(t) \cup P_3(t), \quad \text{for all } t \in [0, T).$$
Our aim is controlling our problem on $ P_1(t) $, $ P_2(t) $ and $ P_3(t) $.
\subsection{In region $P_1$}
We define the similarity variables as the following
\begin{equation}\label{similarity variables}
    W(y,s)=(T-t)^{\frac{1}{p-1}}U(x,t), s=-\log(T-t), y=\frac{x}{\sqrt{T-t}}.
\end{equation}
From \eqref{eq U}, we write the equation satisfied by $W$ as follows
\begin{equation}\label{eq W}
\begin{cases}
    &\partial_s W = \Delta W-\frac{1}{2}y.\nabla W-\frac{1}{p-1}W -\displaystyle  2\frac{|\nabla W|^2}{\displaystyle W+\bar\theta^{{\frac{1}{p-1}}}e^{-\frac{s}{p-1}}}    +\left(W+\bar\theta(s)^{{\frac{1}{p-1}}}e^{-\frac{s}{p-1}}\right)^p\\
    &\hspace{1.2cm}+\frac{1}{p-1}\frac{\bar\theta'(s)}{\bar\theta(s) }W,\\
&W(y,s)=0, x\in \partial \Omega_s, s>-\log T,\\
&W(y,-\log T)=W_0(y), y\in \overline \Omega_s,
\end{cases}
\end{equation}
where
\begin{equation}\label{def theta bar}
\bar\theta(s)=\theta(T-e^{-s}),\  \Omega_s=e^{\frac{s}{2}}\Omega,
\end{equation}
and $W_0$ is given from $U_0$ in \eqref{eq U} with \eqref{similarity variables}. Then, \eqref{def theta bar} and \eqref{theta'/theta to (T-t)**-1} yield to
 $$\frac{\bar\theta'(s)}{\bar\theta(s)}\rightarrow 0 \text{ as } s\rightarrow +\infty.$$
 Let us make this more precise by requiring that 
 \begin{equation}\label{theta'/theta}
\frac{\bar\theta'(s)}{\bar\theta(s)}=-\frac{\beta}{s}+O(s^{-1-\delta}),
\end{equation}
for some $\beta$. Notice that the case where 
\begin{equation}
\left|\frac{\bar\theta'(s)}{\bar\theta(s)}\right|\sim \frac{C}{s^\delta},
\end{equation}
for $\delta\neq 1$ is impossible. Indeed, if $\delta>1$, then $\theta(s)\rightarrow C>0$, and if $\delta<1$, then $\theta(s)\rightarrow +\infty$, as $s\rightarrow +\infty$. Both cases contradict the fact that $\theta(s) \longrightarrow 0$ as $s\rightarrow +\infty$, as required in \eqref{bar theta to 0}. Returning back to \eqref{theta'/theta}, we see that it has a special solution that
\begin{equation}\label{asymp bar theta in s}
\bar\theta(s) = \theta_\infty s^{-\beta} \left( 1 + O\left( \frac{1}{s^\delta} \right) \right), \quad \text{as } s \to +\infty.
\end{equation}
In compliance with \eqref{bar theta to 0}, we require that
\begin{equation}\label{beta > 0}
    \beta>0.
\end{equation}
Moreover, we naturally require
\begin{equation}\label{theta infty > 0}
    \theta_\infty>0.
\end{equation}
This, with \eqref{similarity variables}, implies

\begin{equation}\label{asymp bar theta in t}
\theta(t) = \theta_\infty |\log(T - t)|^{-\beta} \left( 1 + O\left( \frac{1}{|\log(T - t)|^\delta} \right) \right), \quad \text{as } t \to T.
\end{equation}
This will be rigorously justified in Subsection \ref{subsection theta dynamics}, where the constants $\beta$ and $\delta$ will be determined. With this latter, we observe that

\begin{equation}\label{theta'/theta W DL}
\frac{\bar{\theta}'}{\bar{\theta}}  W = - \left(\frac{\beta}{s} + \text{``lower order terms''}\right)W.
\end{equation}
Since linear terms with a coeffision of order $\frac{1}{s}$ may change the dynamics of differential equations (Think about the equation $v'=\lambda v -\frac{\beta}{s}v$), then we cannot neglect the term $\frac{\theta'}{\theta}  W$. Hence, plugging \eqref{theta'/theta W DL} into \eqref{eq W}, we are interested in considering the following proxy equation:
$$\partial_s W = \Delta W-\frac{1}{2}y.\nabla W-\frac{1}{p-1}W -\displaystyle  2\frac{|\nabla W|^2}{\displaystyle W}    +W^{p}-\frac{\beta}{(p-1)s}W.$$
Note that there exists a space-independent solution in the following form:
$$
W = \kappa + O\left( \frac{1}{s} \right),
$$
Consequently, we consider the linearization around the dominated part $\kappa$ by

$$
\tilde{W} = W - \kappa,
$$
which is a solution of
\begin{equation}\label{eq tilde W}
\partial_s \tilde{W} = \mathcal{L} \tilde{W} + \mathcal{B}(\tilde{W}) +\mathcal{R}(\tilde{W}) - \frac{\beta\tilde{W}}{(p-1)s}- \frac{\beta \kappa}{(p-1)s},
\end{equation}
where
$$
\mathcal{L} = \Delta - \frac{1}{2} y \cdot \nabla + I d,
$$

$$
\mathcal{B}(\tilde{W}) = (\tilde{W} + \kappa)^p - \kappa^p - p \kappa^{p-1} \tilde{W},
$$
$$\mathcal{R}(\tilde{W})=-\displaystyle  2\frac{|\nabla \tilde{W}|^2}{\displaystyle \tilde{W}+\kappa}.$$
Note that:
\begin{itemize}
    \item $\mathcal{B}$ is quadratic and $\left| \mathcal{B}(\tilde{W}) - \frac{p}{2\kappa} \tilde{W}^2 \right| \leq C |\tilde{W}|^3.$
    \item $\mathcal{R}$ is also quadratic with respect to $\nabla \tilde W$, if $\tilde W$ is small.
    \item We have a forcing term $- \frac{\beta \kappa}{s}$.
    \item A linear term with small coefficient $- \frac{\beta\tilde{W}}{s}$, which may perturb the dynamics of the main term $\partial_s \tilde{W}$.
\end{itemize}
Moreover, the differential operator $\mathcal{L}$ is a self-adjoint operator on a domain included in $L^2_\rho(\R^N)$, where $\rho(y)= \frac{e^{-\frac{|y|^2}{4}}}{(4\pi)^{\frac{N}{2}}}$. The spectrum of this operator consists of the following eigenvalues:
\begin{equation}\label{spec L}
spec\ \mathcal{L}=\{1-\frac{n}{2}, n\in \N\}.
\end{equation}
The corresponding eigenfunctions for $\beta=(\beta_1,...,\beta_N)\in \N^N$, with $\ |\beta|=n$, are
\begin{equation}\label{def H_beta}
H_\beta: y=(y_1,...,y_N)\mapsto h_{\beta_1}(y_1)...h_{\beta_N}(y_N), \text{ for } |\beta|=\beta_1+...+\beta_N,  
\end{equation}
where $h_{\beta_i}$ are the (rescaled) Hermite polynomials given by
\begin{equation}
h_{\beta_i}(y_i)=\underset{j=0}{\overset{[\beta_i/2]}{\sum}}\frac{\beta_i!}{j!(\beta_i-2j)!}(-1)^jy_i^{\beta_i-2j},
\end{equation}
and $[\beta_i/2]$ is the floor value of $\beta_i/2$. Notice that these polynomials satisfy 
\begin{equation}
\int_{\R}h_{\beta_i}(\xi)h_{\beta_j}(\xi)\rho(\xi) d\xi=2^{\beta_i}\beta_i!\delta_{\beta_i,\beta_j}.
\end{equation}
Since the eigenfunctions of $\mathcal{L}$ span the whole space $L^2_\rho(\R^N)$, we may look for a solution expanded as follows:
$$
\tilde W(y,s) = \sum\limits_{\beta \in \mathbb{N}^N} \tilde W_{\beta}(s) H_\beta(y).
$$
Assuming that $\tilde W$ radial and since the eigenfunctions for $|\beta| \geq 3$ correspond to negative eigenvalues of $\mathcal{L}$, we may consider only
$$
\tilde W(y,s) = \tilde W_0(s) + \tilde W_2(s) \sum\limits_{i=1}^{N} h_2(y_i) = \tilde W_0(s) + \tilde W_2(s)(|y|^2 - 2N),
$$
with $\tilde W_0, \, \tilde W_2 \to 0$ as $s \to \infty.$ Plugging this into \eqref{eq tilde W}, we obtain the following equation system
\begin{equation}
\begin{split}
&\tilde W_0'=\tilde W_0+\frac{p}{2\kappa}(\tilde W_0^2+8N\tilde W_2^2)-\frac{16N}{\kappa}\tilde W_2^2-\frac{\beta}{(p-1)s}(\kappa+\tilde W_0)+O(|\tilde W_0|^3+|\tilde W_2|^3),\\
&\tilde W_2'=\frac{4p}{\kappa}\tilde W_2^2+\frac{p}{\kappa}\tilde W_0\tilde W_2-\frac{8}{\kappa}\tilde W_2^2-\frac{\beta}{(p-1)s}\tilde W_2+O(|\tilde W_0|^3+|\tilde W_2|^3).
    \end{split}
\end{equation}
We assume $$\tilde W_0\sim \frac{\beta\kappa}{(p-1)s}\hspace{0.5cm}  \text{ and}\hspace{0.5cm} \tilde W_2\sim \frac{\gamma}{s},$$
for some $\gamma\in \R$. We obtain that the ODE system is satisfied if
$$\gamma=-\frac{\kappa(1+\beta)}{4(p-2)}.$$
and we recall that $p=\bar p+2>2$. Therefore, the ODE is solved by 
\begin{equation}
    \begin{split}
        &\tilde W_0=\frac{\beta\kappa}{(p-1)s}+o\left(\frac{1}{s}\right),\\
        &\tilde W_2= -\frac{\kappa(1+\beta)}{4(p-2)s}+o\left(\frac{1}{s}\right).
    \end{split}
\end{equation}
Thus,
\begin{equation}\label{W inner expansion}
    \begin{split}
        W&=\kappa-\frac{\kappa(1+\beta)}{4(p-2)}\frac{(|y|^2-2N)}{s}+\frac{\beta\kappa}{(p-1)s}+o\left(\frac{1}{s}\right)\\
        &=\kappa-\frac{\kappa(1+\beta)}{4(p-2)}\frac{|y|^2}{s}-\frac{\kappa N(1+\beta)}{(p-2)s}+\frac{\beta\kappa}{(p-1)s}+o\left(\frac{1}{s}\right)
    \end{split}
\end{equation}
in $L^2_\rho(\R^N)$ and also uniformly on compact sets by standard parabolic regularity. However, we need to go further in space and not just in compact sets. Since \eqref{formal equation} is similar to the semilinear heat equation, we are motivated to seek for a blow-up solution of the form
$$W(y,s)=\varphi_b\left(\frac{y}{\sqrt{s}}\right)+\text{"lower perturbation"},$$
where 
\begin{equation}\label{expression phi profile}
\varphi_b(z)=\left(p-1+b|z|^2\right)^{-\frac{1}{p-1}},
\end{equation}
for some $b>0$. Using a Taylor expansion, we have that
\begin{equation}
\varphi_b\left(\frac{y}{\sqrt{s}}\right)=\kappa-\frac{b\kappa}{(p-1)^2}\frac{|y|^2}{s}+O\left(\frac{|y|^4}{s^2}\right).
\end{equation}
Matching this expansion with \eqref{W inner expansion} with the one just obtained, we get
\begin{equation}\label{expression b}
b=\frac{(p-1)^2(1+\beta)}{4(p-2)}.
\end{equation}
Thus, we deduce
$$W(y,s)\sim \varphi(y,s),$$
where
\begin{equation}\label{def phi}
    \varphi(y,s)=\varphi_b\left(\frac{y}{\sqrt{s}}\right)+\frac{a}{s},
\end{equation}
and
\begin{equation}\label{def a}
    a=\frac{\kappa N(1+\beta)}{2(p-2)}+\frac{\beta\kappa}{(p-1)}.
\end{equation}
This formal result is adapted from \cite{DGKZ2022} and \cite{TZ2019}, where in the latter the authors obtained the $W^{1,\infty} $ estimate
\begin{equation}\label{result TZ2019}
    \|W - \varphi\|_{W^{1,\infty}} \leq \frac{C}{1+\sqrt{|\log(T-t)|}}.
\end{equation}
We assume that \eqref{result TZ2019} hold. Thus, we derive
\begin{equation} \label{P1 estimation}
\left| (T-t)^{\frac{1}{p-1}} \theta(t)^{{\frac{1}{p-1}}} u(x,t) - \varphi_b\left( \frac{|x|}{\sqrt{(T-t)|\log(T-t)|}} \right) \right| \leq \frac{C}{1+\sqrt{|\log(T-t)|}}, 
\end{equation}
\begin{equation}\label{P1 estimation grad}
|\nabla u(x,t)| \leq \frac{C (T-t)^{-\frac{1}{p-1} - \frac{1}{2}}\theta(t)^{-{\frac{1}{p-1}}}}{1+\sqrt{|\log(T-t)|}},
\end{equation}
for all $t\in [0,T)$ and $ x \in P_1(t)$.

\subsection{In region $P_2$}\label{subsec formal P2}
We define $\mathcal{U}$ as follows: For all $x \in P_2(t)$, $\xi \in (T - t(x))^{-\frac{1}{2}} (\bar{\Omega} - x)$ and $\tau \in \left[-\frac{t(x)}{T - t(x)}, 1 \right)$, we define the following:
\begin{equation}\label{def mathcal U}
    \mathcal{U}(x,\xi,\tau)=\varrho(x)^{\frac{1}{p-1}}\theta(t(x))^{\frac{1}{p-1}}u(x+\xi\sqrt{\varrho(x)},\varrho(x)\tau+t(x)),
\end{equation}
where $t(x)$ is defined as the solution of the following equation:
\begin{equation}\label{def rho(x), t(x)}
|x| = \frac{K_0}{4} \sqrt{\varrho(x)|\log\varrho(x)|}, \text{ where } \varrho(x) = T - t(x) \text{ and } t(x) < T.
\end{equation}
Note that since we are in $P_2$, we have that 
{\mathtoolsset{showonlyrefs=false}
\begin{equation}\label{t(x)<t}
t(x) \leq t.
\end{equation}
}
Thus, whenever $|x|\geq \frac{K_0}{4}\sqrt{T|\log T|}$, we have that
\begin{equation}\label{t(x)<0}
t(x)\leq 0.
\end{equation}
In addition to that, using \eqref{def rho(x), t(x)}, we have the following asymptotic
\begin{equation}\label{asym rho}
\varrho(x) \to 0 \quad \text{as} \quad x \to 0,
\end{equation}
which is equivalent to
\begin{equation}\label{t(x) to T}
t(x) \to T \quad \text{as} \quad x \to 0.
\end{equation}
Using \eqref{eq u} and \eqref{def mathcal U}, we write the equation satisfied by $\mathcal{U}$ as follows:
\begin{equation}\label{eq mathcal U}
\partial_\tau \mathcal{U} = \Delta_{\xi} \mathcal{U} - 2 \frac{|\nabla_{\xi} \mathcal{U}|^2}{\mathcal{U} +(\theta(t(x))\varrho)^{\frac{1}{p-1}}} +\frac{\tilde\theta}{\theta(t(x))}\left( \mathcal{U} + (\theta(t(x))\varrho)^{\frac{1}{p-1}} \right)^{p} ,\\
\end{equation}
where
\begin{equation}\label{def tilde theta}
\tilde{\theta}(\tau) = \theta(\tau \varrho(x) + t(x)), 
\end{equation}
with $\theta$ is given in \eqref{def theta}. We study the dynamic of $\mathcal{U}$ on a small region of the local space 
\begin{equation}\label{domain xi}
|\xi|\leq \alpha_0\sqrt{|\log \varrho(x)|}, \text{ for some } \alpha_0>0.
\end{equation}
We define $z=\frac{x+\xi\sqrt{\varrho(x)}}{\sqrt{\varrho(x)|\log(\varrho(x))|}}$. From definition \eqref{def rho(x), t(x)}, we have

\begin{equation}\label{bound z-K0/4}
\begin{split}
\left||z|-\frac{K_0}{4}\right|&\leq \frac{\alpha_0}{|\log\varrho(x)|^{\frac{1}{4}}}, \text{ for all } \xi\leq \alpha_0|\log\varrho(x)|^{\frac{1}{4}}.\\
\end{split}
\end{equation}
Therefore, for $x$ small enough with \eqref{asym rho}, we have that $|z|\leq K_0$. Together with the fact that 
\begin{align}\label{t=t(x)}
\text{when } \tau =0,\text{ we have }t=t(x),
\end{align}
 this implies
$$x+\xi\sqrt{|\log(\varrho(x))|}\in P_1(t) \text{ when } \tau=0.$$
Thus, with \eqref{P1 estimation}, we obtain
\begin{equation}\label{profile mathcal U0}
\underset{\xi\leq \alpha_0|\log(\varrho(x))|^{\frac{1}{4}}}{\sup}\left|\mathcal{U}(x,\xi,0)-\varphi_b\left(z\right)\right|\leq \frac{C}{\sqrt{|\log\varrho(x)|}}.
\end{equation}
In other words, with \eqref{bound z-K0/4} and for small $x$, the initial data $\mathcal{U}(x,\xi,0)$ is considered a perturbation of $\varphi_b\left(\frac{K_0}{4}\right)$. Now, we define 
\begin{equation}\label{def mathcal V}
\mathcal{V}=\frac{1}{\mathcal{U} +\theta(t(x))^{{\frac{1}{p-1}}}\varrho^{\frac{1}{p-1}}},
\end{equation}
which satisfies the following equation:
\begin{equation}\label{eq V}
\partial_\tau \mathcal{V} = \Delta_{\xi} \mathcal{V}- \frac{\tilde\theta}{\theta(t(x))}\frac{1}{\mathcal{V}^{p-2}}.
\end{equation}
Our aim is to show that the flatness is preserved for all
$\tau\in [0; 1)$, where the solution does
not depend on the space variable. Thanks to \eqref{asymp bar theta in t}, \eqref{def theta} and \eqref{t(x) to T}, we have
\begin{equation}\label{theta(t(x)) to 0}
    \theta(t(x))\rightarrow 0 \text{ as } x\rightarrow 0.
\end{equation}
Following all this, with \eqref{asym rho},  $\mathcal{V}$ is considered as a perturbation of $\hat{\mathcal{V}}$, where $\hat{\mathcal{V}}$ is the solution of 
\begin{equation}\label{eq hat mathcal V}
\left\{
\begin{aligned}
&\partial_\tau \hat{\mathcal{V}} =-\frac{\tilde\theta}{\theta(t(x))}\frac{1}{\hat{\mathcal{V}}^{p-2}},\\
&\hat{\mathcal{V}}(0)=\varphi_b\left(\frac{K_0}{4}\right)^{-1}=\left( p-1 + b \frac{K_0^2}{16} \right)^{\frac{1}{p-1}},
\end{aligned}
\right.
\end{equation}
and is given by
\begin{equation}\label{def V hat} 
\hat{\mathcal{V}}(\tau)=\left[(p-1)\left(1-\int^\tau_0\frac{\tilde\theta(\sigma)}{\theta(t(x))}d\sigma \right)  + b \frac{K_0^2}{16}\right]^{\frac{1}{p-1}}.
\end{equation}
From \eqref{theta'/theta}, \eqref{similarity variables} and \eqref{asymp bar theta in t}, we have
$$\theta'(t)\sim -\frac{\theta_\infty\beta}{(T-t)|\log(T-t)|^{\beta+1}},$$
which implies with \eqref{beta > 0} and \eqref{theta infty > 0} that $\theta$ is decreasing. Henceforth, with \eqref{asymp bar theta in t}, we get 
\begin{equation}\label{bounds tilde theta/theta(t(x))}
0\leq\frac{\tilde\theta(\sigma)}{\theta(t(x))}\leq 1, \text{ for all } \sigma\in [0,1),
\end{equation}
which implies that
\begin{equation}\label{bound hat mathcal V formal}
    \left(b \frac{K_0^2}{16}\right)^{\frac{1}{p-1}}\leq \hat{\mathcal{V}}\leq \left( p-1 + b \frac{K_0^2}{16}\right)^{\frac{1}{p-1}}.
\end{equation}
Since it is reasonnable to see $\mathcal{V}$ as a perturbation of $\hat{\mathcal{V}}$, it follows that

$$ \frac{1}{C}\leq|\mathcal{V}(x,\tau)| \leq C \quad \text{ for all } |\xi| \leq  \alpha_0|\log \varrho(x)|^\frac{1}{4}, \text{ for some } \alpha_0>0.$$
which implies with \eqref{def mathcal V}, \eqref{asym rho} and \eqref{theta(t(x)) to 0} that
$$ \frac{1}{C}\leq|\mathcal{U}(x,0,\tau)| \leq C. $$
In addition to that, from the flatness of $\hat{\mathcal{U}}$, it is reasonable to assume the following estimate on the gradient, thanks to some parabolic regularity.
$$ |\nabla_{\xi} \mathcal{U}(x,0,\tau)| \leq \frac{C}{\sqrt{|\log \varrho(x)|}}. $$
Finally, with definition \eqref{def mathcal U} and using Lemma \ref{lemma equiv varrho},  to derive the following asymptotics on $P_2(t)$:  

\begin{equation}\label{P2 estimation}
\left\{
\begin{aligned}
    \frac{1}{C} \left[ |x|^2 \right]^{-\frac{1}{p-1}} |\log |x||^{\frac{1}{p-1}} \theta^{-\frac{1}{p-1}}(t(x))\leq \baru(x,t) &\leq C \left[ |x|^2 \right]^{-\frac{1}{p-1}} |\log |x||^{\frac{1}{p-1}} \theta^{-\frac{1}{p-1}}(t(x)), \\
    |\nabla \baru(x,t)| &\leq C (|x|^2)^{-\frac{1}{p-1} - \frac{1}{2}} |\log |x||^{\frac{1}{p-1}} \theta^{-\frac{1}{p-1}}(t(x)).
\end{aligned}
\right.
\end{equation}
for all $t\in [0,T)$ and $ x \in P_2(t)$.
\subsection{In Region $P_3$}
Since the initial data and its gradient are bounded in $ P_3 $, we choose the blow-up time $ T $ to be sufficiently small so that the right-hand side of equation \eqref{eq u} remains bounded. Consequently, we obtain the following in $ P_3 $:
\begin{equation}\label{P3 estimation}
-\frac{1}{2}\leq \baru(x,t)\leq C, |\nabla\baru(x,t)|\leq C,
\end{equation}
for all $t\in [0,T)$ and $ x \in P_3(t)$.
\subsection{On the $\theta$'s dynamic}\label{subsection theta dynamics}
Here, we justify \eqref{theta'/theta} using the behavior of the solution in the regions $P_1,P_2$ and $P_3$ previously obtained by formal proofs. Our approach is inspired by the work in \cite{DGKZ2022} and \cite{DKZ2021} for similar and the same equation in the subcritical regime. Here, we extend their results to the critical case $ \frac{r}{p-1} = \frac{N}{2} $ and  with a sharper evaluation of integrals allowing us to extend the result to the range $p>2$ and not just $p>3$. In \cite{DGKZ2022}, the authors used a decomposition of the integrals following $\sqrt{(T-t)}|\log(T-t)|$. Here, we do the same following $\sqrt{(T-t) }|\log(T-t)|^\frac{p+1}{4}$. For this matter, we multiply \eqref{eq u} by $r(1+\baru)^{r-1}$ and integrating over $\Omega$, obtaining
\begin{equation}\label{decomp 1+baru on regions}
    \partial_t[\|1+\baru\|_{L^r}^r]=-r(r+1)\int_\Omega |\nabla \baru|^2(1+\baru)^{r-2}+r\theta\int_\Omega (1+\baru)^{ p +r-1}.
\end{equation}
For all $t\in [0,T)$, we decompose the first integral on the right-hand side as follows:
\begin{equation}\label{decomposition 1+u p+r-1}
\begin{split}
\int_\Omega (1+u)^{ p +r-1}= &\int_{|x| \leq \frac{K_0}{4} \sqrt{(T-t) }|\log(T-t)|^\frac{p+1}{4}} (1+\baru)^{ p +r-1}\\
&+\int_{\frac{K_0}{4}\sqrt{(T-t) }|\log(T-t)|^\frac{p+1}{4}\leq |x| \leq \epsilon_0}(1+\baru)^{ p +r-1}\\
&+\int_{\varepsilon_0\leq |x|,\  x\in \Omega}(1+\baru)^{ p +r-1}.
\end{split}
\end{equation}
Note that we expect \eqref{P1 estimation} to be valid on a larger domain and not just in $P_1(t)$. Thanks to \eqref{asymp bar theta in t}, we get in particular
\begin{equation}
(1+u(x,t))=\theta^{-\frac{1}{p-1}}(t)(T-t)^{-\frac{1}{p-1}}\left[\varphi_b\left(\frac{|x|}{\sqrt{(T-t)|\log(T-t)|}}\right) +O\left(\frac{1}{1+\sqrt{|\log(T-t)|}}\right)\right]. 
\end{equation}
From \eqref{critical condition}, $p>2$ and $r>0$, we write
\begin{equation}
\begin{aligned}
    \Bigg| (1+u(x,t))&^{p-1+r} - \theta^{-(\frac{N}{2}+1)}(t)(T-t)^{-(\frac{N}{2}+1)} 
    \varphi_b^{p-1+r} \left(\frac{|x|}{\sqrt{(T-t)\log(T-t)}}\right)\Bigg| 
    \\
    &\leq C\   \theta^{-(\frac{N}{2}+1)}(t)(T-t)^{-(\frac{N}{2}+1)} 
     |\log(T-t)|^{-\frac{p-1}{2}(\frac{N}{2}+1)} \\
    &\quad + C\ \theta^{-(\frac{N}{2}+1)}(t)(T-t)^{-(\frac{N}{2}+1)}  
    |\log(T-t)|^{-\frac{1}{2}} \varphi_b^{p-2+r} 
    \left(\frac{|x|}{\sqrt{(T-t)\log(T-t)}}\right),
\end{aligned}
\end{equation}
where we used the following fundamental inequality, valid for any $\alpha\geq 1, a\geq 0$ and $b\geq -a$,
\begin{equation}\label{fund ineq}
    |(a+b)^\alpha-a^\alpha|\leq C(\alpha)(b^\alpha+ba^{\alpha-1}).
\end{equation}
Switching to polar coordinates and reminding that $p>2$, we get
\begin{equation}\label{integration 1+baru in P1}
\begin{split}
    &\int_{|x| \leq \frac{K_0}{4} \sqrt{(T-t) }|\log(T-t)|^\frac{p+1}{4}}(1+\bar u)^{p +r-1}\\
    =&\theta^{-(\frac{N}{2}+1)}(t)(T-t)^{-1}|\log(T-t)|^{\frac{N}{2}}\int_{0}^{\frac{K_0}{4}|\log(T-t)|^\frac{p-1}{4}} \varphi_b^{p +r-1}\left(\xi\right)\xi^{N-1}d\xi\\
    &+O\left(\theta^{-(\frac{N}{2}+1)}(t)(T-t)^{-1}|\log(T-t)|^{\frac{N}{2}-\frac{p-1}{2}}\right)\\
    &+O\left(\theta^{-(\frac{N}{2}+1)}(t)(T-t)^{-1}|\log(T-t)|^{\frac{N}{2}-\frac{1}{2}}\int_{0}^{\frac{K_0}{4}|\log(T-t)|^\frac{p-1}{4}} \varphi_b^{p-2 +r}\left(\xi\right)\xi^{N-1}d\xi\right)\\
    =&\theta^{-(\frac{N}{2}+1)}(t)(T-t)^{-1}|\log(T-t)|^{\frac{N}{2}}\int_{0}^{\frac{K_0}{4}|\log(T-t)|^\frac{p-1}{4}} \varphi_b^{p +r-1}\left(\xi\right)\xi^{N-1}d\xi\\
    &+O\left(\theta^{-(\frac{N}{2}+1)}(t)(T-t)^{-1}|\log(T-t)|^{\frac{N}{2}-\frac{1}{2}}\right).\\
    \end{split}
\end{equation}
Combining \eqref{P2 estimation}, \eqref{bounds tilde theta/theta(t(x))}, \eqref{asymp bar theta in t}, \eqref{critical condition}, Lemma \ref{lemma fund integral} and the facts that $p>2$, $r>0$, we get

\begin{equation}\label{integration 1+baru in P2}
\begin{split}
&\int_{ \frac{K_0}{4} \sqrt{(T-t) }|\log(T-t)|^\frac{p+1}{4}\leq |x| \leq \epsilon_0}  |1+\baru(x,t)|^{ p +r-1}dx\\
&\leq C\left((T-t)^\frac{N}{2}|\log(T-t)|^{\frac{p+1}{4}N}+\theta^{ -(\frac{N}{2}+1)}(t)\int_{ \frac{K_0}{4}\sqrt{(T-t) }|\log(T-t)|^\frac{p+1}{4}\leq |x| \leq \epsilon_0} \frac{[|x|^2]^{-\frac{N}{2}-1}}{ |\log |x||^{-\frac{N}{2}-1}}d x\right)\\
&\leq C \theta^{-(\frac{N}{2}+1)}(t)(T-t)^{-1}|\log(T-t)|^{\frac{N}{2}-\frac{p-1}{2}}\\
&\leq C \theta^{-(\frac{N}{2}+1)}(t)(T-t)^{-1}|\log(T-t)|^{\frac{N}{2}-\frac{1}{2}}.
\end{split}
\end{equation}
In region $P_3$, we have from \eqref{P2 estimation},
\begin{equation}\label{integration 1+baru in P3}
\int_{\epsilon_0\leq |x|, x\in \Omega}  |1+\baru|^{ p +r-1}\leq C.
\end{equation}
Combining \eqref{decomposition 1+u p+r-1}, \eqref{integration 1+baru in P1},  \eqref{integration 1+baru in P2},  \eqref{integration 1+baru in P3} and \eqref{asymp bar theta in t}, we obtain
\begin{equation}\label{integral 1+baru result}
\begin{split}
\int_\Omega (1+\baru)^{ p +r-1} =&\theta^{-(\frac{N}{2}+1)}(t)(T-t)^{-1}|\log(T-t)|^{\frac{N}{2}}\int_{0}^{\frac{K_0}{4}|\log(T-t)|^\frac{p-1}{4}} \varphi_b\left(\xi\right)^{p +r-1}\xi^{N-1}d\xi\\
    &+O\left(\theta^{-(\frac{N}{2}+1)}(t)(T-t)^{-1}|\log(T-t)|^{\frac{N}{2}-\frac{1}{2})}\right)
\end{split}
\end{equation}
Again, we decompose the following integral:
\begin{equation}\label{decomposition nabla u 1+u r-2}
\begin{split}
\int_\Omega |\nabla \baru|^2(1+\baru)^{r-2}= &\int_{|x| \leq \frac{K_0}{4} \sqrt{(T-t) |\log(T-t)|}}|\nabla \baru|^2(1+\baru)^{r-2}\\
&+\int_{\frac{K_0}{4} \sqrt{(T-t) |\log(T-t)|}\leq |x| \leq \epsilon_0}|\nabla \baru|^2(1+\baru)^{r-2}\\
&+\int_{\varepsilon_0\leq |x|,\  x\in \Omega}|\nabla \baru|^2(1+\baru)^{r-2}.
\end{split}
\end{equation}
In region $P_1$, it follows from \eqref{P1 estimation} and \eqref{asymp bar theta in t}, for all $k\in \R$, the following
\begin{equation}\label{1+baru k}
(1+\bar u(x,t))^{k}\leq\theta^{-\frac{k}{p-1}}(t)(T-t)^{-\frac{k}{p-1}}\left[\varphi_b^{k}\left(\frac{|x|}{\sqrt{(T-t)|\log(T-t)|}}\right) +O\left(\frac{1}{1+\sqrt{|\log(T-t)|}}\right)\right]. 
\end{equation}
Therefore, from \eqref{P1 estimation grad} and switching again to polar coordinates, we obtain
\begin{equation}\label{integration nabla baru 1+baru in P1}
\begin{split}
    \int_{|x| \leq \frac{K_0}{4} \sqrt{(T-t) |\log(T-t)|}} |\nabla \baru|^2(1+\baru)^{r-2}
&\leq C\ \theta^{-\frac{N}{2}}(t)(T-t)^{-1}|\log(T-t)|^{\frac{N}{2}-1}.
\end{split}
\end{equation}
We assume $r\geq 2$ and we combine \eqref{P2 estimation}, \eqref{critical condition} \eqref{bounds tilde theta/theta(t(x))}, \eqref{asymp bar theta in t} and Lemma \ref{lemma fund integral}. Then,
\begin{equation}\label{integration nabla baru 1+baru in P2 r>2}
\begin{split}
&\int_{ \frac{K_0}{4} \sqrt{(T-t)|\log(T-t)| }\leq |x| \leq \epsilon_0}   |\nabla \baru|^2(1+\baru)^{r-2}\\
 \leq &C\int_{ \frac{K_0}{4} \sqrt{(T-t) |\log(T-t)|}\leq |x| \leq \epsilon_0} \theta^{-\frac{2}{p-1}}(t(x))  [|x|^2]^{-\frac{2}{p-1} - 1} |\log |x||^{\frac{2}{p-1} }dx\\
&+C\int_{ \frac{K_0}{4} \sqrt{(T-t) |\log(T-t)|}\leq |x| \leq \epsilon_0}\theta^{-\frac{N}{2}}(t(x))\left[ |x|^2 \right]^{-\frac{N}{2}-1} |\log |x||^{\frac{N}{2}} dx\\
\leq  &C \theta^{-\frac{N}{2}}(t)(T-t)^{-1}|\log (T-t)|^{\frac{N}{2}-\frac{p+1}{2}}\\
\leq  &C \theta^{-\frac{N}{2}}(t)(T-t)^{-1}|\log (T-t)|^{\frac{N}{2}-\frac{3}{2}}.
\end{split}
\end{equation}
The case where $r<2$ follows from similar computations and the fact that $(1+u)^{r-2}\leq u^{r-2}$. From \eqref{P3 estimation}, we have
\begin{equation}\label{integration nabla baru 1+baru in P3}
\int_{\epsilon_0\leq |x|, x\in \Omega}  |\nabla u|^2(1+\baru)^{r-2}\leq C.
\end{equation}
Combining \eqref{decomposition nabla u 1+u r-2} with \eqref{integration nabla baru 1+baru in P1}, \eqref{integration nabla baru 1+baru in P2 r>2} and \eqref{integration nabla baru 1+baru in P3}, we obtain
\begin{equation}
\begin{split}
\int_{\Omega}|\nabla \baru|^2(1+\baru)^{r-2}\leq C \theta^{-\frac{N}{2}}(t)(T-t)^{-1}|\log(T-t)|^{\frac{N}{2}-1}.
\end{split}
\end{equation}
Together with \eqref{decomp 1+baru on regions}, \eqref{integral 1+baru result}, we obtain

\begin{equation}\label{eq 1+u}
\begin{split}
\partial_t[\|1+\baru\|_{L^r}^r]=&r\theta^{-\frac{N}{2}}(t)(T - t)^{-1} |\log(T - t)|^{\frac{N}{2}}\left( 
\int_{0}^{+\infty}  \varphi_b\left(\xi\right)^{p-1+r} \xi^{N-1} d\xi
\right)\\
&+ O\left(\theta^{-\frac{N}{2}}(t)(T-t)^{-1}|\log(T-t)|^{\frac{N}{2}-\frac{1}{2}} \right).
\end{split}
\end{equation}
With definition \eqref{def theta}, we have
\begin{equation}\label{theta' proof}
\begin{split}
\theta'(t)=&-rq\gamma\theta^{1+\frac{1}{q}-\frac{N}{2}}(t)(T - t)^{-1} |\log(T - t)|^{\frac{N}{2}}\left( 
\int_{0}^{+\infty}  \varphi_b\left(\xi\right)^{p-1+r} \xi^{N-1} d\xi
\right)\\
&+ O\left(\theta^{1+\frac{1}{q}-\frac{N}{2}}(t)(T-t)^{-1}|\log(T-t)|^{\frac{N}{2}-\frac{1}{2}} \right),
\end{split}
\end{equation}
and after integration, we obtain
\begin{equation}
\begin{split}
\theta^{-(\frac{1}{q}-\frac{N}{2})}(t)=&r\gamma\frac{1-q\frac{N}{2}}{1+\frac{N}{2}} |\log(T - t)|^{\frac{N}{2}+1}\left( 
\int_{0}^{+\infty}  \varphi_b\left(\xi\right)^{p-1+r} \xi^{N-1} d\xi
\right)\\
&+ O\left(|\log(T-t)|^{\frac{N}{2}+\frac{1}{2}} \right).
\end{split}
\end{equation}
Using \eqref{critical condition}, Lemma \ref{lemma int phi} and \eqref{def theta infty and beta}, we obtain

\begin{equation}\label{expression theta formal}
\begin{split}
\theta(t)=& |\log(T - t)|^{-\frac{q(\frac{N}{2}+1)}{1-q\frac{N}{2}}}\left(r\gamma\frac{1-q\frac{N}{2}}{1+\frac{N}{2}} 
\int_{0}^{+\infty}  \varphi_b\left(\xi\right)^{p-1+r} \xi^{N-1} d\xi
\right)^{-\frac{q}{1-q\frac{N}{2}}}\\
&+ O\left(|\log(T-t)|^{-\frac{q(\frac{N}{2}+1)}{1-q\frac{N}{2}}-\frac{1}{2}} \right)\\
&=\theta_\infty|\log(T - t)|^{-\beta}+ O\left(|\log(T-t)|^{-\beta-\frac{1}{2}} \right),\\
\end{split}
\end{equation}
we obtain \eqref{def theta infty and beta}. Note that \eqref{critical condition} ensures that $\beta>0$ as it is required in \eqref{beta > 0}.

\section{Formulation of the problem}\label{Formulation of the problem}
In this section, we resume the proof presented earlier but with a rigorous approach. It should be noted that the rigorous approach diverges in certain aspects from the formal one outlined in the previous section. Our goal is to construct a solution $u$ of \eqref{eq u} such that 
$$u(x,t)\sim \theta^{-\frac{1}{p-1}}(t)(T-t)^{-\frac{1}{p-1}}\varphi(x\sqrt{T-t},-\log (T-t)) \text{ when } t\to T,$$
and where $\varphi$ is given by \eqref{def phi}.
\subsection{Similarity variables}
We introduce the following cut-off functions:

\begin{align}\label{def chi0}
&\chi_0 \in C_0^{\infty}([0,+\infty)),  \operatorname{supp}\left(\chi_0\right) \subset[0,2],\\
&0 \leq \chi_0(x) \leq 1,\quad \text{for all }  x \in [0,+\infty), \chi_0(x)=1, \text{for all }  x \in[0,1],
\end{align}
and 
\begin{equation}\label{def chi1}
\chi_1(x,t) = \chi_0 \left( \frac{|x|}{K_0 \sqrt{T - t} |\log(T - t)|^{\frac{p+1}{4}}} \right), \text{for all } (x,t)\in \Omega\times [0,T).
\end{equation}  
We will first localize equation \eqref{eq u}, by introducing $U$ as follows  

\begin{equation}\label{def U cut-off}
U(x,t) = \theta^{\frac{1}{p-1}}(t) \chi_1(x,t) \baru(x,t),
\end{equation}  
where $\theta$ is introduced in \eqref{def theta}. Note that our new definition of $U$ has a cut-off factor, which does not appear in our first definition of $U$ in \eqref{def U}. With \eqref{eq u}, $U$ solves the following equation:

\begin{equation}\label{eq U rigorous}
\begin{cases}
    &\partial_t U = \Delta U - 2\frac{|\nabla U|^2}{U+\theta ^{{\frac{1}{p-1}}}}   +\left(U+\theta^{{\frac{1}{p-1}}}\right)^{p}+\frac{1}{p-1}\frac{\theta'}{\theta}U+ F(\baru,U),\\
&U(x,t)=0, x\in \partial \Omega, t>0,\\
&U(x,0)=U_0(x), x\in \overline \Omega.
\end{cases}
\end{equation}
where
\begin{align}\label{def tilde F}
 F(\baru,U)=&\theta^\frac{1}{p-1}\partial_t \chi_1\baru-\theta^\frac{1}{p-1}\Delta \chi_1\baru-2\theta^\frac{1}{p-1}\nabla\chi_1\nabla \baru+2 \frac{|\nabla U|^2}{U+\theta^\frac{1}{p-1}}\\
&-2\chi_1\frac{|\nabla(\theta^\frac{1}{p-1}\baru)|^2}{\theta^\frac{1}{p-1}\baru+\theta^\frac{1}{p-1}}-(U+\theta^\frac{1}{p-1})^p+\chi_1(\theta^\frac{1}{p-1}\baru+\theta^\frac{1}{p-1})^p.
\end{align}
Moreover, we introduce $W$ from $U$ as in \eqref{similarity variables}, we get from \eqref{eq U rigorous} the following equation satisfied by $W$.
\begin{equation}\label{eq W rigorous}
\begin{cases}
&\partial_s W=\Delta  W-\frac{1}{2} y \cdot \nabla W-\frac{1}{p-1} W-2 \frac{|\nabla W|^2}{W+(\bar\theta e^{-s})^\frac{1}{p-1}}+\left(W+(\bar\theta e^{-s})^\frac{1}{p-1}\right)^{p}+\frac{1}{p-1}\frac{\bar\theta^{\prime}(s)}{\bar\theta(s)} W+\bar F,\\
&W(x,t)=0, x\in \partial \Omega_s, s>-\log T,\\
&W(y,-\log T)=W_0(y), y\in \overline \Omega_s,
\end{cases}
\end{equation}
where

\begin{align}\label{def F}
    &\bar F(y,s)=e^{-\frac{p}{p-1}s}F(\baru,U)(y\sqrt{T-t},T-e^{-s}),
\end{align}
and $\bar \theta$ and $\overline \Omega_s$ are given in \eqref{def theta bar}. We consider the linearization around the profile $\varphi$ given by  

\begin{equation}\label{def q}
q := W - \varphi.
\end{equation}  
where $\varphi$ is given in \eqref{def phi}. Hence, $q$ solves  

\begin{equation}\label{eq q equality}
\partial_s q=(\mathcal{L}+V) q+J(q)+B(q)+N(q)+R+\bar F,
\end{equation}  

where  

\begin{equation}\label{eq q}
\begin{aligned}
& \mathcal{L}=\nabla-\frac{1}{2} y \cdot \nabla+I d, \\
& V(y, s)=p(\bar\varphi^{p-1}(y, s)-\frac{1}{p-1}), \\
& J(q, \theta)=-2 \frac{|\nabla q+\nabla \bar \varphi|^2}{q+\bar\varphi}+2 \frac{|\nabla \bar\varphi|^2}{\bar\varphi}, \\
& B(q,\theta)=\left(q+\bar\varphi\right)^{p}-\bar\varphi^{p}-p \bar\varphi^{p-1} q, \\
& R=-\partial_s \varphi+\Delta \varphi-\frac{1}{2} y \cdot \nabla \varphi-\frac{\varphi}{p-1}+\bar\varphi^{p}-2 \frac{|\nabla \varphi|^2}{\bar\varphi}, \\
& N(q,\theta)=\frac{1}{p-1}\frac{\bar\theta^{\prime}(s)}{\bar\theta(s)}(q+\varphi),\\
&\bar \varphi (y,s)=\varphi(y,s) +(\bar\theta(s)\  e^{-s})^{\frac{1}{p-1}},
\end{aligned}
\end{equation}
and $\bar F$ is given in \eqref{def F}. The potential $V$ has two impotant properties:
\begin{itemize}
    \item The potential $V(., s) \rightarrow 0$ in $L_\rho^2\left(\mathbb{R}^n\right)$ as $s \rightarrow+\infty$.
    \item $V(y, s)$ is almost a constant on the region $|y| \geq K_0 \sqrt{s}$ : For all $\epsilon>0$, there exists $\mathcal{C}_\epsilon>0$ and $s_\epsilon$ such that
\end{itemize}
$$
\sup _{s \geq s_\epsilon, \frac{|y|}{\sqrt{s}} \geq \mathcal{C}_\epsilon}\left|V(y, s)-\left(-\frac{p}{p-1}\right)\right| \leq \epsilon.
$$

One may see that $-\frac{p}{p-1}<-1$ and from \eqref{spec L} that the largest eigenvalue of $\mathcal{L}$ is 1 . Hence, roughly speaking, we may assume that $\mathcal{L}+V$ admits a strictly negative spectrum. Thus, we can easily control our solution in the region $\left\{|y| \geq K_0 \sqrt{s}\right\}$ with $K_0$ large enough. Consequently, it is natural to decompose every $r \in L^{\infty}\left(\mathbb{R}^n\right)$ as follows:
\begin{equation}\label{def rb and re}
r(y)=r_b(y)+r_e(y) := \chi(y, s) r(y)+(1-\chi(y, s)) r(y),
\end{equation}
where 
\begin{equation}\label{def chi}
\chi(y, s)=\chi_0\left(\frac{|y|}{K_0 \sqrt{s}}\right),
\end{equation}
 and $\chi_0$ is given \eqref{def chi0}. In addition to that, we are interested in expanding $r_b$ in $L_\rho^2\left(\mathbb{R}^n\right)$ according to the basis which is created by the eigenfunctions of operator $\mathcal{L}$:

\begin{equation}\label{projection}
\begin{aligned}
r_b(y) & =r_0+r_1y+y^T\ r_2\ y-2\ Tr(r_2)+r_{-}(y), \\
& \text { or } \\
r_b(y) & =r_0+r_1 y+r_{\perp}(y),
\end{aligned}
\end{equation}
where
\begin{align*}
&r_0(s)=P_0(r_b)(y,s),\\
&r_1(s).y=P_1(r_b)(y,s),\\
&r_2(y)=\int_\Omega r_b(y,s) M(y)\varrho dy,\\
&r_{-}=\sum_{n \geq 3} P_n\left(r_b\right),\\
&r_{\perp}=\sum_{n \geq 2} P_n\left(r_b\right),\\
\end{align*}
with
$$
P_n(r_b)=\underset{|\beta|=n}{\sum}\frac{H_\beta}{\left\|H_\beta\right\|_{L_\varrho^2}}\int_{\mathbb{R}} H_\beta r_b  \varrho d y,
$$
and
$$M(y)=\left(\frac{1}{4}y_iy_j-\frac{1}{2}\delta_{ij}\right)_{1\leq i,j \leq N}.$$

The renewed statement of purpose is to construct a global solution $q$ for equation \eqref{eq q} defined on $[s_0, +\infty)$, for some $s_0>0$ and such that
\begin{equation}\label{conv q}
\|q(s)\|_{L^\infty} \rightarrow 0 \text{ as } s \rightarrow +\infty.
\end{equation}
\subsection{Localization variable}
Furthermore, we impose another constraint. In accordance with our formal approach in subsection \ref{subsec formal P2} and using reverse transformations $q\rightarrow w\rightarrow U \rightarrow u$ with \eqref{def q}, \eqref{similarity variables}, \eqref{def U cut-off}, we will further impose that for all $t \in [0, T)$, $x \in P_2(t)$ (defined in \eqref{def P}), $|\xi| \leq \alpha_0 \sqrt{|\log \varrho(x)|}$, for some $\alpha_0$ and $\tau(x,t) \in \left[0, 1 \right)$, that
\begin{equation}\label{close to V hat}
|\mathcal{V}(x, \xi, \tau) - \hat{\mathcal{V}}(x, \tau)|\leq \delta_0
\end{equation}
where $\delta_0>0$ will be chosen small, $\varrho(x)$, $\hat{\mathcal{V}}$ are given in \eqref{def rho(x), t(x)} and \eqref{def V hat} respectively,  and $\mathcal{V}$ is introduced through the definition of $\mathcal{U}$ (given in \eqref{def mathcal V}), which in turn is constructed from the original function $u$ (see \eqref{def mathcal U}). Moreover, we will impose for all $x\in P_3(t)$ (defined in \eqref{def P}), that $u$ and $\nabla u$ remain close to the initial data in a certain sens. This leads to defining an accurate shrinking set, where the solution should be trapped.

\subsection{Shrinking set $S(t)$ to the blow-up profile}

In the following, we construct a specialized set that ultimately leads to the derivation of the asymptotic behavior \eqref{conv q} and \eqref{close to V hat} and boundedness of $u$ and $\nabla u$ in $P_3$. This construction is inspired by \cite{DGKZ2022} and \cite{DKZ2021}, with necessary modifications tailored to the underlying critical regime case.

We introduce a shrinking set $S(t)$ to the blow-up profile, where we trap our construced solution.
\begin{defn}[A shrinking set $S(t)$ to the blow-up profile]\label{def shrinking set S}
Let us consider $ T$ ,$K_0$ ,$\epsilon_0$ , $\alpha_0$, $A$, $\delta_0$, $C_0$, $\eta_0 > 0 $ and $ t \in [0, T) $. Then, we introduce the following set:
$$
S^*(T, K_0, \epsilon_0, \alpha_0, A, \delta_0, \eta_0, C_0, t) \quad (S^*(t) \text{ for short}),
$$

as a subset of $ C^{2,1}(\Omega \times [0, t)) \cap C(\bar{\Omega} \times [0, t]) $, containing all functions $ u $ satisfying the following conditions:

\begin{itemize}
    \item \textbf{Estimates for $ |x|\in [0,2K_0\sqrt{(T-t)}|\log (T-t)|^\frac{p+1}{4}] $:} We have $ q(s) \in V_{K_0, A}(s) $, where $ q(s) $ is introduced in \eqref{def q}, $ s = -\log(T-t) $ and $ V_{K_0, A}(s) $ is a subset of all functions $ r $ in $ L^\infty(\mathbb{R}^N) $, satisfying the following estimates:

\begin{align}
&|r_0| \leq \frac{A^3}{s^{\frac{3}{2}}}, \quad |r_1| \leq \frac{A}{s^{2}}, \quad  |r_2| \leq \frac{A^4}{s^{\frac{3}{2}}}, \\
&|r_-(y)| \leq \frac{A^6}{s^2 }(1 + |y|^3),|(\nabla r)_\perp| \leq \frac{A^6}{s^2 }(1 + |y|^3), \forall y \in \mathbb{R}^N,\\
& \|r_{\text{e}}\|_{L^\infty(\mathbb{R})} \leq \frac{A^7}{\sqrt{s}},\\
\end{align}
    where the definitions of $ r_i$, for $i=0,1,2$, $r_-, (\nabla r)_\perp $ are given in \eqref{projection}.
    
    \item \textbf{Estimates in $ P_2(t) $:} For all $ x \in P_2(t) $, $ \tau(x, t) = \frac{t - t(x)}{\varrho(x)} $ and $ |\xi| \leq \alpha_0 \sqrt{|\log \varrho(x)|} $, we have the following:
    \begin{equation}\label{V-hat V}
    | \mathcal{V}(x_0, \xi, \tau) - \hat{\mathcal{V}}(x,\tau) | \leq \delta_0,
    \end{equation}
    $$
    |\nabla_\xi \mathcal{U}(x, \xi, \tau(x,t))| \leq \frac{C_0}{\sqrt{|\log \varrho(x)|}},
    $$
    where $ \mathcal{U},\mathcal{V} , \hat{\mathcal{V}} $ and $ \varrho(x) $ are given in \eqref{def mathcal U}, \eqref{def mathcal V},  \eqref{def V hat} and \eqref{def rho(x), t(x)}, respectively.

    \item \textbf{Estimates in} $ P_3(t) $: For all $ x \in P_3(t)$, we have
$$
|u(x,t) - u(x,0)| \quad \leq \quad \eta_0,
$$
$$
|\nabla u(x,t) - \nabla S(t) u(x,0)| \quad \leq \quad \eta_0.
$$
where $S(t)$ is the heat semi-group generated by $\Delta$ with Dirichlet boundary conditions.
\end{itemize}
\end{defn}
We have the following direct consequence from the shrinking set.
\begin{lemma}[Growth estimates]\label{Growth estimates}
We consider $K_0 \geq 1$, $A \geq 1$, $s_0\geq 0$, and $q$ such that $q(s)\in V_A(s)$, for all $s\geq s_0$. Then, there exists $\bar s_0(A, K_0)\geq s_0$ such that for all $s \geq \bar s_0$, the following hold:
\begin{align}
&|q(y, s)| \leq \frac{C(K_0)A^7}{\sqrt{s}},\label{estimate q 1}\\
&|q(y, s)| \leq \frac{C(K_0)A^7(1+|y|^3)}{s^{\frac{3}{2}}},\label{estimate q 2}
\end{align}
In particular, we have the following estimates:
\begin{align}
&|q_b(y,s)|\leq \frac{C(K_0)A^6(1 + |y|^3)}{s^{\frac{3}{2}}},\label{estimate q b}\\
&|(\nabla q)_b(y, s)| \leq \frac{C(K_0)A^6(1+|y|^2)}{s^\frac{3}{2}},\label{estimate grad q 2}\\
&|(\nabla q)_b(y, s)| \leq \frac{C(K_0)A^6}{\sqrt{s}},\label{estimate grad q 3}
\end{align}
where $q_b$ and $(\nabla q)_b$ are given in \eqref{projection}.
\end{lemma}

\begin{proof} The estimates on $q$ immediately follows from the definition of $V_A(s)$ in Definition \ref{def shrinking set S}. Now for $\nabla q$, since $q_2$ is symmetric, then we have from \eqref{projection} that 
\begin{equation}\label{decomp nabla q b}
(\nabla q)_b=q_1+2q_2 h_1+(\nabla q)_\perp.
\end{equation}
Again, with Definition \ref{def shrinking set S} the proof follows.
\end{proof}
In particular, it is enough to construct a solution $u$ for equation \eqref{eq u} so that for all $t\in [0,T], u(t)\in S^*(t)$, for some $T>0$ and adequate constants.
\section{The existence proof without technical details}\label{Section existence without technic}
Here, we establish the existence of an initial datum $u_0 \in S(0)$ such that the corresponding solution of equation \eqref{eq u} remains confined within the shrinking set $S(t)$ for every $t\in [0,T)$, for some $T>0$.
\subsection{Preparation of the initial data}

We will first introduce the initial condition $u_0 \in S(0)$ corresponding to equation \eqref{eq u}. For this matter, we define $H^*$ as a suitable modification of the final asymptotic profile in the intermediate region such that
\begin{equation}\label{def H*}
H^*(x) =
\begin{cases} 
 \left[ b\frac{|x|^2}{2|\log |x||}  \frac{\theta_\infty}{(2|\log|x||)^\beta}\right]^{-\frac{1}{p-1}}, & \forall |x| \leq \min \left(\frac{1}{4} d(0, \partial \Omega), \frac{1}{2} \right), \\
0, & \forall |x| \geq \frac{1}{2} d(0, \partial \Omega),
\end{cases}
\end{equation}
and $H^*$ is decreasing between the two regions in \eqref{def H*} and where $b,\beta$ and $\theta_\infty$ are defined as in \eqref{expression b} and \eqref{def theta infty and beta}. Consider $T>0$ and $(d_0, d_1) \in \mathbb{R}^{1+N}$, we define initial data for equation \eqref{eq u} as follows:
\begin{equation}\label{initial data}
\begin{aligned}
u_{d_0, d_1,\theta_0}(x)=&\ \theta_0^{-\frac{1}{p-1}}T^{-\frac{1}{p-1}}  \Bigg[ \varphi \left( \frac{x}{\sqrt{T}}, -\log T \right)  \\
&+\left( d_0 \frac{A^3}{|\log T|^{\frac{3}{2}}} + d_1\frac{A}{|\log T|^2}  \cdot \frac{x}{\sqrt{T}} \right) \chi_0 \left( \frac{|z_0|}{\frac{K_0}{32}} \right)\Bigg] \chi_1(x,0)\\
&+ H^*(x) (1 - \chi_1(x,0)) ,
\end{aligned}
\end{equation}
for some $\theta_0\in \R^+$, which will be fixed later and where $z_0 = \frac{x}{\sqrt{T|\log T|} }$, and $\varphi, \chi_0, \chi_1$ and $H^*$ are defined in \eqref{def phi}, \eqref{def chi0}, \eqref{def chi1}, and \eqref{def H*} respectively, where $\theta_0$ will be chosen so that
\begin{align}\label{initial data system}
    &\theta_0=\left( 1 + \gamma
    \displaystyle \int_{\Omega} (1+ \baru_{d_0, d_1,\theta_0}(x))^r dx \right)^{-q}.
\end{align}
The result is given by the following Lemma:
\begin{lemma}[existence of $\theta_0$]\label{existence theta0}
For all $K_0,A>0$ and $(d_0,d_1)\in [-2,2]^{N+1}$, there exists $\epsilon_1(K_0,A)>0$, for all $\epsilon_0\leq\epsilon_1$, there exists $T_1(K_0,A)>0$, for all $T\in (0,T_1]$, there exists $\theta_0>0$ solution to \eqref{initial data} and \eqref{initial data system} and
    $$ \theta_0\sim\frac{\theta_\infty}{|\log T|^\beta} \text{ as } T\rightarrow 0,$$
    where $\theta_\infty$ and $\beta$ are given in \eqref{def theta infty and beta}.
\end{lemma}
Before diving into the proof, let us introduce $P_1$ and $P_2$ such that
\begin{equation}\label{def P1, P2}
\begin{split}
P_1(x)=&T^{-\frac{1}{p-1}}  \Bigg[ \varphi \left( \frac{x}{\sqrt{T}}, -\log T \right) \\
&+ \left( d_0 \frac{A^3}{|\log T|^{\frac{3}{2}}} + d_1\frac{A}{|\log T|^2}  \cdot \frac{x}{\sqrt{T}} \right) \chi_0 \left( \frac{|z_0|}{\frac{K_0}{32}} \right)\Bigg] \chi_1(x,0),\\
P_2(x)=&H^*(x) (1 - \chi_1(x,0)),
\end{split}
\end{equation}
for all $x\in \Omega$.
\begin{proof}
We fix $A>0$, $(d_0,d_1)\in [-2,2]^{N+1}$ and define the following functional:
\begin{equation}
\begin{array}{ccccc}
    \Phi &: & \R^+ &\longrightarrow  &\R^+\\
   & &\xi &\mapsto &\left( 1 + \gamma
    \displaystyle \int_{\Omega} (1+ \xi^{-\frac{1}{p-1}} P_1(x)+P_2(x))^r dx \right)^{-q},
\end{array}
\end{equation}
where $P_1$ and $P_2$ are given in \eqref{def P1, P2}. Notice that $P_1,P_2>0$ for $T$ small enough, which implies that $\xi\in (0,1]$. Our aim is to prove the following:
\begin{equation}\label{ineq Phi bound}
\begin{aligned}
\Phi\left((1-\bar \epsilon)\frac{\theta_\infty}{|\log T|^\beta}\right)\geq(1-\bar \epsilon)\frac{\theta_\infty}{|\log T|^\beta},\\
\Phi\left((1+\bar \epsilon)\frac{\theta_\infty}{|\log T|^\beta}\right)\leq (1+\bar \epsilon)\frac{\theta_\infty}{|\log T|^\beta},
\end{aligned}
\end{equation}
for some $\bar \epsilon=\bar \epsilon(T)$ fixed later, with $\bar \epsilon\to 0$ as $T\to 0$. Indeed, since $\Phi$ is continuous, then the proof follows from \eqref{ineq Phi bound}. For this matter, we denote by 
$$\theta_\pm=(1\pm\bar \epsilon)\frac{\theta_\infty}{|\log T|^\beta},$$
and decompose the following:
\begin{equation}\label{I=sum I_1,2,3,4,5}
\begin{aligned}
I &:= \int_{\Omega} (1+u_{d_0,d_1,\theta_\pm}(x))^r dx \\
&= \Bigg(\int_{|x|\leq \frac{K_0}{16} \sqrt{T|\log T|}}+\int_{|x|\in[\frac{K_0}{16} \sqrt{T|\log T|}, K_0 \sqrt{T}\,|\log T|^{\frac{p+1}{4}}]}  +\\
&\int_{|x|\in [K_0 \sqrt{T}\,|\log T|^{\frac{p+1}{4}}, 2K_0 \sqrt{T}\,|\log T|^{\frac{p+1}{4}}]}+\int_{|x|\leq K_0 \sqrt{T}\,|\log T|^{\frac{p+1}{4}}}+\int_{|x|\geq \epsilon_0, x \in \Omega}\Bigg) (1+u_{d_0,d_1,\theta_\pm}(x))^r dx \\
&=: I_1 + I_2+I_3+I_4+I_5.
\end{aligned}
\end{equation}
From \eqref{initial data}, we have for all $|x|\leq \frac{K_0}{16} \sqrt{T|\log T|}$ that
\begin{equation}
\begin{aligned}
u_{d_0,d_1,\theta_\pm}(x) = T^{-\frac{1}{p-1}}  \theta_\pm^{-\frac{1}{p-1}} 
\Bigg[ &\varphi\left( \frac{x}{\sqrt{T}}, - \log T  \right) 
\\
&+ \left(d_0 \frac{A^3}{|\log T|^{\frac{3}{2}}} + d_1\frac{A}{|\log T|^2}\frac{x}{\sqrt{T}} \right)\chi_0\!\left( \frac{|z_0|}{K_0/32}\right)\Bigg]\chi_1(x,0).
\end{aligned}
\end{equation}
Consequently,  
\begin{align}
|u_{d_0,d_1,\theta_\pm}(x)| \leq C(K_0) T^{-\frac{1}{p-1}} |\log T|^{\frac{\beta}{p-1}},
\end{align}
which yields, with \eqref{critical condition}, to
\begin{align}\label{bound I1}
|I_1|&\leq C(K_0)\int_{|x|\leq \frac{K_0}{16}\sqrt{T|\log T|}} 1+u_{d_0,d_1,\theta_\pm}(x)^r\ dx \\
&
\leq  C(K_0)|\log T|^{\frac{N}{2}(\beta+1)}.
\end{align}
For $|x|\in \big[ \frac{K_0}{16}\sqrt{T|\log T|}, K_0\sqrt{T}|\log T|^{\frac{p+1}{4}}\big]$, we have from \eqref{initial data} the following 
\begin{align}
u_{d_0,d_1,\theta_\pm}(x) &=T^{-\frac{1}{p-1}} \theta_\pm^{-\frac{1}{p-1}}
\left[ \varphi_b\!\left( \frac{|x|}{\sqrt{T|\log T|}}\right) + \frac{a}{|\log T|}\right].
\end{align}
Therefore with \eqref{critical condition} and \eqref{critical condition}, we have
\begin{equation}\label{bound I2}
\begin{aligned}
|I_2|
\leq C(K_0)&\Bigg( T^\frac{N}{2}|\log T|^{N\frac{p+1}{4}}\\
&\ +|\log T|^{\frac{N}{2}(\beta+1)}
\int_{\frac{K_0}{16}}^{K_0|\log T|^{\frac{p-1}{4}}}
\left[ \varphi_b^r(\xi) + \frac{a}{|\log T|^r} \right]\xi^{N-1} d\xi\Bigg) \\
\leq C(K_0) &|\log T|^{\frac{N}{2}(\beta+1)}.
\end{aligned}
\end{equation}
 For all $|x|\in \big[ K_0\sqrt{T}|\log T|^\frac{p+1}{4}, 2K_0\sqrt{T}|\log T|^{\frac{p+1}{4}} \big]$, it is easy to see that  
\begin{align}
 T^{-\frac{1}{p-1}}|\log T|^{-\frac{\beta}{p-1}} \varphi_b\!\left( \frac{|x|}{\sqrt{T|\log T|}}\right) \sim  H^*(x) \text{ when } T\rightarrow 0,
\end{align}
and thus with \eqref{initial data} and \eqref{critical condition},  we obtain  
\begin{equation}\label{bound I3}
\begin{aligned}
|I_3|\leq &C(K_0)\Big(T^{\frac{N}{2}}|\log T|^{N\frac{p+1}{4}}+ \int_{|x|\in [K_0 \sqrt{T}\,|\log T|^{\frac{p+1}{4}}, 2K_0 \sqrt{T}\,|\log T|^{\frac{p+1}{4}}]}
\left[ |x|^2 \right]^{-\frac{N}{2}} |\log|x||^{\frac{N}{2}(1+\beta)} dx\Big) \\
\leq &C(K_0)|\log|\log T||^{1+\frac{N}{2}(1+\beta)}.
\end{aligned}
\end{equation}
Now, for $I_4$, we have from \eqref{initial data} and  \eqref{critical condition} the following
\begin{equation}\label{bound I4}
\begin{aligned}
I_4&= \left(1+O\left(\frac{1}{H^*(\epsilon_0)} \right)\right)\int_{2K_0\sqrt{T}|\log T|^{\frac{p+1}{4}} \leq |x| \leq\epsilon_0}(H^*)^r\,dx\\
&=  \left(1+O\left(\frac{1}{H^*(\epsilon_0)} \right)\right)\int_{2K_0\sqrt{T}|\log T|^{\frac{p+1}{4}} \leq |x| \leq\epsilon_0}
(\theta_\infty b)^{-\frac{N}{2}}
\,| x|^{-N} | 2\log| x||^{\frac{N}{2}(1+\beta)}dx\\
&= \frac{\theta_\infty^{-\frac{N}{2}}
b^{-\frac{N}{2}}}{2\,\big(1 + \frac{N}{2}(1+\beta)\big)}|\log T|^{1 + \frac{N}{2}(1+\beta)}
\left(1 + O\left(\frac{| \log| \log T||}{| \log T|}\right)+O\left(\frac{1}{H^*(\epsilon_0)}\right)\right).
\end{aligned}
\end{equation}
For $I_5$, we have from \eqref{initial data} that  
\begin{align}\label{bound I5}
|I_5| \leq C(\epsilon_0),
\end{align}
where $C(\epsilon_0)\to +\infty$, as $ \epsilon_0 \to 0$. Combining \eqref{I=sum I_1,2,3,4,5}, \eqref{bound I1}, \eqref{bound I2}, \eqref{bound I3}, \eqref{bound I4} and \eqref{bound I5}, we obtain 

\begin{align}\label{bound I}
    I=&\frac{\theta_\infty^{-\frac{N}{2}}
b^{-\frac{N}{2}}}{2\,\big(1 + \frac{N}{2}(1+\beta)\big)}|\log T|^{1 + \frac{N}{2}(1+\beta)}
\left(1 + O\left(\frac{| \log| \log T||}{| \log T|}\right)+O\left(\frac{1}{H^*(\epsilon_0)}\right)\right)\\
&+O(C(\epsilon_0)).
\end{align}
Taking $\epsilon_0$ small enough, then $T$ small enough, we obtain
\begin{align}
    \Phi(\theta_-)&= \theta_\infty
| \log T|^{-q(1 + \frac{N}{2}(1+\beta))}
\left(1 + O\left(\frac{| \log| \log T||}{| \log T|}\right)+O\left(\frac{1}{H^*(\epsilon_0)}\right)\right)\\
&\geq (1-\bar \epsilon) \theta_\infty
| \log T|^{\beta}=\theta_-,
\end{align}
for some $\bar \epsilon=\bar \epsilon(\epsilon_0,T)>0$. Similar computation are required to prove the opposite inequality for $\theta_+$, which concludes the proof of Lemma \ref{existence theta0}.
\end{proof}

Thanks to Lemma \ref{existence theta0} and \eqref{initial data}, we have that 
\begin{equation}\label{theta(0)=theta_0}
    \theta(0)=\theta_0.
\end{equation}
Accordingly, we apply on \eqref{initial data} the transformation \eqref{def U cut-off} to derive $U_{d_0,d_1, \theta_0}(0)$. Then \eqref{similarity variables}, to obtain $W_{d_0, d_1,\theta_0}$, and \eqref{def q} to finally get

\begin{align}\label{q initial without expression}
q_{d_0, d_1,\theta_0}(y_0, s_0) &= W_{d_0, d_1,\theta_0}(y_0, s_0)-\varphi \left( y_0, s_0 \right),
\end{align}
where $y_0=\frac{x}{\sqrt{T}}$  and $s_0=-\log T$. We obtain the following result proving that the initial data of the form \eqref{initial data} lies in the shrinking set $S(0)$.

\begin{prp}[Constructing initial data]\label{Initial data in S(0)}
There exists $K_2>0$, for all $K_0\geq K_2$ and $\delta_2 > 0$, there exist $C_2(K_0) > 0$ and $\alpha_2(K_0, \delta_2) > 0$ such that for all $\alpha_0 \in (0, \alpha_2)$, there exists $\epsilon_2(K_0, A,\delta_0, \alpha_0) > 0$ such that for all $\epsilon_0 \in (0, \epsilon_2]$, for all $A\geq 1$, we can find $T_2(K_0, \delta_0, \epsilon_0, A, C_0) > 0$ small enough, such that for all $T \leq T_2$, there exists $\mathcal{D}_A \subset [-2, 2] \times [-2, 2]^N$, such that for all $\eta_0>0$, the following hold:
\begin{itemize}
    \item[(I)]\label{initial data in S I} For all $(d_0, d_1)\in \mathcal{D}_A$, the initial data $u_{d_0,d_1,\theta_0}\in S^*(T,K_0,\epsilon_0,\alpha_0,A,\delta_0,\eta_0,C_0,0)$ defined in \eqref{initial data}, where $\theta_0$ is given by Lemma \ref{existence theta0}. More precisely, $u_{d_0,d_1,\theta_0}$ satisfy following estimates:

\begin{itemize}[label=\textbullet]
\item Estimates in similarity variables: $q_{d_0, d_1}(s_0)$ defined in \eqref{q initial without expression} satisfy
$$
|q_0(s_0)| \leq \frac{A^3}{s_0^\frac{3}{2}}, \quad |q_{1}(s_0)| \leq \frac{A}{s_0^2}, \quad |q_{2}(s_0)| \leq \frac{1}{s_0^2}, 
$$
$$
|q_-(s_0)| \leq \frac{1+|y|^3 }{s_0^2} , \quad |\nabla q_-( s_0)| \leq \frac{1+|y|^3}{s_0^2} , \forall y \in \mathbb{R}^N.
$$

$$
\|q_e(s_0)\|_{L^\infty} \leq \frac{C(K_0)}{\sqrt{s_0}},
$$
where $s_0 = |\log T|$.
\item Estimates in $P_2(0)$: For all $|x| \in \left[ \frac{K_0}{4} \sqrt{T |\log T|},\epsilon_0\right)$, $\tau_0(x) = -\frac{t(x)}{T- t(x)}$ and $|\xi| \leq  2\sqrt{\alpha_0 |\log\varrho(x)|}$, we have
$$
|\mathcal{V}(x,\xi, \tau_0(x)) - \hat{\mathcal{V}}(\xi, \tau_0(x))| \leq \delta_0 \text{ and } |\nabla_\xi \mathcal{U}(\xi, \tau_0(x))| \leq \frac{C_0}{\sqrt{|\log\varrho(x)|}},
$$
\end{itemize}

where $\varrho(x)$ $\mathcal{V}$ and $\hat{\mathcal{V}}$ are defined as in \eqref{def rho(x), t(x)}, \eqref{def mathcal V} and \eqref{def V hat} respectively. 

\item[(II)] Moreover, the following mapping
$$
\Gamma: \mathbb{R}^{1+N} \to \mathbb{R}^{1+N}
$$
$$
(d_0, d_1) \mapsto (q_0, q_1)(s_0),
$$
is affine, one to one from $\mathcal{D}_A$ to $\hat{\mathcal{D}}_A(s_0)$, where $\hat{\mathcal{D}}_A(s)$ defined by
\begin{equation}\label{hat D A}
\hat{\mathcal{D}}_A(s) = \left[ -\frac{A^3}{s^\frac{3}{2}}, \frac{A^3}{s^\frac{3}{2}} \right] \times \left[ -\frac{A}{s^2}, \frac{A}{s^2} \right]^N. 
\end{equation}

In addition we have
$$
\Gamma|_{\partial\mathcal{D}_A} \subset \partial\hat{\mathcal{D}}_A(s_0),
$$
and
$$
\deg \left( \Gamma|_{\mathcal{D}_A} \right)=1.
$$
\end{itemize}
\end{prp}
\begin{proof}
See Appendix \ref{proof of initial data in S(0)}.
\end{proof}

\subsection{Dynamics of $\theta$}

Here, we provide a rigorous proof of \eqref{asymp bar theta in t}, presented formally in Section \ref{Section formal proof} and which we recall here
$$\theta(t)=\frac{\theta_\infty}{|\log (T-t)|^\beta}\left(1+O\left(\frac{1}{\sqrt{|\log (T-t)}|}\right)\right).$$
We will proceed in three steps:
\begin{itemize}
\item First, we derive estimates for $u$ under certain assumptions on $\theta$.
\item Next, we obtain asymptotic profiles for $\theta$ and $\theta'$, using the same assumptions as in the first step.
\item Finally, we recover the previous estimates without relying on the initial assumptions on $\theta$.
\end{itemize}

\textbf{Step 1: Estimates on $u$}\\
Here, we derive estimates for $u$ after it becomes confined within the shrinking set $S(t)$.
\begin{lemma}[Estimates on $u$ in $S^*$]\label{lemma consequence shrinking set}
For all $K_0>0, \delta_0\in \left(0,\frac{1}{2}\left[ b \frac{K_0^2}{16}\right]^{\frac{1}{p-1}}\right]$ there exists $\epsilon_3(K_0,\delta_0)>0$, for all $\epsilon\in (0,\epsilon_3)$, for all $A\geq 1$, $\alpha_0,C_0>0$, $\eta_0\in (0,\frac{1}{2}]$, there exists $T_3(K_0,A,\delta_0,\alpha_0,\eta_0)>0$, for all $T\in (0,T_3]$, such that if $\baru \in S(K_0, \epsilon_0, \alpha_0, A, \delta_0, C_0, \eta_0, t)$ for all $t \in [0, t_3]$, for some $t_3\leq T$ and where $u(0)=u_{d_0,d_1,\theta_0}$ defined in \eqref{initial data} for some $(d_0, d_1) \in \mathbb{R}^{1+N}$ and $\theta_0$ given by Lemma \ref{existence theta0}, and if $t_3>0$, we further assume that $\theta'(t)< 0$ and $\theta(t)>0$ for all $t\in (0,t_3]$, then the following hold:
\begin{enumerate}[label=(\roman*)]
\item \label{estimate until in P1} For all $|x| \leq K_0 \sqrt{(T - t)} |\log(T - t)|^{\frac{p+1}{4}}$, we have
\begin{equation}\label{estimate u until t1}
\begin{aligned}
\Bigg| u(x,t) - \theta^{-\frac{1}{p-1}}(t)(T - t)^{-\frac{1}{p-1}} \varphi_b\Bigg( &\frac{|x|}{\sqrt{(T - t)|\log(T - t)|} } \Bigg) \Bigg| \\
&\leq \frac{C(K_0) A^7 (T - t)^{-\frac{1}{p-1}} \theta^{-\frac{1}{p-1}}(t)}{1 + \sqrt{|\log(T - t)|}},
\end{aligned}
\end{equation}
where $\varphi_b$ defined in \eqref{expression phi profile}, together with the gradient estimate, for all \\$|x| \leq K_0 \sqrt{(T - t) |\log(T - t)|}$, we have
{\mathtoolsset{showonlyrefs=false}
\begin{equation}\label{estimate grad u until t1}
\left|\nabla u(x,t)\right|\leq \frac{C(K_0)A^6\theta^{-\frac{1}{p-1}}(t)(T-t)^{-\frac{1}{p-1}-\frac{1}{2}}}{\sqrt{\log(T-t)}}.
\end{equation}
}

\item \label{estimate until in P2}For all $|x| \in \left[ \frac{K_0}{4} \sqrt{(T - t) |\log(T - t)|}, \epsilon_0 \right]$, we have
\begin{align}\label{bounds u in P2}
\frac{1}{C(K_0,C_0)} &\left[ |x|^2 \right]^{-\frac{1}{p-1}} |\log |x||^{\frac{1 }{p - 1}} \theta^{-\frac{1}{p-1}}(t(x))\\
&\leq u(x,t) \leq C(K_0,C_0) \left[ |x|^2 \right]^{-\frac{1}{p-1}} |\log |x||^{\frac{1 }{p - 1}}\theta^{-\frac{1}{p-1}}(t(x)),
\end{align}
and
{\mathtoolsset{showonlyrefs=false}
\begin{align}\label{bounds grad u in P2}
|\nabla_x u(x,t)| \leq C(K_0,C_0)(|x|^2)^{-\frac{1}{p-1}-\frac{1}{2}} |\log |x||^{\frac{1}{p-1}}\theta^{-\frac{1}{p-1}}(t(x)).
\end{align}
}
\item \label{estimate until in P3}For all $|x| \geq \epsilon_0$, we have
{\mathtoolsset{showonlyrefs=false}
\begin{equation}\label{item P3}
-\frac{1}{2}\leq u(x,t) \leq C(\epsilon_0, \eta_0),\ |\nabla_x u(x,t)| \leq C(\epsilon_0,\eta_0).
\end{equation}
}
\end{enumerate}
Moreover,
\begin{equation}\label{upper bound theta}
    \theta(t) \leq C|\log T|^{-q\beta\frac{N}{2}}|\log(T-t)|^{-q\frac{N}{2}}.
\end{equation}
\end{lemma}

\begin{proof}
\textbf{Proof of \ref{estimate until in P1}:} It follows directly from \eqref{estimate q 1}, \eqref{estimate grad q 3} and applying reversed transformations $q\rightarrow w\rightarrow U\rightarrow u$ with \eqref{def q}, \eqref{similarity variables} and \eqref{def U cut-off} respectively.

\textbf{Proof of \ref{estimate until in P2}:}
The proof of \eqref{bounds grad u in P2} follows directly from \eqref{def mathcal U}. Now, let's focus first on obtaining \eqref{bounds u in P2} for $t=0$. We have from Lemma \ref{existence theta0} that
$$\theta_0>0,$$
for small $T$. Moreover,\eqref{bound hat mathcal V formal} is valid for $t=0$ (see later in \eqref{hat mathcal V tau0}). Therefore, for $\delta_0 \leq \frac{1}{2}\left[ b \frac{K_0^2}{16}\right]^{\frac{1}{p-1}}$ and $x$ small, we obtain \ref{estimate until in P2} at $t=0$. Moreover, with \eqref{t(x)<0}, we have
\begin{equation}\label{log rho < log T}
|\log \varrho (x)|\leq |\log T|, \text{ when } x\in P_2(0).
\end{equation}
Thus, with \eqref{def theta}, we obtain
\begin{equation}\label{theta t(x)>0}
\theta(t(x))\geq \theta_0.
\end{equation}
Combing all this with \ref{estimate until in P1} and Lemma \ref{existence theta0} and similar computation as for \eqref{theta' proof}, one can show that
$$\theta'(0)<0,$$
for small $T$. Together with the positiveness and decreasing assuption of $\theta$ for $t\in (0,t_3]$, we have that
\begin{equation}\label{theta>0, theta'<0}
\theta(t)>0 \text{ and } \theta'(t)<0, \text{ for all }t\in [0,t_3].
\end{equation}
Then, \eqref{bounds tilde theta/theta(t(x))} is valid for all $\sigma\in \left[0,\frac{t_1-t(x)}{T-t(x)}\right]$, which implies that
\begin{equation}\label{bound hat mathcal V}
    \left(b \frac{K_0^2}{16}\right)^{\frac{1}{p-1}}\leq \hat{\mathcal{V}}\leq \left( p-1 + b \frac{K_0^2}{16}\right)^{\frac{1}{p-1}}, \text{ for all }\sigma\in \left[0,\frac{t_1-t(x)}{T-t(x)}\right].
\end{equation}
Therefore, with Definition \ref{def shrinking set S}, \eqref{def mathcal V}, we have
\begin{equation}
    \frac{1}{2}\left( p-1 + b \frac{K_0^2}{16}\right)^{-\frac{1}{p-1}}-(\theta(t(x))\varrho(x))^{\frac{1}{p-1}}\leq \mathcal{U}\leq \left(b \frac{K_0^2}{16}\right)^{-\frac{1}{p-1}}+(\theta(t(x))\varrho(x))^{\frac{1}{p-1}}.
\end{equation}
for $\delta_0 \leq \frac{1}{2}\left[ b \frac{K_0^2}{16}\right]^{\frac{1}{p-1}}$. With the decreasing and positiveness assumptions on $\theta$, we have that $\theta(t(x))\to C\in [0,+\infty)$, as $x\to 0$. Thus, with \eqref{asym rho} and \eqref{def mathcal U}, we obtain \eqref{bounds u in P2} for $\epsilon_0$ small enough. Hence, the proof of \ref{estimate until in P2} is completed.

\textbf{Proof of \ref{estimate until in P3}:} It follows directly from Definition \ref{def shrinking set S}, \eqref{initial data} and Lemma \ref{lemma bound nabla (Sf-f)}, for $T$ small enough. 

\textbf{Proof of \eqref{upper bound theta}:} We have from \ref{estimate until in P1}, \ref{estimate until in P2}, \ref{estimate until in P3}, \eqref{theta>0, theta'<0}, Lemma \ref{existence theta0}  that
\begin{align}
\int_\Omega (1+u)^r &\geq \int_{|x|\leq K_0\sqrt{(T-t)|\log(T-t)|}}  (1+u)^r+\int_{|x|\geq \epsilon_0, x\in \Omega}  (1+u)^r\\
& \geq C\left(\theta^{-\frac{N}{2}}(t)|\log(T-t)|^{\frac{N}{2}}+1\right)\\
&\geq C|\log T|^{\beta\frac{N}{2}}|\log(T-t)|^{\frac{N}{2}}.
\end{align}
Then, \eqref{upper bound theta} follows from this and \eqref{def theta}. This conlcudes the proof of Lemma \ref{lemma consequence shrinking set}.
\end{proof}

\textbf{Step 2: On the dynamics of $\theta$}\\
Next, we aim to give the rigorous proof to \eqref{asymp bar theta in t} in the formal approach part.

\begin{prp}[Dynamics of $\theta$]\label{dynamics of theta}
For all $K_0>0, \delta_0\in \left(0,\frac{1}{2}\left[ b \frac{K_0^2}{16}\right]^{\frac{1}{p-1}}\right]$ there exists $\epsilon_4(K_0,\delta_0)>0$, for all $\epsilon\in (0,\epsilon_4)$, for all $A\geq 1$, $\alpha_0,C_0>0$, $\eta_0\in (0,\frac{1}{2}]$, there exists $T_4(K_0,A,\delta_0,\alpha_0,\eta_0)>0$, for all $T\in (0,T_4]$ such that for all $T \leq T_4$ the following holds: Consider $u$ the solution of equation \eqref{eq} with initial date $u_{d_0,d_1,\theta_0}$ defined in \eqref{initial data} for some $(d_0, d_1) \in \mathbb{R}^{1+N}$ and $\theta_0$ given by Lemma \ref{existence theta0} and we assume $u \in \mathcal{S}(T, K_0, \epsilon_0, \alpha_0, A, \delta_0, C_0, \eta_0, t) = S(t)$ for all $t \in [0,t_1]$ for some $t_1<T$,and if $t_1>0$, we further assume that $\theta'(t)< 0$ and $\theta(t)>0$ for all $t\in [0,t_1]$, then the following hold:

\begin{equation}\label{expression theta rigorous}
\left| \theta(t) - \theta_{\infty} \left| \log(T - t) \right|^{-\beta} \right| \leq C A^7 \left| \log(T - t) \right|^{-\beta - \frac{1}{2}},
\end{equation}
and
\begin{equation}\label{expression theta' rigorous}
\left| \theta'(t) - \theta_{\infty}(-\beta)(T - t)^{-1} \log(T - t)^{-\beta - 1} \right|
\leq  Crq\gamma A^7 (T - t)^{-1} \log(T - t)^{-\beta - \frac{3}{2}}, 
\end{equation}

where $\theta_{\infty}$ and $\beta$ defined as in \eqref{def theta infty and beta}. In particular, if we take $rq\gamma \leq A^{-5}$, then we obtain the following estimate

\begin{equation}\label{expression theta'/theta}
\left| \frac{\theta'}{\theta} + \frac{\beta}{(T-t)|\log(T-t)|} \right|
\leq \frac{C A^2}{(T-t)|\log(T-t)|^{\frac{3}{2}}}.
\end{equation}
\end{prp}

\begin{rmk}\label{remark theta'0<0}
The estimate \eqref{expression theta' rigorous} at $t=0$ remains valid without the assumption that $\theta$ is positive and decreasing.
\end{rmk}


\begin{proof}



In Subsection \ref{subsection theta dynamics}, the condition \eqref{asymp bar theta in t} was imposed to guarantee the convergence of $\theta$ to $0$, under the additional requirement that $\theta$ is decreasing. We now replace condition \eqref{asymp bar theta in t} with the alternative bound given in \eqref{upper bound theta} and use similar computations as for \eqref{theta' proof} to obtain

\begin{equation}\label{theta' proof t>0}
\begin{split}
\theta'(t)=&-rq\gamma\theta^{1+\frac{1}{q}-\frac{N}{2}}(t)(T - t)^{-1} |\log(T - t)|^{\frac{N}{2}}\left( 
\int_{0}^{+\infty}  \varphi_b\left(\xi\right)^{p-1+r} \xi^{N-1} d\xi
\right)\\
&+ O\left(\theta^{1+\frac{1}{q}-\frac{N}{2}}(t)(T-t)^{-1}|\log(T-t)|^{\frac{N}{2}-\frac{1}{2}} \right), \text{ for } t\in (0,t_1].
\end{split}
\end{equation}
For $t=0$, we do not have the positiveness and decreasing assumption. Thus, we replace it with \eqref{theta t(x)>0} and Lemma \ref{existence theta0} to have \eqref{theta' proof t>0} for $t=0$, with similar computations. Combining this with \eqref{theta' proof t>0} the following:
\begin{equation}\label{theta' proof all t}
\begin{split}
\theta'(t)=&-rq\gamma\theta^{1+\frac{1}{q}-\frac{N}{2}}(t)(T - t)^{-1} |\log(T - t)|^{\frac{N}{2}}\left( 
\int_{0}^{+\infty}  \varphi_b\left(\xi\right)^{p-1+r} \xi^{N-1} d\xi
\right)\\
&+ O\left(\theta^{1+\frac{1}{q}-\frac{N}{2}}(t)(T-t)^{-1}|\log(T-t)|^{\frac{N}{2}-\frac{1}{2}} \right), \text{ for } t\in [0, t_1].
\end{split}
\end{equation}

Thus, proceeding as for \eqref{expression theta formal}, we obtain \eqref{expression theta rigorous}. The proof of \eqref{expression theta' rigorous} and \eqref{expression theta'/theta}, follow from \eqref{theta' proof all t} and \eqref{expression theta rigorous}. This concludes the proof of Proposition \ref{dynamics of theta}.
\end{proof}

\textbf{Step 3: Dynamics of $\theta$ with no constraints}\\
We now relax the positivity and the monotonicity assumptions and establish the general behavior of $\theta$.
\begin{cor}[General dynamics of $\theta$ without the decreasing and positivity constraints]\label{cor theta'<0}
For all $K_0>0, \delta_0\in \left(0,\frac{1}{2}\left[ b \frac{K_0^2}{16}\right]^{\frac{1}{p-1}}\right]$ there exists $\epsilon_{14}(K_0,\delta_0)>0$, for all $\epsilon\in (0,\epsilon_{14})$, for all $A\geq 1$, $\alpha_0,C_0>0$, $\eta_0\in (0,\frac{1}{2}]$, there exists $T_{14}(K_0,A,\delta_0,\alpha_0,\eta_0)>0$, for all $T\in (0,T_{14}]$ the following holds: Consider $u$ the solution of equation \eqref{eq}, where the initial data is defined in \eqref{initial data} for some $(d_0, d_1) \in \mathbb{R}^{1+N}$ and $\theta_0$ given by Lemma \ref{existence theta0}. We assume $u \in \mathcal{S}(T, K_0, \epsilon_0, \alpha_0, A, \delta_0, C_0, \eta_0, t) = S(t)$ for all $t \in [0,t_1]$ for some $t_1<T$. Then, the estimates in Proposition \ref{dynamics of theta} and Lemma \ref{lemma consequence shrinking set} hold for all $t\in [0,t_1]$.
\end{cor}
\begin{proof}
The proof follows, if
$$\theta'(t)<0 \text{ and } \theta(t)>0 \text{ for all } t\in [0,t_1].$$
Thanks to Remark \ref{remark theta'0<0}, we have that
\begin{equation}\label{theta'0>0, theta0>0}
\theta'(0)<0 \text{ and } \theta(0)>0.
\end{equation}
We assume by contradiction that there exists $\bar t\in (0,t_1]$ such that 
\begin{equation}\label{contradiction assumption}
\theta'(\bar t)\geq 0\text{ or } \theta(\bar t)\leq 0.
\end{equation}
We treat the first case. Then, from \eqref{theta'0>0, theta0>0} and by continuity, we have that there exits $t^*\in (0,\bar t]$ such that
\begin{equation}\label{theta'(t*)=0}
\theta'(t^*)=0 \text{ and } \theta'(t)<0 \text{ for all } t\in [0,t^*).
\end{equation}
Applying Proposition \ref{dynamics of theta} with $t_1=t^*$, we have that the estimate \eqref{expression theta' rigorous} is valid for $t^*$. Thus,
$$\theta'(t^*)<0,$$
provided that $T$ is small enough. This contradicts \eqref{theta'(t*)=0}. The contradiction for second case of \eqref{contradiction assumption} follows from similar arguments. We conclude the proof of Corollary \ref{cor theta'<0}.
\end{proof}
We would like to derive the following estimate, which complements the estimates of Lemma \ref{Growth estimates}.
\begin{cor}[Growth estimate of $\nabla q$]\label{Growth estimates 2}
Under the quantifiers and hypothesis of Corollary \ref{cor theta'<0}, we have the following estimate on $\nabla q$:
\begin{equation}\label{global estimate grad q}
|(\nabla q)_e|\leq \frac{C(K_0,C_0)}{\sqrt{s}}.
\end{equation}
\end{cor}
\begin{proof}
The proof follows directly from \eqref{bounds grad u in P2}, \eqref{def U cut-off}, \eqref{similarity variables} and \eqref{def q}.
\end{proof}

\subsection{Existence of a solution in the shrinking set}
Here we state the existence result for $u$ within the shrinking set.
\begin{prp} [Existence of a solution in the shrinking set]\label{existence of u in S}
We can find parameters $T$, $K_0$, $\epsilon_0$, $\alpha_0$, $A$, $\delta_0$, $C_0> 0$ and $\frac{1}{2}\geq \eta_0>0$ such that there exist $(d_0, d_1) \in \mathbb{R}^{n+1}$ and $\theta_0>0$ where with initial data $\baru_{d_0, d_1,\theta_0}(0)$ given in \eqref{initial data}, with equation \eqref{eq u} has a solution $\baru$ on $\Omega \times [0, T)$ and
$$
\baru \in S(T, K_0, \epsilon_0, \alpha_0, A, \delta_0, C_0, \eta_0, t), \text{ for all } t \in [0, T),
$$
where $S(T, K_0, \epsilon_0, \alpha_0, A, \delta_0, C_0, \eta_0, T)$ is given in Definition \ref{def shrinking set S}.
\end{prp}
\begin{proof}
This proposition shares similarities with the arguments in \cite{MZ1997Nonlinearity, MZ1997DMJ}, it presents additional challenges due to the novel behavior of the nonlocal term $\theta$. Our method closely aligns with the cited references as well as the technical developments in \cite{DGKZ2022}, particularly because of the role played by $\theta$. The main complication stems from the gradient term appearing in equation \eqref{eq U}. As such, this result contributes uniquely to the methodology for constructing blow-up solutions. For the reader’s benefit, we provide the full proof in Section \ref{Section proof existence of u in S}.
\end{proof}

\section{Proof of Theorem \ref{thm profile}}\label{Section proof thm}
In this section, we give the proof of Theorem \ref{thm profile}. We choose positive constants $T$, $K_0$, $\epsilon_0$, $\alpha_0$, $A$, $\delta_0$, $C_0$ and $\eta_0$  such that Proposition \ref{existence of u in S} is satisfied. We will proceed in three steps:
\begin{itemize}
\item In step 1, we directly derive the estimates \eqref{profile result} and \eqref{asymp theta result} from the previously established results.
\item In step 2, we prove that the solution does not blow up for any spatial point different from $0$.
\item In step 3, we prove \ref{item final profile result}, and consequently, we conclude the single blow-up property of $u$ at $0$.
\end{itemize}

\medskip
\textbf{Step 1:} The proof of  \eqref{profile result} and \eqref{asymp theta result} follow directly from Corollary \ref{cor theta'<0}.

\medskip
\textbf{Step 2:} We will the show that $u$ doesn't blow-up in $\Omega\setminus \{0\}$. Consider $x\in \Omega\setminus\{0\}$. We consider $\epsilon_0>0$ fixed latter. Then, we have two cases:
\begin{itemize}
\item[-] If $|x|\leq \epsilon_0$, then, there  exists $\bar t\in [0,T]$ such that  $|x| \in \left[ \frac{K_0}{4} \sqrt{(T - \bar t) |\log(T - \bar t)|}, \epsilon_0 \right]$. Combining, again Corollary \ref{cor theta'<0} and Lemma \ref{lemma consequence shrinking set}, \eqref{asymp theta result}, and Lemma \ref{lemma equiv varrho}, we have 
\begin{align}
|u(x,t)|& \leq C\left[ |x|^2 \right]^{-\frac{1}{p-1}} |\log |x||^{\frac{1 }{p - 1}}\theta^{-\frac{1}{p-1}}(t(x))\\
&\leq C\left[ |x|^2 \right]^{-\frac{1}{p-1}} |\log |x||^{\frac{1 }{p - 1}}|\log |x||^{-\frac{1}{p-1}(\beta+1)}.
\end{align}
This means that $u$ does not blow up at any $|x|\leq \epsilon_0$ with $x\neq 0$.
\item[-] If $|x|\geq \epsilon_0$, then we have from Corollary \ref{cor theta'<0} and \ref{lemma consequence shrinking set} that
\begin{equation}
|u(x,t)| \leq C.
\end{equation}
\end{itemize}
Therefore, $u$ does not blow up at any $|x|\geq \epsilon_0$. Thus, with \eqref{def u}, we deduce that $v$ does not quench at any $|x|\in \Omega\setminus \{0\}$.

\medskip
\textbf{Step 3:} It remain to prove \ref{item final profile result}. The existence of the quenching profile $v^*$ is equivalent to showing the existence of the blowup profile $u^*$, which is quite the same as in \cite[proof of (ii) in Proposition 3.5]{DZ2019} and thus we only give a proof \eqref{final profile result}.

From similar computations as for \eqref{profile mathcal U0}, we have that
\begin{equation}
\underset{\xi\leq 2\alpha_0|\log(\varrho(x))|^{\frac{1}{4}}}{\sup}\left|\mathcal{U}(x,\xi,0)-\varphi_b\left(z\right)\right|\leq \frac{C}{\sqrt{|\log\varrho(x)|}},
\end{equation}
where $z=\frac{x+\xi\sqrt{T-t(x)}}{\sqrt{(T-t)|\log(T-t)|}}$. Again from similar computation as for \eqref{bound z-K0/4} and the fact that $\varphi_b$ is lipschitz, we obtain for all $\xi\leq 2\alpha_0|\log(\varrho(x))|^{\frac{1}{4}}$ the following
\begin{align}
    \left|\mathcal{U}(x,\xi,0)-\varphi_b\left(\frac{K_0}{4}\right)\right|&\leq \left|\mathcal{U}(x,\xi,0)-\varphi_b\left(z\right)\right|+\left|\varphi_b\left(z\right)-\varphi_b\left(\frac{K_0}{4}\right)\right|\\
    &\leq  \frac{C}{\sqrt{|\log\varrho(x)|}}+\left|z-\frac{K_0}{4}\right|\\
    &\leq \frac{C\alpha_0}{|\log\varrho(x)|^{\frac{1}{4}}}.
\end{align}
We recall that $\hat{ \mathcal{U}}$ is defined from $\hat{ \mathcal{V}}$ by \eqref{def mathcal V}. Since, $\hat{\mathcal{V}}(0)=\varphi_b\left(\frac{K_0}{4}\right)$, we have for $|x|$ small enough and for all $|\xi|\leq 2\alpha_0|\log(T-t(x))|^{\frac{1}{4}}$ the following
\begin{align}\label{mathcal U(0)- hat mathcal U(0)}
        \left| \mathcal{U}(x,\xi,0) - \hat{\mathcal{U}}(0)\right| &\leq \left| \mathcal{U}(x,\xi,0)-\varphi_b\left(\frac{K_0}{4}\right)\right|+\left|\varphi_b\left(\frac{K_0}{4}\right) - \hat{\mathcal{U}}(0)\right|\\
        &\leq  \frac{C\alpha_0}{|\log\varrho(x)|^{\frac{1}{4}}}+ \theta(t(x))^{{\frac{1}{p-1}}}\varrho^{\frac{1}{p-1}}(x)\\
        &\leq \frac{C\alpha_0}{
        |\log\varrho(x)|^{\frac{1}{4}}}, 
\end{align}
where we used Corollary \ref{cor theta'<0} and \eqref{asym rho}. From \eqref{eq hat mathcal V}, we have that $ \hat{\mathcal{U}}$ satisfies the following equation
\begin{equation}\label{eq mathcal hat U}
\partial_\tau \hat{\mathcal{U}} = \frac{\tilde\theta}{\theta(t(x))}\left( \hat{\mathcal{U}} + (\theta(t(x))\varrho)^{\frac{1}{p-1}} \right)^{p}.
\end{equation}
We consider 
$$\eta(x,\xi,\tau)= \mathcal{U}(x,\xi,\tau) - \hat{\mathcal{U}}(\tau),$$
which satisfies the following equation:
\begin{equation}
\partial_\tau \eta = \Delta_{\xi} \eta+G_1+ G_2, \text{ for all }|\xi| \leq 2 \alpha_0 |\log \varrho(x)|^\frac{1}{4},
\end{equation}
where
\begin{align}
&G_1= - 2 \frac{|\nabla_{\xi} \mathcal{U}|^2}{\mathcal{U} +(\theta(t(x))\varrho)^{\frac{1}{p-1}}}, \\
&G_2=\frac{\tilde\theta}{\theta(t(x))}\left[\left(  \mathcal{U}  + (\theta(t(x))\varrho)^{\frac{1}{p-1}} \right)^{p}-\left(\hat{\mathcal{U}}+ (\theta(t(x))\varrho)^{\frac{1}{p-1}} \right)^{p}\right],
\end{align}
thanks to \eqref{eq mathcal U}. Combining \eqref{asym rho}, Corollary \ref{cor theta'<0} and Lemma \ref{pre lemma priori estimates in P2}, we obtain for $x$ small enough

\begin{align}\label{bound G1}
    |G_1|&\leq C\frac{|\nabla_{\xi} \mathcal{U}|^2}{|\mathcal{U}|}\leq \frac{C}{|\log \varrho|}.
\end{align}
The function $x\mapsto x^p$ is Lipschitz on every compact set. Therefore with Lemma \ref{pre lemma priori estimates in P2}, we have
\begin{equation}\label{bound G2}
|G_2|\leq C\frac{\tilde\theta}{\theta(t(x))}|\eta|.
\end{equation}
We now consider
\begin{equation}
\bar{\chi}(\xi) := \chi_0 \left( \frac{|\xi|}{\alpha_0 \sqrt{|\log \varrho(x)|}} \right), 
\end{equation}
and
$$\bar \eta (x,\xi,\tau)=\bar \chi(\xi)\eta(x,\xi,\tau),$$
where $\chi_0$ is given in \eqref{def chi0}.
Therefore, $\bar \eta$ satisfies 
\begin{equation}
\partial_\tau \bar \eta = \Delta_{\xi} \bar \eta-2\nabla \eta \nabla \chi_1-\eta \nabla \chi_1+\chi_1G_1+ \chi_1G_2,
\end{equation}
Using Duhamel’s principal, the solution $\bar \eta$ satisfies:
\begin{equation*}\label{duhamel bar eta}
\bar \eta(\tau) = S(\tau) \bar \eta(0) + \int_0^\tau S(\tau - \sigma) [-2\nabla \eta \nabla \chi_1-\eta \nabla \chi_1+\chi_1G_1+ \chi_1G_2 ](\sigma) d\sigma.
\end{equation*}
As a matter of fact, we have some estimates on $\chi_1$
\begin{equation}\label{estimates chi 1}
|\nabla_\xi \chi_1|\leq \frac{C}{|\log(\varrho(x))|^{\frac{1}{4}}} \quad \text{and} \quad |\Delta_\xi \chi_1| \leq \frac{C}{|\log(\varrho(x))|^{\frac{1}{2}}}.
\end{equation}
We recall these well-known smoothing effects of the heat semigroup: For all $t>0$ and $f\in W^{1,\infty}(\R^N)$, we have
\begin{equation}\label{heat kernel prp}
\|S(t)f\|_{L^\infty(\R^N)}\leq \|f\|_{L^\infty(\R^N)} \text{ and } \| \nabla S(t)  f\|_{L^\infty(\R^N)}\leq C t^{-\frac{1}{2}}\|f\|_{L^\infty(\R^N)}.
\end{equation}
This implies with \eqref{duhamel bar eta}, \eqref{mathcal U(0)- hat mathcal U(0)}, \eqref{bound G1}, \eqref{bound G2}, \eqref{estimates chi 1}, and Definition \ref{def shrinking set S} ii) that for all $\tau \in [0,1)$, we have
\begin{equation*}
\|\bar \eta(\tau)\|_{L^\infty(\mathbb{R}^n)} \leq \frac{C}{|\log \varrho|^{\frac{1}{4}}} + C \int_0^\tau \frac{\tilde\theta}{\theta(t(x))}\|\bar \eta(\sigma)\|_{L^\infty(\mathbb{R}^n)} d\sigma.
\end{equation*}
With Gronwall’s inequality, we get
\begin{equation*}
| \eta(\xi,\tau)| \leq \frac{C}{|\log \varrho|^{\frac{1}{4}}}exp\left(\int^1_0\frac{\tilde\theta(\sigma)}{\theta(t(x))}d\sigma\right), \quad  \text{ for all } |\xi|\leq \alpha_0|\log\varrho|^\frac{1}{4}, \tau \in [0,1),
\end{equation*}
From similar computation as for \cite[(4.20)]{DGKZ2022}, we have
\begin{equation}\label{equiv int tilde theta/theta(x)}
    \int^1_0\frac{\tilde\theta(\sigma)}{\theta(t(x))}d\sigma\sim 1 \text{ as } x\rightarrow 0.
\end{equation}
Therefore,
\begin{equation*}
| \eta(\xi,\tau)| \leq \frac{C}{|\log \varrho|^{\frac{1}{4}}}, \quad  \text{ for all } |\xi|\leq \alpha_0|\log\varrho(x)|^\frac{1}{4}, \tau \in [0,1).
\end{equation*}
In particular,
\begin{equation*}
|\underset{\tau \rightarrow 1}{\lim}\ \mathcal{U}(x,0,\tau)-\hat{\mathcal{U}}(1)| \leq \frac{C}{|\log \varrho|^{\frac{1}{4}}}, \quad  \text{ for all } |\xi|\leq |\log\varrho(x)|^\frac{1}{4}, \tau \in [0,1),
\end{equation*}
It follows from \eqref{asym rho}, \eqref{def mathcal U}, \eqref{asymp theta result} and \eqref{def mathcal V} that
\begin{align}
    u^*(x)&=\underset{\tau \rightarrow 1}{\lim}\theta^{-1}_\infty\varrho(x)^{-\frac{1}{p-1}}|\log \varrho(x)|^{\frac{\beta}{p-1}}\mathcal{U}(x,0,\tau)\\
    &\sim \theta^{-1}_\infty(T-t(x))^{-\frac{1}{p-1}}|\log(T-t(x))|^{\frac{\beta}{p-1}}\hat{\mathcal{U}}(1)\\
    &\sim \theta^{-1}_\infty(T-t(x))^{-\frac{1}{p-1}}|\log(T-t(x))|^{\frac{\beta}{p-1}}\hat{\mathcal{V}}(1)^{-1}.\\
\end{align}
Together with \eqref{equiv int tilde theta/theta(x)}, we  obtain
\begin{align}
    u^*(x)&\sim \left(\theta_\infty\varrho(x)|\log\varrho(x)|^{-\beta}b\frac{K^2_0}{16}\right)^{-\frac{1}{p-1}},\\
\end{align}
and with Lemma \ref{lemma equiv varrho}, we have that there exists $\epsilon_0>0$ such that for all $|x|\leq \epsilon_0$,
$$u^*(x)\sim\left(\theta_\infty b \frac{|x|^2}{(2|\log |x||)^{1+\beta}}\right)^{-\frac{1}{p-1}}.$$
With \eqref{def u}, we obtain \eqref{final profile result}. Hence, with Step 2, we obtain that $u$ blows up only at $0$. This concludes the proof of Theorem \ref{thm profile}.
\section{Existence of a solution in the shrinking set}\label{Section proof existence of u in S}
The goal of this section is to prove Proposition \ref{existence of u in S}. Our strategy relies on two steps:
\begin{itemize}
\item In Subsection \ref{reduction}, we reduce the problem to a finite dimensional problem. More precisely, the task of proving that $u(t)\in S(t)$, for all $t\in [0,T)$ is reduced to the control of $(q_0,q_1)$ in $\mathcal{D}_A$ given by Proposition \ref{Initial data in S(0)}.
\item In Subsection \ref{topo}, we use a topological argument from index theory to prove that there exists $(d_0,d_1)\in \mathcal{D}_A$, such that $u(t)\in S^*(t)$, for all $t\in [0,T)$ .
\end{itemize}
\subsection{Reduction to finite dimensional problem}\label{reduction}
In this subsection, we reformulate the task of controlling $u(t) \in S(t)$ into a finite-dimensional framework, where the focus is placed only on the two unstable modes $q_0$ and $q_1$ within ${\mathcal{D}}_A(s)$ given by Proposition \ref{Initial data in S(0)}. This is the main goal of our subsection:
\begin{prp} [Reduction to a finite dimensional problem]\label{prp reduction}
There exist positive parameters $K_0, \epsilon_0, \alpha_0, A, ;\delta_0, C_0, \eta_0, T>0$ such that for all $(d_0, d_1) \in \mathcal{D}_A$, where $\mathcal{D}_A$ is introduced in Proposition \ref{Initial data in S(0)}, there exists $\theta_0>0$, such that if we consider $u$ a solution of equation \eqref{eq u}, with initial data $u_{d_0, d_1, \theta_0}$ defined in \eqref{initial data}, and $u \in S(t)$ for all $t \in [0, t_*]$ for some $t_*<T$ and $u(t_*) \in \partial S(t_*),$ where $S(t)$ is the shrinking set introduced in Definition~\ref{def shrinking set S}, then the following properties hold:

\begin{itemize}
    \item[(i)] \label{reduction item i}At $s_* = -\ln(T-t_*)$, we have $(q_0, q_1)(s_*) \in \partial \hat{\mathcal{D}}_A(s_*)$.
    \item[(ii)] \label{reduction item ii}We can find $\nu_0 > 0$ such that
\end{itemize}

\begin{equation}
(q_0, q_1)(s_* + \nu) \notin \hat{\mathcal{D}}_A(s_* + \nu), \quad \forall \nu \in (0, \nu_0),
\end{equation}
where $\hat{\mathcal{D}}_A$ is given by \eqref{hat D A}.
\noindent In particular, we have the fact that there exists $\nu_1 > 0$ such that

\begin{equation}
u \notin S(t_* + \nu), \quad \forall \nu \in (0, \nu_1).
\end{equation}
\end{prp}
The proof of this Proposition relies on some improvement of the estimates in the regions of $S(t)$, as defined in Definition \ref{def shrinking set S}. We proceed below in four steps, with the first three ones dedicated to the improvement of the estimates in $P_1,P_2$ and $P_3$ with strict inequalities until $t_1$, except the bounds on $q_0$ and $q_1$.
 In the last step, we conclude the proof of Proposition \ref{prp reduction}. The overall strategy is inspired by \cite{DKZ2021,DZ2019,BK1994,MZ1997Nonlinearity,MZ1997DMJ}, but the presence of the perturbations $\frac{|\nabla u|}{u+1}$ and $\theta$ introduces additional difficulties, which require nontrivial adjustments to the classical approach.

\subsubsection{Estimates in $P_1(t)$}
We have the following first result:
\begin{lemma}[A priori estimates in $P_1(t)$]\label{lemma eq qi}
There exists $K_{10}, A_{10}>0$, for all $K_0\in (0,K_{10}]$, $A\in (0,A_{10}]$ and $rq\gamma \leq A^{-4}$ , $\delta_0\in \left(0,\frac{1}{2}\left[ b \frac{K_0^2}{16}\right]^{\frac{1}{p-1}}\right]$, there exists $\epsilon_{10}(K_0,\delta_0)>0$, for all $\epsilon\in (0,\epsilon_{10})$, $\alpha_0,C_0>0$, $\eta_0\in (0,\frac{1}{2}]$, there exists $T_{10}(K_0,A,\delta_0,\alpha_0,\eta_0)>0$, for all $T\in (0,T_{10}]$ and $l^* > 0$, the exists $l \in [0, l^*]$ and under the following assumptions:

\begin{itemize}
    \item the initial data of $u$ is given by $u_{d_0,d_1,\theta_0}$ as introduced in \eqref{initial data} for some $(d_0,d_1)$ such that $|d_0|,|d_2|\leq 2$ and $\theta_0$ given by Lemma \ref{existence theta0},
    \item $u(t)$ belongs to $S(T, K_0, \epsilon_0, \alpha_0, A, \delta_0, C_0, \eta_0, t)$ for all $t \in [T - e^{-\sigma}, T - e^{-(\sigma + l)}]$, for some $\sigma \geq s_0$, and $l \in [0, l^*]$,
\end{itemize}

then, For all $s \in [\sigma, \sigma + l]$, the following estimates hold:

\begin{enumerate}[label=(\Roman*).]
    \item We have that
{\mathtoolsset{showonlyrefs=false}
\begin{align}\label{bound q 0,1,2 s0}
&\left| q_0'(s) - q_0(s) \right| \leq \frac{C(K_0)A^2}{s^{\frac{3}{2}}},\\
&\left| q_1'(s) - \frac{1}{2} q_1(s) \right| \leq \frac{C(K_0)}{s^2},\\
&\left| q_2'(s) + \frac{2 + \beta}{s} q_2(s) \right| \leq \frac{C A^3}{s^{\frac{5}{2}}}.
\end{align}
}
\item 
\begin{enumerate}[label=(\roman*).]
\item If $\sigma = s_0$, then
\begin{align}\label{bound q- s0}
\frac{|q_-(y,s)|}{1 + |y|^3} &\leq \frac{(1+s-\sigma)}{s^2}\left(e^{s - \sigma}((s - \sigma)^2 +(s-\sigma)+ 1)+  e^{- \frac{s - \sigma}{2}}+  e^{-(s - \sigma)^2}\right).
\end{align}
If $\sigma > s_0$,then
\begin{align}\label{bound q-}
\frac{|q_-(y,s)|}{1 + |y|^3}&\leq  \frac{(1+s-\sigma)}{s^2}\left(e^{s - \sigma}(1 +s-\sigma+ (s - \sigma)^2) A^4 + e^{-\frac{s - \sigma}{2}} A^6 + e^{-(s - \sigma)^2} A^7  \right).
\end{align}

\item If $s = s_0$, then
\begin{align}\label{bound grad q perp s0}
    \frac{|\nabla q_{\perp}(y,s)|}{1 + |y|^3} \leq \frac{C(K_0)}{s^2}\Big(& (1+s-\sigma)e^{\frac{s-\sigma}{2}}+e^{s-\sigma}\sqrt{s-\sigma}+s^{-\frac{1}{2}}e^{\frac{s-\sigma}{2}}((s-\sigma)^2\\
    &+(s-\sigma)^3)\Big).
\end{align}

If $\sigma > s_0$, then
\begin{align}\label{bound grad q perp}
\frac{|\nabla q_{\perp}(y,s)|}{1 + |y|^3} \leq \frac{C(K_0)}{s^2} \Big(&\left( s^{-\frac{1}{2}}e^{\frac{s - \sigma}{2}}((s - \sigma)(1+s-\sigma) + 1)+e^{-\frac{s-\sigma}{2}-(s - \sigma)}\right) A^6\\
&+(s-\sigma)e^{\frac{s-\sigma}{2}}+e^{s-\sigma}\sqrt{s-\sigma}+s^{-\frac{1}{2}}e^{\frac{s-\sigma}{2}}((s-\sigma)^2\\
&+(s-\sigma)^3)\Big).
\end{align}

\item If $\sigma = s_0$, then
\begin{align}\label{bound qe s0}
    |q_e| \leq \frac{C(K_0)}{\sqrt{s}}e^{s-\sigma}(1+s-\sigma). 
\end{align}

If $\sigma > s_0$, then
\begin{align}\label{bound qe}
|q_e|\leq &\frac{C(K_0)}{\sqrt{s}} \left( A^7 e^{-\frac{s - \sigma}{p}} + e^{s - \sigma} A^6 +e^{s-\sigma}(s-\sigma) \right).
\end{align}
 
\end{enumerate}
\end{enumerate}
\end{lemma}

\begin{proof}
\begin{enumerate}[label=(\Roman*).]
    \item From \eqref{eq q equality}, \eqref{eq q}, Lemma \ref{lemma estimate V}, \eqref{bound J 0,1,2}, and Lemmas \ref{lemma estimate B}, \ref{lemma estimate R}, \ref{lemma estimate N}, \ref{Estimate on F}, \eqref{def a}, \eqref{expression b} and \eqref{estimate q b}, we have
    
    \begin{align}
q_0' =& q_0 + \left( a - \frac{2bNk}{(p-1)^2} - \frac{\kappa \beta}{(p-1)} \right) \frac{1}{s} + \kappa \tilde{\lambda}(s) + O\left( \frac{1}{s^2} \right),\\
q_1' = &\frac{1}{2}q_1 + O\left( \frac{1}{s^2} \right),\\
q_2' =& \frac{q_2}{s} \Bigg( -\frac{8bp}{(p-1)^2} + \frac{ap}{\kappa} - \frac{2Nbp}{(p-1)^2}+\frac{16b}{(p-1)^2} - \frac{\beta}{(p-1)} \Bigg) - \frac{q_0bp}{(p-1)^2 s} \\
&+ \frac{1}{s^2} \left(  \frac{bp}{(p-1)^2} \left( \frac{2bNk}{(p-1)^2} - a \right) + \frac{\kappa b}{(p-1)^2} \left( \frac{4(p-2)b}{(p-1)^2} - 1+\frac{\beta}{p-1}\right)\right) \\
&+ \tilde{\lambda} q_2 - \tilde{\lambda}(s) \frac{\kappa b}{(p-1)^2 s} + O\left( \frac{1}{s^3} \right),
\end{align}

\noindent where $\tilde{\lambda}(s) = \frac{1}{p-1} \left( \frac{\bar{\theta}'(s)}{\bar{\theta}(s)} + \frac{\beta}{s} \right)$. Together with \eqref{def a} and \eqref{expression b}, we get some special cancellations which yields

\begin{equation}
q_0' = q_0 + \kappa \tilde{\lambda}(s) + O \left( \frac{1}{s^2} \right),
\end{equation}

\begin{equation}
q_1 = \frac{1}{2} q_1 + O \left( \frac{A^6}{s^2} \right),
\end{equation}

\begin{equation}
q_2' = -\frac{2+\beta}{s} q_2 - \frac{q_0bp}{(p-1)^2 s} + \tilde{\lambda}(s) q_2 - \tilde{\lambda}(s) \frac{\kappa b}{(p - 1)^2 s} + O \left( \frac{A^6}{s^3} \right).
\end{equation}
From Corollary \ref{cor theta'<0}, we have
$$
|\tilde{\lambda}(s)| \leq \frac{CA^2}{s^{\frac{3}{2}}},
$$
provided that $rq\gamma\leq A^{-4}$. Thus, with Definition \ref{def shrinking set S}, item $(I)$ follows.
\item 

From \eqref{eq q equality}, we write $q$ and $\nabla q$ as the form of
\begin{equation}\label{duhamel q}
q(s) = K(s,\sigma)(q(\sigma)) + \int_{\sigma}^{s} K(s,\tau)(J(q)+B(q) + N(q)+ R + F)(\tau) d\tau, 
\end{equation}

\begin{equation}\label{duhamel grad q}
\nabla  q(s) = K_1(s,\sigma)(\nabla q(\sigma)) + \int_{\sigma}^{s} K_1(s,\tau)\nabla (J(q)+B(q) + N(q)+ R + F)(\tau) d\tau, 
\end{equation}
where $K(s,\sigma)$ and $K_1(s,\sigma)$ are the fundamental solution associated to the linear operators $\mathcal{L} + \mathcal{V}$ and $\mathcal{L} + \mathcal{V}-\frac{1}{2}$ respectively. Using \eqref{lemma psi -} for $v=q$ and Definition \ref{def shrinking set S}, we obtain

\begin{align}\label{bound Kq-}
\frac{|\left( K(s,\sigma)q(\sigma)\right)_{-}|}{1 + |y|^3}& \leq C \frac{e^{s - \sigma}(1 + (s - \sigma)^2) }{s \, }\left(\frac{A^3}{\sigma^{\frac{3}{2}}}+\sqrt{s}\frac{A^4}{\sigma^{\frac{3}{2}}}\right)
+ C \frac{e^{-\frac{s - \sigma}{2}} A^6}{\sigma^2}
+ \frac{C e^{-(s - \sigma)^2}A^7}{s^{\frac{3}{2}}\sqrt{\sigma}}\\
&\leq C \left(e^{s - \sigma}(1 + (s - \sigma)^2) A^4 + e^{-\frac{s - \sigma}{2}} A^6 + e^{-(s - \sigma)^2} A^7  \right) \frac{1}{s^2},
\end{align}
provided that $\frac{1}{\sigma} \leq \frac{2}{s}$. We write as the following
\begin{align}\label{bound R + N}
\Bigg|\Bigg(\int_{\sigma}^{s} K(s,\tau)(R + N(q))(\tau) d\tau\Bigg)_-\Bigg|\leq &\left|\int_{\sigma}^{s} (K(s,\tau) J(q))_- d\tau\right|\\
&+\left|\int_{\sigma}^{s} K(s,\tau) \left(R(\tau)-\frac{\beta\kappa}{(p-1)s}\right) d\tau\right|\\
&+\left|\int_{\sigma}^{s} \left(K(s,\tau)\tilde{N}(q)(\tau)\right)_- d\tau\right|,
\end{align}
where $\tilde{N}(q)=N(q)+\frac{\beta\kappa}{(p-1)s}$.
From \eqref{cor bound R}, we have that 
\begin{align}\label{bound K R}
&\left|\int_{\sigma}^{s} K(s,\tau) \left(R(\tau)-\frac{\beta\kappa}{(p-1)\tau}\right) d\tau \right| \leq e^{s-\sigma}(s-\sigma)\frac{C(K_0)(1+|y|^2)}{s^2},\\
\end{align}
and we apply \eqref{lemma psi -} with $v=\tilde{N}(q)$ and \eqref{estimate proj N} to obtain
\begin{align}\label{bound K q}
\left| \frac{(K(s,\tau) \tilde{N}(q)(\tau))_-}{1 + |y|^3} \right| &\leq C \frac{e^{s - \sigma}((s - \sigma)^2 + 1)}{s} \left(\frac{1}{\sigma} +  \frac{1}{\sigma^{2}} + \sqrt{s}\frac{1}{\sigma^{2}}\right)+ C \frac{e^{- \frac{s - \sigma}{2}}}{\sigma^{2} }
+ C \frac{e^{-(s - \sigma)^2}}{s^{\frac{3}{2}}\sigma}\\
&\leq C(K_0)\left( \frac{e^{s - \sigma}((s - \sigma)^2 + 1)}{s^{2}}+ e^{- \frac{s - \sigma}{2}}\frac{1}{s^2 }
+ \frac{e^{-(s - \sigma)^2}}{s^{5/2}}\right).\\
\end{align}
Then, combining this two latters with \eqref{bound R + N} and \eqref{bound K R}, we obtain

\begin{align}\label{bound K R + N final}
&\left(\int_{\sigma}^{s} K(s,\tau)(R + N(q))(\tau) d\tau\right)_-\leq  \frac{C(K_0)(1+|y|^3)(s-\sigma)}{s^2}\\
&\left(e^{s - \sigma}((s - \sigma)^2 +(s-\sigma)+ 1)+  e^{- \frac{s - \sigma}{2}}+  e^{-(s - \sigma)^2}\right),
\end{align}
on one hand. On the other hand, thanks to \eqref{cor bound J}, \eqref{cor bound B in s} and Lemma \ref{Estimate on F}, we have that
\begin{align}\label{bound K J B F final}
&\left| \int_{\sigma}^{s}K(s,\tau) (J(q)+B(q) +F)(\tau)  \right|\\ 
&\leq \int_{\sigma}^{s} e^{s-\tau}\left(\frac{C(K_0,A)(1+|y|^3)}{\tau^\frac{5}{2}}+\frac{C(K_0)(1+|y|^3)}{\tau^2}\right) d\tau \\
&\leq  C(K_0)e^{s-\sigma}\frac{(1+|y|^3)}{s^2}(s-\sigma). \\
\end{align}
Combining this latter with \eqref{duhamel grad q}, \eqref{duhamel q}, \eqref{bound Kq-}, \eqref{bound K R + N final}, we obtain \eqref{bound q-}.
The proof of \eqref{bound q- s0} follows from Proposition \ref{Initial data in S(0)}, \eqref{bound K R + N final} and \eqref{bound K J B F final}.

(iii). Similarly, we apply \eqref{lemma psi +} with $v=q$ and Definition \ref{def shrinking set S} to obtain

\begin{align}\label{bound K q e}
\|(K(s,\sigma)q(\sigma))_e\|_{L^\infty} &\leq C(K_0) e^{s - \sigma} \left(\frac{A^3}{\sigma^{\frac{3}{2}}}+ s^{\frac{1}{2}} \frac{A}{\sigma^{2}}+s \frac{A^4}{\sigma ^{\frac{3}{2}}}+ s^{\frac{3}{2}} \frac{A^6}{\sigma^2 } \right)
+ C(K_0) e^{-\frac{s - \sigma}{p}} \frac{A^7}{\sqrt{s}}\\
&\leq\frac{C(K_0)}{\sqrt{s}} \left( A^7 e^{-\frac{s - \sigma}{p}} + e^{s - \sigma} A^6 \right).
\end{align}

We use \eqref{bound J exterior}, \eqref{cor bound B in s}, \eqref{bound R}, \eqref{Bound N} and Lemma \ref{Estimate on F} have

\begin{align}\label{bound K J B N R + N final 2}
\left| \int_{\sigma}^{s}K(s,\tau) (J(q)+B(q) + N(q)+ R + F)(\tau)d\tau  \right|&\leq \int_{\sigma}^{s} e^{s-\tau}\frac{C(K_0,A)}{\tau} d\tau \\
&\leq C(K_0)e^{s-\sigma}\frac{(s-\sigma)}{\sqrt{s}} .  \\
\end{align}
Therefore, using this latter with \eqref{bound K q e}, we obtain \eqref{bound qe}. The proof \eqref{bound qe s0} follows from Proposition \ref{Initial data in S(0)} and \eqref{bound K J B N R + N final 2}.

(ii). We apply Corollary \ref{cor psi perp} with Definition \ref{def shrinking set S} and Corollary \ref{Growth estimates 2} get
\begin{align}\label{psi perp result}
    &\left\|\frac{\left(K_1(s,\sigma)(\nabla q(\sigma))\right)_\perp}{1+|y|^3}\right\|_{L^\infty}\\
    &\hspace{2cm} \leq C(K_0) \frac{e^{\frac{s - \sigma}{2}}((s - \sigma)(1+s-\sigma) + 1)}{s^{5/2}} A^6+ C(K_0) \frac{A^6e^{-\frac{s-\sigma}{2}-(s - \sigma)}}{s^2}.
\end{align}

Note that
\begin{equation}\label{K1=exp K}
K_1(s,\sigma)=e^{-\frac{s-\sigma}{2}}K(s,\sigma).
\end{equation}
Then, with \eqref{bound nabla R} and \eqref{Bound nabla N}, we have
\begin{align}\label{N+R perp}
|\int_{\sigma}^{s} K_1(s,\tau)\nabla (N(q)+R)(\tau) d\tau|&=|\int_{\sigma}^{s} e^{-\frac{s-\tau}{2}}K(s,\tau)\nabla (N(q)+R)(\tau) d\tau|\\
& \leq C(s-\sigma)e^{\frac{s-\sigma}{2}}\frac{1+|y|^3}{s^2}.
\end{align}

From similar computations as in \cite[proof of (79)]{MZ1997Nonlinearity}, provided that  $V\leq C/s$ et $\left|\frac{\partial^n V}{\partial y^n}\right|\leq C s^{-n/2} $ for $n=0,1,2$ (which is valid thanks to \eqref{eq q} and Corollary \ref{cor theta'<0}) and from \eqref{cor bound J}, \eqref{cor bound B in s} and Lemma \ref{Estimate on F} , we have
\begin{align}\label{J+B+F perp}
\Big|\int_{\sigma}^{s}& K_1(s,\tau)\nabla (J(q)+B(q)+F)(\tau) d\tau\Big|\\
\leq & C(K_0)\frac{(1+|y|^3)}{s^2}(e^{s-\sigma}\sqrt{s-\sigma}+s^{-\frac{1}{2}}e^{\frac{s-\sigma}{2}}((s-\sigma)^2+(s-\sigma)^3))).
\end{align}
\end{enumerate}
Then, combining \eqref{psi perp result}, \eqref{N+R perp} and \eqref{J+B+F perp}, we obtain \eqref{bound grad q perp}. For \eqref{bound grad q perp s0}, we combine Proposition \ref{Initial data in S(0)}, \eqref{N+R perp} and \eqref{J+B+F perp}. This concludes the proof of Lemma \ref{lemma eq qi}.
\end{proof} 
Next, we have the following estimate result in $P_1(t)$.
\begin{prp}[Improvement of some estimates in $ P_1(t) $]\label{prp a priori estimates in P1}
There exists $K_{11}, A_{11}>0$, for all $K_0\in (0,K_{11}]$, $A\in (0,A_{11}]$ and $rq\gamma \leq A^{-4}$ , $\delta_0\in \left(0,\frac{1}{2}\left[ b \frac{K_0^2}{16}\right]^{\frac{1}{p-1}}\right]$, there exists $\epsilon_{11}(K_0,\delta_0)>0$, for all $\epsilon\in (0,\epsilon_{11})$, $\alpha_0,C_0>0$, $\eta_0\in (0,\frac{1}{2}]$, there exists $T_{11}(K_0,A,\delta_0,\alpha_0,\eta_0)>0$, for all $T\in (0,T_{11}]$, the following holds: If $ u $ is a non negative solution of equation \eqref{eq u} satisfying $u(t) \in \mathcal{S}(T, K_0, \epsilon_0, \alpha_0, A, \delta_0, C_0, \eta_0, t),$ for all $ t \in [0, t_{11}] $ for some $ t_{11} \in [0, T) $, with initial data $u_{d_0,d_1,\theta_0}$ as introduced in \eqref{initial data} for some $(d_0,d_1)$ such that $|d_0|,|d_2|\leq 2$ and $\theta_0$ given by Lemma \ref{existence theta0}, then for all $ s \in [-\log T, -\log(T - t_{11})] $, we have the following:

\begin{align}\label{strict ineq q2}
&\forall i, j \in \{1, \cdots, n\}, \quad |q_{2,i,j}(s)| <\frac{A^4}{s^{\frac{3}{2}}},\\
\end{align}
\begin{equation}\label{strict ineq q}
\begin{aligned}
&\left\| \frac{q_-(\cdot, s)}{1 + |y|^3} \right\|_{L^\infty(\mathbb{R}^N)} 
\leq \frac{A^6}{2s^2}, \quad
\left\| \frac{(\nabla q(\cdot, s))_\perp}{1 + |y|^3} \right\|_{L^\infty(\mathbb{R}^N)} 
\leq \frac{A^6}{2s^2},\\
&\|q_e(s)\|_{L^\infty(\mathbb{R}^N)} \leq \frac{A^7}{2\sqrt{s}}.
\end{aligned}
\end{equation}
\end{prp}
\begin{proof}
The proof of \eqref{strict ineq q2} follows from dynamical and contradiction arguments. We assume that \eqref{strict ineq q2} does not hold. Then, by continuity, there exists $s^* \in [-\log T, -\log (T-t_9)]$, such that
$$|q_{2,i,j}(s)| <\frac{A^4}{s^{\frac{3}{2}}}\text{ for all }s\in [s_0,s^*)\text{ and }|q_{2,i,j}(s^*)| =\frac{A^4}{s^{*\frac{3}{2}}}, \text{ for some } i, j \in \{1, \cdots, n\}.$$
Therefore,
\begin{equation}\label{q2'> bound shriking '}
q_{2,i,j}'(s^*)\geq \left[\frac{A^4}{s^{\frac{3}{2}}}\right]'_{s=s^*}=-\frac{3}{2}\frac{A^4}{s^{*\frac{5}{2}}}.
\end{equation}
Thanks to \eqref{bound q 0,1,2 s0}, we have that
$$q_{2,i,j}'(s^*)\leq -\frac{2+\beta}{s^*}q_{2,i,j}(s^*)+\frac{CA^3}{s^{*\frac{5}{2}}}=-\frac{(2+\beta)A^4}{s^{*\frac{5}{2}}}+\frac{CA^3}{s^{*\frac{5}{2}}}<-\frac{3}{2}\frac{A^4}{s^{*\frac{5}{2}}}.$$
Thanks to \eqref{def theta infty and beta}, \eqref{critical condition} and for $A$ large enough, we have this latter contradicts \eqref{q2'> bound shriking '}. This proves \eqref{strict ineq q2}.

For \eqref{strict ineq q}, let $ l_1 > l_2 >0$ (to be fixed later). It is then enough to prove \eqref{strict ineq q2}, on the one hand for all $ s - s_0 \leq l_1 $, and on the other hand for all $ s - s_0 > l_2 $.

\textbf{Case 1:} $ s - s_0 \leq l_1 $. Since we have $ \forall \tau \in [s_0, s], q(\tau) \in \mathcal{V}_A(\tau) $, we apply Lemma \ref{lemma eq qi} with $ l^* = l_1 $, and $ l = s - s_0 $. We recall that $s-s_0\leq l_1$. Then, it is enough to satisfy
\begin{align}
C_0 \left(e^{s - \sigma}((s - \sigma)^2 +(s-\sigma)+ 1)+  e^{- \frac{s - \sigma}{2}}+  e^{-(s - \sigma)^2}\right) &\leq \frac{A^6}{2},\\
C_0\Big((1+s-\sigma)e^{\frac{s-\sigma}{2}}+e^{s-\sigma}\sqrt{s-\sigma}+s^{-\frac{1}{2}}e^{\frac{s-\sigma}{2}}((s-\sigma)^2+(s-\sigma)^3)\Big)&\leq \frac{A^6}{2},\\
C_0e^{s-\sigma}(1+s-\sigma)&\leq \frac{A^7}{2},
\end{align}
or
\begin{align}
C_1\Big((e^{l_1}(l_1^2 +l_1+ 1)+1\Big) &\leq \frac{A^6}{2},\\
C_1\Big(e^{\frac{l_1}{2}}(1 + l_1)+e^{l_1}\sqrt{l_1}+s^{-\frac{1}{2}}e^{\frac{l_1}{2}}(l_1^2+l_1^3)\Big) &\leq \frac{A^6}{2},\\
C_1\Big(e^{l_1}(1+l_1)\Big)&\leq \frac{A^7}{2},
\end{align}
or $C_1\leq A$, $e^{l_1}\leq A$, $l_1^2\leq A$, which is possible if we fix $l_1=\log A$ and $A,s_0$ large enough.

\textbf{Case 2:} $ s - s_0 \geq l_2 $. Again, it is enough to satisfy
\begin{align}
C_0(1+s-\sigma)\Big(& e^{-\frac{s - \sigma}{2}} A^6 + e^{-(s - \sigma)^2} A^7\Big) \leq \frac{A^6}{4},\\
C_0 &  A^6e^{-\frac{s-\sigma}{2}-(s - \sigma)} \leq \frac{A^6}{4},\\
C_0 &  A^7 e^{-\frac{s - \sigma}{p}} \leq \frac{A^7}{4},
\end{align}
and 

\begin{align}
C_0&e^{s - \sigma}(1+s-\sigma)(1 +s-\sigma+ (s - \sigma)^2) A^4 \leq \frac{A^6}{4}\\
C_0\Big( & \frac{e^{\frac{s - \sigma}{2}}((s - \sigma)(1+s-\sigma) + 1)}{s^{\frac{1}{2}}} A^6+(s-\sigma)e^{\frac{s-\sigma}{2}}\\
&+e^{s-\sigma}\sqrt{s-\sigma}+s^{-\frac{1}{2}}e^{\frac{s-\sigma}{2}}((s-\sigma)^2+(s-\sigma)^3)\Big) \leq \frac{A^6}{4},\\
C_0 \Big(&  e^{s - \sigma} A^6 +e^{s-\sigma}(s-\sigma) \Big)\leq \frac{A^7}{4}.
\end{align}
We take $l_2=l^*=l=s-\sigma$. Then, it is sufficient to have
\begin{align}
C_0l_2\Big(& e^{-\frac{l_2}{2}} A^6 + e^{-l_2^2} A^7\Big) \leq \frac{A^6}{4},\\
C_0 &  A^6e^{-\frac{l_2}{2}-l_2} \leq \frac{A^6}{4},\\
C_0 & A^7 e^{-\frac{l_2}{p}} \leq \frac{A^7}{4},
\end{align}
and 

\begin{align}
C_0&e^{l_2}(1+l_2)(1 +l_2+ l_2^2) A^4 \leq \frac{A^6}{4},\\
C_0\Big( & \frac{e^{\frac{l_2}{2}}(l_2(1+l_2) + 1)}{s^{\frac{1}{2}}} A^6+l_2e^{\frac{l_2}{4}},\\
&+e^{l_2}\sqrt{l_2}+s^{-\frac{1}{2}}e^{\frac{l_2}{2}}(l_2^2+l_2^3)\Big) \leq \frac{A^6}{4},\\
C_0 \Big(&  e^{l_2} A^6 +e^{l_2}l_2 \Big)\leq \frac{A^7}{2},
\end{align}

It suffices to require $C_0 e^{l_2} \leq \frac{A}{4}$, which is equivalent to choosing $l_2 = \log \!\left(\frac{A}{4C_0}\right)$, together with $A$ and $s_0$ taken sufficiently large. This concludes the proof of Proposition \ref{prp a priori estimates in P1}.
\end{proof}

\subsubsection{Estimates in $P_2(t)$}
In this subsection, we prove that the solution satisfies the bounds of $S(t)$ for $P_2(t)$ with strict inequalities. We have the following preliminary Lemma.

\begin{lemma}[A priori estimates in $P_2(t)$]\label{pre lemma priori estimates in P2}
There exists $K_{12}, A_{12}>0$, for all $K_0\in (0,K_{12}]$, $A\in (0,A_{12}]$, $\delta_0\in \left(0,\frac{1}{2}\left[ b \frac{K_0^2}{16}\right]^{\frac{1}{p-1}}\right]$, there exist $\alpha_{12}(K_0,A,\delta_0),C_{12}(K_0,A,\delta_0)>0$, for all $\alpha_0\in (0, \alpha_{12}]$, $C_0\in (0,C_{12}]$, there exists $\epsilon_{12}(K_0,A,\alpha_0,\delta_0,C_0)>0$, for all $\epsilon\in (0,\epsilon_{12}]$, there exists $\eta_{12}(K_0,A,\alpha_0,\delta_0,C_0,\epsilon_0)>0$ and $T_{12}(K_0,A,\alpha_0,\delta_0,C_0,\epsilon_0)>0$, for all $\eta_0\in (0,\eta_{12}]$ and $T\in (0,T_{12})$, we have the following property: assume that $u \in S(T, K_0, \epsilon_0, \alpha_0, A, \delta_0, C_0, \eta_0, t)$, $\forall t \in [0, t_{12}]$, for some $t_{12} \in [0, T)$, and initial data $u_{d_0,d_1,\theta_0}$ defined as in \eqref{initial data} for some $|d_0|, |d_1| \leq 2$ and $\theta_0>0$ given in Lemma \ref{existence theta0}, then the following hold for all $|x| \in \left[ \frac{K_0}{4} \sqrt{(T - t_{12}) \log(T - t_{12})}, \epsilon_0 \right]$:

For all $|\xi| \leq \frac{7}{4} \alpha_0 \sqrt{|\log \varrho(x)|}$ and $\tau \in \left[\max\left(0, -\frac{ t_{12}(x)}{\varrho(x)}\right), \frac{t_{12} - t(x)}{\varrho(x)}\right]$ the following estimates are valid
\begin{equation}\label{bound mathcal U}
\frac{1}{16}\varphi_b\left(\frac{K_0}{4}\right)\leq \mathcal{U}(x,\xi,\tau)\leq 16 \kappa,
\end{equation}
\begin{equation}\label{bound nabla U estimate}
|\nabla \mathcal{U}(x, \xi, \tau)| \leq \frac{C(K_0, C_0,A)}{\sqrt{|\log \varrho(x)|}}, 
\end{equation}

\end{lemma}

\begin{proof}
We consider $t\in (0, t_12)$, $|x| \in \big[\frac{K_0}{4}\sqrt{(T-t)|\log(T-t)|},\epsilon_0\big]$,  
$|\xi|\leq \frac{7}{4}\alpha_0\sqrt{|\log \varrho(x)|}$\\ and $\tau \in \big[\max(0,-\frac{t(x)}{\varrho(x)}), \frac{t_{12}-t(x)}{\varrho(x)}\big]$ and denote by  
\begin{align}
&X = x + \xi\sqrt{\varrho(x)},\\
&t= \varrho(x)\tau+t(x).
\end{align}
From $X$'s definition and \eqref{def rho(x), t(x)}, we have that
\begin{equation}\label{bounds X by x}
    |x|(1-7\frac{\alpha_0}{K_0})\leq|X|\leq |x|(1+7\frac{\alpha_0}{K_0}).
\end{equation}
\textbf{Proof of \eqref{bound mathcal U}:} We have the three following cases:  

\textbf{Case 1:} The case where $|X|\leq \frac{K_0}{4}\sqrt{(T-t)|\log(T-t)|}$. We write  
\begin{align}
\mathcal{U}(x,\xi,\tau) &= \varrho^{\frac{1}{p-1}}(x)\theta(t(x))^\frac{1}{p-1}u(X,t)\\
&=\left(\frac{\varrho(x)\theta(t(x))}{(T-t)\theta(t)}\right)^{\frac{1}{p-1}} w(Y,s),
\end{align}
where $Y=\frac{X}{\sqrt{T-t}}, s=-\log(T-t)$. We have from Definition \ref{def shrinking set S} that

\begin{equation}
\frac{1}{2}\varphi_b\left(\frac{K_0}{4}\right)\leq w(Y,s) \leq 2\varphi_b(0),
\end{equation}
and since $|x|\mapsto \varrho(x)$ is an increasing function. Therefore, with \eqref{bounds X by x}, \eqref{t(x)<t} and with Lemma \ref{lemma equiv varrho}, we have that
\begin{align}
T-t\leq \varrho(x)&\leq \varrho\left(\left(\frac{1}{1-7\frac{\alpha_0}{K_0}}\right)\frac{K_0}{4}\sqrt{(T-t)|\log(T-t)|}\right)\leq (1+\bar\epsilon)(T-t),
\end{align}
where $\bar \epsilon\rightarrow 0$ when $\alpha_0 \to 0$ in terms of $K_0$. Therefore, with Corollary \ref{cor theta'<0}, we have
\begin{align}
    \frac{1}{2^{p-1}}\leq\frac{\varrho(x)\theta(t(x))}{(T-t)\theta(t)}\leq 2^{p-1}.
\end{align}
This yields to
\begin{align}\label{bound mathcal U P1}
    \frac{1}{4}\varphi_b\left(\frac{K_0}{4}\right)\leq \mathcal{U}(x,\xi,\tau)\leq 4 \varphi_b(0).
\end{align}
\textbf{Case 2:} The case where $|X|\in [\frac{K_0}{4}\sqrt{(T-t)|\log(T-t)|},\epsilon_0]$. We write  
\begin{align}\label{mathcal U in P2}
\mathcal{U}(x,\xi,\tau) &= \varrho^{\frac{1}{p-1}}(x)\theta(t(x))^\frac{1}{p-1}u(X,t).\\
&=\left(\frac{\varrho(x)\theta(t(x))}{\varrho(X)\theta(t(X))}\right)^{\frac{1}{p-1}} \mathcal{U}\left(X,0,\frac{t-t(X)}{\varrho(X)}\right),
\end{align}
From Definition \ref{def shrinking set S} and \eqref{equiv int tilde theta/theta(x)}, we have that
\begin{equation}
     \frac{1}{4}\left(b\frac{K_0^2}{16}\right)^\frac{1}{p-1}\leq \hat{\mathcal{V}}(1)-\delta_0\leq\mathcal{V}\left(X,0,\frac{t-t(X)}{\varrho(X)}\right)\leq \hat{\mathcal{V}}(0)+\delta_0\leq2\varphi_b\left(\frac{K_0}{4}\right)^{-1},
\end{equation}
provided that $\delta_0\leq \frac{1}{4}\left(b\frac{K_0^2}{16}\right)^\frac{1}{p-1}$. From \eqref{def mathcal V}, \eqref{asym rho} and Corollary \ref{cor theta'<0}, we have that
\begin{equation}\label{bound mathcal U X}
    \frac{1}{4}\varphi_b\left(\frac{K_0}{4}\right)\leq \mathcal{U}\left(X,0,\frac{t-t(X)}{\varrho(X)}\right)\leq 8\left(b\frac{K_0^2}{16}\right)^{-\frac{1}{p-1}}
\end{equation}
for small $\epsilon_0$. Thanks to \eqref{bounds X by x} and Corollary \ref{cor theta'<0}, we have that
\begin{equation}\label{bound rho theta(X) / rho theta (x)}
    \frac{1}{2^{p-1}}\leq \frac{\varrho(x)\theta(t(x))}{\varrho(X)\theta(t(X))}\leq 2^{p-1},
\end{equation}
for $\alpha_0$ small enough in terms of $K_0$. We combine this later with \eqref{mathcal U in P2}and \eqref{bound mathcal U X} to obtain 
\begin{equation}\label{bound mathcal U P2}
    \frac{1}{8}\varphi_b\left(\frac{K_0}{4}\right)\leq \mathcal{U}\leq 16\left(b\frac{K_0^2}{16}\right)^{-\frac{1}{p-1}}.
\end{equation}

\textbf{Case 3:} The case where $|X|\geq \epsilon_0$. We write from Definition \ref{def shrinking set S} the following  
\begin{align}
\mathcal{U}(x,\xi,\tau) &= \left(\varrho(x)\theta(t(x))\right)^{\frac{1}{p-1}}u(X,t).\\
&\geq \left(\varrho(x)\theta(t(x))\right)^{\frac{1}{p-1}} (u(X,0)-\eta_0).\\
\end{align}
Thanks to \eqref{bounds X by x}, we have that 
\begin{equation}\label{X < epsilon0}
|X|\leq 2 \epsilon_0,
\end{equation}for $\alpha_0$ small enough in terms of $K_0$. This latter with \eqref{initial data} and again \eqref{bounds X by x}, Corollary \ref{cor theta'<0} and Lemma \ref{lemma equiv varrho}, we obtain
\begin{equation}\label{bound mathcal U P3 1}
\begin{aligned}
\mathcal{U}(x,\xi,\tau) &\geq\frac{1}{2}\left(\varrho(x)\theta(t(x))\right)^{\frac{1}{p-1}} u(X,0)\\
&\geq\frac{1}{4}\varrho(x)^{\frac{1}{p-1}}  \left[ b\frac{|X|^2}{2|\log |X||}  \right]^{-\frac{1}{p-1}}\\
&\geq\frac{1}{8}  \left[ \frac{K_0^2}{16}b \right]^{-\frac{1}{p-1}}\\
&\geq\frac{1}{8}  \varphi_b\left(\frac{K_0}{4}\right),
\end{aligned}
\end{equation}
provided that $\eta_0\leq \frac{1}{2}\left[ b\frac{4\epsilon_0^2}{(2|\log (2\epsilon_0)|)^{\beta+1}}  \right]^{-\frac{1}{p-1}}$ and for $\alpha_0$ small enough in terms of $K_0$. From similar computations, we obtain 
\begin{align}\label{bound mathcal U P3 2}
\mathcal{U}(x,\xi,\tau) \leq 8   \left[ \frac{K_0^2}{16}b \right]^{-\frac{1}{p-1}}.
\end{align}
Therefore, combining \eqref{bound mathcal U P1}, \eqref{bound mathcal U P2}, \eqref{bound mathcal U P3 1} and \eqref{bound mathcal U P3 2}, we obtain \eqref{bound mathcal U}.

\textbf{Proof of \eqref{bound nabla U estimate}:} \textbf{Case 1:} The case where $|X|\leq \frac{K_0}{4}\sqrt{(T-t)|\log(T-t)|}$. With similar computations as before and from \eqref{estimate grad q 3}, Corollary \ref{Growth estimates 2}, \eqref{t(x)<t} and Corollary \ref{cor theta'<0}, we have 
\begin{equation}\label{bound nabla mathcal U 1}
\begin{aligned}
|\nabla_\xi\mathcal{U}(x,\xi,\tau)| &= \varrho^{\frac{1}{p-1}+\frac{1}{2}}(x)\theta(t(x))^\frac{1}{p-1}|\nabla_X u(X,t)|\\
&=\left(\frac{\varrho(x)}{(T-t)}\right)^{\frac{1}{p-1}+\frac{1}{2}}\left(\frac{\theta(t(x))}{\theta(t)}\right)^{\frac{1}{p-1}}| \nabla_Y W(Y,s)|\\
&\leq \frac{C(K_0,A)}{\sqrt{|\log(T-t)|}}\\
&\leq \frac{C(K_0,A)}{\sqrt{|\log \varrho(x)|}}.\\
\end{aligned}
\end{equation}
\textbf{Case 2:} The case where $|X|\in [\frac{K_0}{4}\sqrt{(T-t)|\log(T-t)|},\epsilon_0]$. We write  
\begin{equation}
\begin{aligned}\label{grad xi U}
\nabla_\xi\mathcal{U}(x,\xi,\tau) &= \varrho^{\frac{1}{p-1}+\frac{1}{2}}(x)\theta(t(x))^\frac{1}{p-1}\nabla_X u(X,t).\\
&=\left(\frac{\varrho(x)}{\varrho(X)}\right)^{\frac{1}{p-1}+\frac{1}{2}} \left(\frac{\theta(t(x))}{\theta(t(X))}\right)^{\frac{1}{p-1}} \nabla_\Xi\left[\mathcal{U}\left(X,\Xi,\frac{t-t(X)}{\varrho(X)}\right)\right]_{\Xi=0}.
\end{aligned}
\end{equation}
From Definition \ref{def shrinking set S} and \eqref{bounds X by x}, we have that
\begin{equation}
     \nabla_\Xi\left[\mathcal{U}\left(X,\Xi,\frac{t-t(X)}{\varrho(X)}\right)\right]_{\Xi=0}\leq \frac{C_0}{\sqrt{|\log \varrho(X)|}}\leq \frac{2C_0}{\sqrt{|\log \varrho(x)|}},
\end{equation}
provided that $\alpha_0$ is small enough in terms of $K_0$. We combine this with similar computations as for \eqref{bound rho theta(X) / rho theta (x)} with \eqref{grad xi U} and get
\begin{equation}\label{bound nabla mathcal U 2}
\begin{aligned}
|\nabla_\xi\mathcal{U}(x,\xi,\tau)|\leq \frac{3C_0}{\sqrt{|\log \varrho(x)|}}.
\end{aligned}
\end{equation}

\textbf{Case 3:} The case where $|X|\geq \epsilon_0$. We write from Definition \ref{def shrinking set S} the following  
\begin{align}\label{bound grad mathcal U P3 1}
|\nabla_\xi\mathcal{U}(x,\xi,\tau)| &= \varrho(x)^{\frac{1}{p-1}+\frac{1}{2}}\theta(t(x))^{\frac{1}{p-1}}|\nabla_Xu(X,t)|\\
&\leq  \varrho(x)^{\frac{1}{p-1}+\frac{1}{2}}\theta(t(x))^{\frac{1}{p-1}}(|\nabla_Xu(X,0)|\\
&\hspace{0.5cm}+|\nabla_XS(t)u(X,0)-\nabla_Xu(X,0)|+\eta_0).\\
\end{align}
Thanks to \eqref{X < epsilon0} and with straightforward computations, it is easy to show 
\begin{align}\label{equi nabla u}
    &\nabla_Xu(X,0)=\nabla_XH^*(X)\sim  -\frac{\sqrt{2b}}{(p-1)}\frac{1}{\sqrt{|\log|x||}}\left[ b\frac{|X|^2}{2|\log |X||} \right]^{-\frac{1}{p-1}-\frac{1}{2}}\left[\frac{\theta_\infty}{(2|\log|X||)^\beta}\right]^{-\frac{1}{p-1}},
\end{align}
Thanks to Lemma \ref{lemma bound nabla (Sf-f)}, we have that
\begin{align}\label{nabla Su0-u0}
|\nabla_XS(t)u(X,0)-\nabla_Xu(X,0)|\leq C T\leq \eta_0,
\end{align}
for $T$ small enough. Thanks to \eqref{X < epsilon0}, we can take $$\eta_0\leq \underset{|x|\in [\epsilon_0,2\epsilon_0]}{\min}\nabla_xu(x,0)=\underset{|x|\in [\epsilon_0,2\epsilon_0]}{\min}\nabla_xH^*(x).$$
Combining this latter with \eqref{equi nabla u}, \eqref{nabla Su0-u0}, \eqref{bounds X by x}, Lemma \ref{lemma equiv varrho} and Corollary \ref{cor theta'<0}, we get
\begin{equation}\label{bound nabla mathcal U 3}
\begin{aligned}
\nabla_\xi\mathcal{U}(x,\xi,\tau)&\leq  3\varrho(x)^{\frac{1}{p-1}+\frac{1}{2}}\theta(t(x))^{\frac{1}{p-1}}\nabla_Xu(X,0)\\
&\leq \frac{C}{\sqrt{|\log \varrho(x)|}},\\
\end{aligned}
\end{equation}
for $\alpha_0$ small enough in terms of $K_0$. Therefore, combining \eqref{bound nabla mathcal U 1}, \eqref{bound nabla mathcal U 2}, \eqref{bound nabla mathcal U 3} and \eqref{bound mathcal U P3 2}, we obtain \eqref{bound nabla U estimate}. This concludes the proof of Lemma \ref{pre lemma priori estimates in P2}.
\end{proof}

Now, we use the parabolic regularity to derive the following:

\begin{prp}[Improvement of the estimates in $P_2(t)$]\label{prp a priori estimates in P2}
There exists $K_{13}, A_{13}>0$, for all $K_0\in (0,K_{13}]$, $A\in (0,A_{13}]$, there exists $\delta_{13}\in \left(0, \frac{1}{4}\left(b\frac{K_0^2}{16}\right)^\frac{1}{p-1}\right]$, for all $\delta_0\in \left(0,\delta_{13}\right]$, there exist $\alpha_{13}(K_0,A,\delta_0),C_{13}(K_0,A,\delta_0)>0$, for all $\alpha_0\in (0, \alpha_{13}]$, $C_0\in (0,C_{13}]$, there exists $\epsilon_{13}(K_0,A,\alpha_0,\delta_0,C_0)>0$, for all $\epsilon\in (0,\epsilon_{13}]$, there exists $\eta_{13}(K_0,A,\alpha_0,\delta_0,C_0,\epsilon_0)>0$ and $T_{13}(K_0,A,\alpha_0,\delta_0,C_0,\epsilon_0)>0$, for all $\eta_0\in (0,\eta_{13}]$ and $T\in (0,T_{13})$, we have the following property: assume that $u \in S(T, K_0, \epsilon_0, \alpha_0, A, \delta_0, C_0, \eta_0)$ for all $t \in [0, t_{13}]$ for some $t_{13} \in [0, T)$, and initial data $u_{d_0,d_1,\theta_0}$ defined as in \eqref{initial data}, for some $|d_0|, |d_1| \leq 2$ and $\theta_0$ given by Lemma \ref{existence theta0}, 0 then for all $|x| \in \left[ \frac{K_0}{4} \sqrt{(T - t_{13}) \log(T - t_{13})}, \epsilon_0 \right]$, $|\xi| \leq \alpha_0 \sqrt{|\log \varrho(x)|}$ and $\tau_{13} = \frac{t_{13} - t(x)}{\varrho(x)}$, we have
\begin{align}\label{better estimate mathcal V in P2}
|\mathcal{V}(x, \xi, \tau_{13}) - \hat{\mathcal{V}}(x, \tau_{13})| \leq \frac{\delta_0}{2},
\end{align}
and
\begin{align}\label{better estimate grad mathcal V in P2}
|\nabla \mathcal{U}(x, \xi, \tau_{13})| \leq \frac{C_0}{2\sqrt{|\log \varrho(x)|}}.
\end{align}
\end{prp}
\begin{proof}
For \eqref{better estimate mathcal V in P2}, we introduce, for all $|\xi| \leq \alpha_0 \sqrt{|\log \varrho(x)|}$ and $\tau \in [-\frac{t(x)}{T-t(x)},\tau_{13}]$, the following:
\begin{equation}
    \bar{\mathcal{V}}:=\mathcal{V}(x_0, \xi, \tau) - \hat{\mathcal{V}}(x_0,\tau),
\end{equation}
where $\mathcal{V}$ and $\hat{\mathcal{V}}$ are given in \eqref{def mathcal V},  \eqref{def V hat}. Then, $\bar{\mathcal{V}}$ satisfies
$$
\partial_\tau \bar{\mathcal{V}} = \Delta_\xi \bar{\mathcal{V}} -\frac{\tilde \theta}{\theta(t(x))}\left(\frac{1}{\mathcal{V}^{p-2}}-\frac{1}{\hat{\mathcal{V}}^{p-2}}\right) .
$$
We now consider 
\begin{align}\label{def chi7/4}
&\chi_{\frac{7}{4}} \in C_0^{\infty}([0,+\infty)),  \operatorname{supp}\left(\chi_{\frac{7}{4}}\right) \subset[0,\frac{7}{4}],\\
&\quad 0 \leq \chi_0(x) \leq 1,\quad \text{for all }  x \in [0,+\infty), \chi_0(x)=1, \text{for all }  x \in[0,1],
\end{align}
and 
\begin{equation}\label{tilde chi}
\tilde{\chi}(\xi) := \chi_{\frac{7}{4}} \left( \frac{|\xi|}{\alpha_0 \sqrt{|\log \varrho(x)|}} \right), 
\end{equation}
Next, we define
$$
\tilde{\mathcal{V}} := \tilde{\chi} \bar{\mathcal{V}}.
$$
Hence, we derive
\begin{equation}\label{eq tilde V}
\partial_\tau \tilde{\mathcal{V}} = \Delta \tilde{\mathcal{V}} -\tilde \chi\frac{\tilde \theta}{\theta(t(x))}\left(\frac{1}{\mathcal{V}^{p-2}}-\frac{1}{\hat{\mathcal{V}}^{p-2}}\right) +G(\bar{\mathcal{V}}), 
\end{equation}
where
$$G(\bar{\mathcal{V}})=-2div( \bar{\mathcal{V}}\nabla\tilde\chi)+\bar{\mathcal{V}}\Delta\tilde\chi,$$
and whose solution is given for all $\tau \in [\tau_0, \tau^*]$ by 
\begin{equation}\label{sol tilde V}
    \tilde{\mathcal{V}}(\tau) = e^{(\tau - \tau_0) \Delta} \tilde{\mathcal{V}}(\tau_0) + \int_{\tau_0}^\tau e^{(\tau - \tau') \Delta} \left(\tilde \chi\frac{\tilde \theta(\tau')}{\theta(t(x))}\left(\frac{1}{\mathcal{V}^{p-2}(\tau')}-\frac{1}{\hat{\mathcal{V}}^{p-2}(\tau')}\right) +G( \bar{\mathcal{V}}, \tau'), \right) d\tau'.
\end{equation}
We have 
\begin{equation}\label{estimates tilde chi}
|\nabla_\xi \chi_{\frac{7}{4}}|\leq \frac{C}{|\log(\varrho(x))|^{\frac{1}{4}}} \quad \text{and} \quad |\Delta_\xi \chi_{\frac{7}{4}}| \leq \frac{C}{|\log(\varrho(x))|^{\frac{1}{2}}}.
\end{equation}
Combining \eqref{heat kernel prp}, Lemma \ref{pre lemma priori estimates in P2} and \eqref{estimates tilde chi}, we have that
\begin{align}\label{bound G}
\int_{\tau_0}^\tau |e^{(\tau - \tau') \Delta}G( \mathcal{V}, \tau')| &\leq \frac{C(K_0,C_0,\delta_0)}{\sqrt{|\log \varrho(x)|}}\int_{\tau_0}^\tau \left((\tau-\tau')^{-\frac{1}{2}}+1\right)d \tau'\\
&\leq \frac{C(K_0,A,C_0,\delta_0)}{\sqrt{|\log \varrho(x)|}}\,
\end{align}
On the other hand, we have
\begin{equation}\label{bound V-hatV ^ p-2}
\tilde \chi\left|\frac{1}{\mathcal{V}^{p-2}}-\frac{1}{\hat{\mathcal{V}}^{p-2}}\right|\leq C|\bar{\mathcal{V}}|,
\end{equation}
provided that $\delta_0\leq \frac{1}{4}\left(b\frac{K_0^2}{16}\right)^\frac{1}{p-1}$. Thus, with \eqref{heat kernel prp}, we get
\begin{equation}
    \int_{\tau_0}^\tau \left|e^{(\tau - \tau') \Delta} \left(\tilde \chi\frac{\tilde \theta(\tau')}{\theta(t(x))}\left(\frac{1}{\mathcal{V}^{p-2}(\tau')}-\frac{1}{\hat{\mathcal{V}}^{p-2}( \tau')}\right) \right) d\tau'\right|\leq C\int_{\tau_0}^\tau \frac{\theta(\tau')}{\theta(t(x))} \|\tilde{\mathcal{V}}(\tau')\|_{L^\infty}d \tau'.
\end{equation}
Combining this latter with Proposition \ref{Initial data in S(0)}, \eqref{sol tilde V} and \eqref{bound G}, we get
$$
\|\tilde{\mathcal{V}}(\tau)\|_{L^\infty} \leq C \left( \bar\delta_0 +  \frac{C(K_0,C_0,\delta_0)}{\sqrt{|\log \varrho(x)|}} \right) + C \int_{\tau_0}^\tau \frac{\tilde\theta}{\theta(t(x))}\|\tilde{\mathcal{V}}(\tau')\|_{L^\infty} d \tau'.
$$
for some $\bar \delta_0$, which will be fixed shortly. Via Gronwall’s lemma, Corollary \ref{cor theta'<0} and \eqref{bounds tilde theta/theta(t(x))}, we finally obtain
$$
\|\tilde{\mathcal{V}}(\tau)\|_{L^\infty} \leq C \left( \delta_6 +\frac{C(K_0,C_0,\delta_0)}{\sqrt{|\log \varrho(x)|}} \right) \leq \frac{\delta_0}{2},
$$
provided that $\bar\delta_0 \leq \delta_7$ and $\epsilon_0$ is small. The proof of \eqref{better estimate grad mathcal V in P2} follows the same technique and so it is omitted. This concludes the proof of Proposition \ref{prp a priori estimates in P2}.
\end{proof}

\subsubsection{Estimates in $P_3(t)$}
This subsection is devoted to the estimates in $P_3$. The result is the following:
\begin{prp}[Improvement of the estimates in $P_3(t)$]\label{prp a priori estimates in P3}
For all $K_0>0, \delta_0\in \left(0,\frac{1}{2}\left[ b \frac{K_0^2}{16}\right]^{\frac{1}{p-1}}\right]$ there exists $\epsilon_{14}(K_0,\delta_0)>0$, for all $\epsilon\in (0,\epsilon_{14})$, for all $A\geq 1$, $\alpha_0,C_0>0$, $\eta_0\in (0,\frac{1}{2}]$, there exists $T_{14}(K_0,A,\delta_0,\alpha_0,\eta_0)>0$, for all $T\in (0,T_{14}]$, the following property holds: assume that $u$ be a non negative solution to \eqref{eq u} for all $t \in [0,t_{14}]$ and it is satisfied that 
$u(t) \in S(T,K_0,\epsilon_0,\alpha_0,A,\delta_0,C_0,\eta_0,t)$ for all $t \in [0,t_{14}]$ corresponding to initial data $u_{d_0,d_1,\theta_0}$ ( defined in \eqref{initial data}) for some $|d_0|,|d_1|\leq 2$ and $\theta_0$ given by Lemma \ref{existence theta0}, then for all $x \in \Omega \cap \{ |x| \geq \tfrac{\epsilon_0}{4}\}$, we have
\begin{align}\label{improve estimate u-u0}
|u(x,t_{14}) - u(0)| \leq \tfrac{\eta_0}{2},
\end{align}
and
\begin{align}\label{improve estimate nabla (u-Su0)}
\left| \nabla u(x,t_{14}) - \nabla S(t_{14})u(0) \right| \leq \tfrac{\eta_0}{2}.
\end{align}
\end{prp}
\begin{proof}
\textbf{Proof of \eqref{improve estimate u-u0}:}
We introduce, for all $x \in \Omega \cap \{ |x| \geq \tfrac{\epsilon_0}{4}\}$, $t \in [0,t_{12}]$ and the following:
\begin{align}
   \chi_{\epsilon_0}\in C^\infty_0,\text{ where  } \chi_{\epsilon_0}(x)\equiv 1 \text{ when } |x|\geq \frac{\epsilon_0}{2} \text{ and } \chi_{\epsilon_0}(x)\equiv 0\text{ when } |x|\leq \frac{\epsilon_0}{4},
\end{align}
and 
\begin{align}
    u_{\epsilon_0}=u\chi_{\epsilon_0}.
\end{align}
Thanks to \eqref{eq u}, the equation satified by $u_{\epsilon_0}$ is 

\begin{align}
    \partial_t u_{\epsilon_0}= \Delta u_{\epsilon_0} - 2\frac{\chi_{\epsilon_0}|\nabla \baru|^2}{\chi_{\epsilon_0}(\baru+1)}    +
\displaystyle  \theta\chi_{\epsilon_0}(1 +\baru)^{p}+G_{\epsilon_0},
\end{align}
where 
$$G_{\epsilon_0}=-2 div(\nabla \chi_{\epsilon_0}  u)+u\Delta \chi_{\epsilon_0}.$$
We use duhamel formulation to write  $u_{\epsilon_0}$ as
\begin{align}\label{duhamel u epsilon 0}
     u_{\epsilon_0}(t)= S(t) u_{\epsilon_0}(0)+\int^t_0  S(t-\tau)\left(  - 2\chi_{\epsilon_0}\frac{|\nabla \baru|^2}{(\baru+1)}    +  \theta\chi_{\epsilon_0}(1 +\baru)^{p}+G_{\epsilon_0}(\tau)\right)(\tau) d\tau.
\end{align}
We write
\begin{align}\label{improve estimate P3 computation 1}
    u_{\epsilon_0}(t)-u_{\epsilon_0}(0)=\left(u_{\epsilon_0}(t)-S(t) u_{\epsilon_0}(0)\right)+ \left(S(t) u_{\epsilon_0}(0)-u_{\epsilon_0}(0)\right).
\end{align}
From Lemma \ref{lemma bound nabla (Sf-f)}, we have that
\begin{align}\label{improve estimate P3 computation 2}
    \left|S(t) u_{\epsilon_0}(0)-u_{\epsilon_0}(0)\right|\leq \frac{\eta_0}{4}, \text{ for all } t\leq T,
\end{align}
for $T$ small enough. Therefore, with \eqref{heat kernel prp}, we get
\begin{align}\label{improve estimate P3 computation 3}
    \| u_{\epsilon_0}(t)-S(t)u_{\epsilon_0}(0)\|&\leq T\left\| 2\chi_{\epsilon_0}\frac{|\nabla \baru|^2}{(\baru+1)}    +  \theta\chi_{\epsilon_0}(1 +\baru)^{p}+G_{\epsilon_0}\right\|_{L^\infty_{x,t}}. \\
\end{align}
From Definition \ref{def shrinking set S},  Lemma \ref{lemma bound nabla (Sf-f)} and \eqref{initial data in S I}, we have that
\begin{align}\label{bound nabla u in P3}
|\nabla \baru|&\leq |\nabla \baru-\nabla S(t)\baru|+|\nabla S(t)\baru-\nabla \baru|+|\nabla \baru_0|\\
&\leq\eta_0+CT+C(\epsilon_0)\\
&\leq C(\epsilon_0, \eta_0).
\end{align}
Again with definition \ref{def shrinking set S}, we have that
\begin{align}\label{bound u in P3}
|u|\leq C \text{ and } 1+u\geq \frac{1}{2},
\end{align}
provided that $\eta_0\leq  \frac{1}{2}$. Combining this latter with Corollary \ref{cor theta'<0}, \eqref{bound nabla u in P3} and \eqref{improve estimate P3 computation 3}, we obtain 
\begin{align}
     \| u_{\epsilon_0}(t)-S(t)u_{\epsilon_0}(0)\|&\leq \frac{\eta_0}{4}, \\
\end{align}
for $T$ small enough. Together with \eqref{improve estimate P3 computation 1} and \eqref{improve estimate P3 computation 2}, we obtain \eqref{improve estimate u-u0}.

\textbf{Proof of \eqref{improve estimate nabla (u-Su0)}:} We write from \eqref{duhamel u epsilon 0}, the following 
\begin{align}
     \nabla u_{\epsilon_0}(t)= \nabla S(t) u_{\epsilon_0}(0)+\int^t_0  \nabla S(t-\tau)\left(  - 2\chi_{\epsilon_0}\frac{|\nabla \baru|^2}{(\baru+1)}    +  \theta\chi_{\epsilon_0}(1 +\baru)^{p}+G_{\epsilon_0}(\tau)\right)(\tau) d\tau.
\end{align}

Then, with \eqref{bound u in P3} and \eqref{bound nabla u in P3}, we obtain
\begin{align}
     |\nabla u_{\epsilon_0}(t)- \nabla S(t) u_{\epsilon_0}(0)|&\leq \sqrt{T}  \left\|2\chi_{\epsilon_0}\frac{|\nabla \baru|^2}{(\baru+1)}    +  \theta\chi_{\epsilon_0}(1 +\baru)^{p}+G_{\epsilon_0}(\tau)\right\|\\
     &\leq CT\\
     &\leq \frac{\eta_0}{2},
\end{align}
for $T$ small enough, which concludes the proof of Lemma \ref{prp a priori estimates in P3}.
\end{proof}
\subsubsection{Conclusion of the proof of Proposition \ref{prp reduction}}
It this section, we conclude the proof of Proposition \ref{prp reduction} using the estimates obtained previously in Subsection \ref{reduction}.
\begin{proof}[Proof of Proposition \ref{prp reduction}]
We begin by selecting parameters $K_0$, $A$, $\delta_0$, $\epsilon_0 $, $C_0$, $\alpha_0,\eta_0 > 0$, and $T_{15} > 0$ such that Propositions \ref{Initial data in S(0)}, \ref{prp a priori estimates in P1}, \ref{prp a priori estimates in P2} and \ref{prp a priori estimates in P3} are satisfied for all $T<T_{13}$, for some $u$ that solves \eqref{eq u}, with initial data $u_{d_0, d_1, \theta_0}$ defined in \eqref{initial data}, where $\theta_0$ is given by Lemma \ref{existence theta0} and fulfilling the following condition:
$$
u(t) \in S(T, K_0, \alpha_0, \epsilon_0, A, \delta_0, C_0, \eta_0, t) = S(t),
$$
for every $t \in [0, t_*]$ where $t_* \in (0, T]$, and
\begin{equation}\label{u(t*) in partial s(t*)}
u(t_*) \in \partial S(t_*).
\end{equation}
From similar argument as in \cite[Proposition 1.2.2 and Corollary 1.2.4]{KS2018}, one can show that $u\in  C^{2,1}(\Omega \times [0, t^*)) \cap C(\bar{\Omega} \times [0, t^*])$.
\begin{itemize}
\item[(i)] In Propositions \ref{prp a priori estimates in P3}, \ref{prp a priori estimates in P2} and \ref{prp a priori estimates in P1}, we improved all the estimates of $S^*$ except for $q_0$ and $q_1$. Then, it follows from \eqref{u(t*) in partial s(t*)} that
$$(q_0,q_1)\in \partial \hat{\mathcal{D}}_A.$$
\item[(ii)] It is enough to show that the flow of $(q_0,q_{1,1}, q_{1,2},...,q_{1,N})$ is transverse outgoing at the exiting time $t^*$ with it's barrier. More precisely, it is enough to show that:
\begin{itemize}
    \item if $q_0(s^*)=\omega \frac{A^3}{s^{*\frac{3}{2}}}$, for some $\omega\in \{-1,1\}$, then $\omega q_0'(s^*)\geq  \partial_s\left[\frac{A^3}{s^{\frac{3}{2}}}\right]_{s=s^*}$,
    \item if there exists $i\in \{1,...,N\}$, such that $q_{1,i}(s^*)=\omega \frac{A}{s^{*2}}$, for some $\omega\in \{-1,1\}$, then $\omega q_{1,i}'(s^*)\geq  \partial_s\left[\frac{A}{s^2}\right]_{s=s^*}$.
\end{itemize}
Without loss of generality, we assume that the first case occurs, since the argument is the same for the second case. Moreover, in the first case, we consider only the situation where 
$$q_0(s_*)=\frac{A^3}{s^\frac{3}{2}},$$
as the other situation follows by the same reasoning. Using Lemma \ref{lemma eq qi}, we have that
$$ q_0'(s_*)\geq q_0(s_*)-\frac{C(K_0)A^2}{s_*^{\frac{3}{2}}}=\frac{A^3}{s^\frac{3}{2}}-\frac{C(K_0)A^2}{s_*^{\frac{3}{2}}}>-\frac{3}{2}\frac{A^3}{s^\frac{5}{2}}=\left[\frac{A^3}{s^\frac{3}{2}}\right]_{s=s_*},$$
for $s_0$ large enough. Hence, \textit{(ii)}. This concludes the proof of Proposition \ref{prp reduction}.
\end{itemize}
\end{proof}
\subsection{Topological argument}\label{topo} 
In this subsection, our aim is to conclude the proof of  Proposition \ref{existence of u in S}. More specifically, we show that there exist parameters $T, K_0, \epsilon_0, \alpha_0, A, \delta_0, C_0, \eta_0$ and $(d_0, d_1) \in \mathcal{D}_A$ such that, for the initial data $u_{d_0, d_1,\theta_0}(0)$ defined in \eqref{initial data} and where $\theta_0$ is given by  Lemma \ref{existence theta0}, the corresponding solution of \eqref{eq u} is well-defined on $[0, T)$ and belongs to $S^*(t)$, for all $t\in [0,T)$, where $S^*(t)$ is introduced in Definition \ref{def shrinking set S}.  

To proceed, take parameters $T > 0$, $K_0 > 0$, $\epsilon_0 > 0$, $\alpha_0 > 0$, $A > 0$, $\delta_0 > 0$, $C_0, \eta_0 > 0$ so that Proposition \ref{prp reduction} is satisfied. By Proposition \ref{Initial data in S(0)}, we obtain
$$
\forall (d_0, d_1) \in \mathcal{D}_A, \quad u_{d_0, d_1,\theta_0}(0) \in S(0),
$$
where $\theta_0$ is given by  Lemma \ref{existence theta0}. Hence, for every $(d_0, d_1) \in \mathcal{D}_A$, we define $t^*(d_0, d_1) \in [0, T)$ as the largest time such that
$$
u(t) \in S(t), \qquad \forall t \in [0, t^*(d_0, d_1)).
$$
Then, two distinct scenarios may occur:
\begin{enumerate}
  \item If $t^*(d_0, d_1) = T$ for some $(d_0, d_1) \in \mathcal{D}_A$, then the proof is complete.
  \item If $t^*(d_0, d_1) < T$ for every $(d_0, d_1) \in \mathcal{D}_A$, this situation actually cannot happen, as we will show below.
\end{enumerate}

We assume that Case 2. holds and we will reach a contradiction. By continuity in time of the solution and by the definition of the maximal time $t^*(d_0, d_1)$, it follows that
$$
u(t^*(d_0, d_1)) \in \partial S(t^*(d_0, d_1)).
$$
Using item \textit{(i)} of Proposition \ref{prp reduction}, we deduce the following:
$$
(q_0, q_1)(s_*(d_1, d_2)) \in \partial \hat{\mathcal{D}}_A(s_*(d_0, d_1)).
$$
This allows us to define the following mapping:
\begin{equation}
\left\{
\begin{matrix}
\Phi : &\mathcal{D}_A &\longrightarrow & ([-1, 1] \times [-1, 1]^n)\\
&(d_0, d_1) &\longmapsto &  \left(q_0(s_*(d_0, d_1))\frac{s_*^\frac{3}{2}(d_0, d_1)}{A^3}, q_1(s_*(d_0, d_1))\frac{s_*^2(d_0, d_1)}{A}\right).
\end{matrix}
\right.
\end{equation}
Together with item \textit{(ii)} of Proposition \ref{prp reduction}, we have that $(q_0,q_1)(s_*)$ is tranverse on $\partial \hat{\mathcal{D}}_A(s_*)$. It follows that $\Phi$ is continuous on $\mathcal{D}_A$.  
Furthermore, by item \textit{(II)} of Proposition \ref{Initial data in S(0)}, we have for all $(d_0, d_1) \in \partial \mathcal{D}_A$:
$$
(q_0, q_1)(s_0) \in \partial \hat{\mathcal{D}}_A(s_0), \qquad s_0 = -\ln T.
$$
which implies
$$
s_*(d_0, d_1) = s_0, \text{ and } \Phi(d_0, d_1) = \left(\frac{s_0^\frac{3}{2}}{A^3} \Gamma_0(d_0, d_1), \frac{s_0^2}{A}\Gamma_1(d_0, d_1)\right),
$$
where $\Gamma_0(d_0, d_1)$ and $\Gamma_1(d_0, d_1)$ are the component of $\Gamma(d_0, d_1)$, and $\Gamma$ is the mapping introduced in Proposition \ref{Initial data in S(0)}.  
Using \textit{(II)} of Proposition \ref{Initial data in S(0)}, we deduce that
$$
\deg(\Lambda|_{\partial \mathcal{D}_A})=1.
$$

Such a continuous map $\Phi$ cannot exist according to index theory, implying that case 2. is impossible.   Consequently, only case 1. remains valid, which completes the proof of Proposition \ref{existence of u in S}.  
\qed

\appendix
\section{Some usefull computations}
We give here some usefull results for our work that are easily obtained with straightforward computations. This is our first statement concerning $\varrho$.
\begin{lemma}\label{lemma equiv varrho}
We have for all $|x| \leq \epsilon_0$ that

$$
\varrho(x) = \frac{8}{K_0^2} \frac{|x|^2}{|\log |x||} \left( 1 + O \left( \frac{|\log |\log |x|||}{|\log |x||} \right) \right),
$$

and 

$$
\log \varrho(x) \sim 2 \log |x| \left( 1 + O \left( \frac{|\log |\log |x|||}{|\log |x||} \right) \right), \quad \text{as} \quad x \to 0,
$$
where $\varrho(x)$ is defined in \eqref{def rho(x), t(x)}.

\end{lemma}
\begin{proof}
The proof directly follows from $\varrho$'s definition.
\end{proof}
We have the following usefull computational result.
\begin{lemma}\label{lemma fund integral}
For all $K_0,\epsilon_0>0$ and $\delta,k,k'\in \R$, we have the following fundamental integral: 

\begin{equation}
\begin{split}
&\int_{ \frac{K_0}{4} \sqrt{(T-t) }|\log(T-t)|^{\delta}\leq |x| \leq \epsilon_0} \frac{[|x|^2]^{-k}}{ |\log |x||^{-k'}}d x\\
&\leq C(K_0,\epsilon_0,k,k')\left\{
\begin{aligned}
&|\log(T-t)|^{\max(0,k'+1)} &\text{ if } 2k=N \text{ and } k'\neq -1,\\
&|\log|\log(T-t)||&\text{ if } 2k=N \text{ and } k'= -1,\\
&1 &\text{ if } 2k<N,\\
&|\log (T-t)|^{k'+\delta(N-2k)}(T-t)^{\frac{N}{2}-k} &\text{ if } 2k>N,
\end{aligned}
\right.
\end{split}
\end{equation}

\end{lemma}
\begin{proof}
 Switching to polar coordinates, we have
\begin{equation}
\begin{split}
I:&= \int_{ \frac{K_0}{4} \sqrt{(T-t)}|\log(T-t)|^{\delta}\leq |x| \leq \epsilon_0} \frac{[|x|^2]^{-k}}{ |\log |x||^{-k'}}d x\\
&=\int_{ \frac{K_0}{4} \sqrt{(T-t)}|\log(T-t)|^{\delta}}^{\epsilon_0} \frac{|\log \xi|^{k'}}{ \xi^{2k-N+1} }d \xi.
\end{split}
\end{equation}
\begin{itemize}
\item If $2k=N$, then
\begin{equation}\label{I equal}
I\leq C(K_0,\epsilon_0,k')
\begin{cases}
|\log(T-t)|^{\max(0,k'+1)} &\text{ if } k'\neq -1,\\
|\log|\log(T-t)|| &\text{ if } k'=-1.
\end{cases}
\end{equation}

\item If $2k\neq N$, we use integration by part to have
\begin{align}
I=& \frac{|\log \sqrt{T-t}|^{k'}(1+o(1))}{ \left(2k-N\right)\left(\frac{K_0}{4}\sqrt{(T-t)|}\log (T-t)|^{\delta}\right)^{2k-N}}+f(\epsilon_0)\\
&+\frac{k'}{2k-N}\int_{ \frac{K_0}{4} \sqrt{(T-t)}|\log(T-t)|^{\delta}}^{\epsilon_0} \frac{|\log \xi|^{k'-1}}{ \xi^{2k-N+1} }d \xi,
\end{align}
where $f(\epsilon_0)=\frac{|\log \epsilon_0|^{k'}}{(N-2k)\epsilon_0^{2k-N}}$.
\item If $2k-N< 0$, then

\begin{equation}\label{I less}
\begin{split}
I\leq C(\epsilon_0,k,k').\\
\end{split}
\end{equation}
\item If $2k-N> 0$, then, we get 
\begin{equation}
I\left(1+O\left(\frac{1}{|\log\epsilon_0|}\right)\right)= \frac{\left(\frac{K_0}{4}\right)^{N-2k}}{ 2^{k'}\left(2k-N\right)}|\log (T-t)|^{k'+\delta(N-2k)}(T-t)^{\frac{N}{2}-k}(1+o(1))+f(\epsilon_0).
\end{equation}
Thus,
\begin{equation}\label{I great}
\begin{split}
I\leq C(K_0,\epsilon_0,k,k')(T-t)^{\frac{N}{2}-k}|\log (T-t)|^{k'+\delta(N-2k)}.
\end{split}
\end{equation}
\end{itemize}
We combine \eqref{I equal},\eqref{I less} and \eqref{I great}, to obtain Lemma \ref{lemma fund integral}, which concludes the proof of Lemma \ref{lemma fund integral}.
\end{proof}

We recall the fllowing Lemma from \cite{DGKZ2022}.
\begin{lemma}[Bubble integral]\label{lemma int phi}
Let us consider $N>0$, $p>1$, and $b>0$. Then,
\begin{equation}
\int_{0}^{\infty} \big(p-1 + b \xi^2\big)^{-1-\tfrac{N}{2}} \, \xi^{N-1}\, d\xi = \frac{ b^{-\tfrac{N}{2}}}{(p-1)N}.
\end{equation}
\end{lemma}
\begin{proof}
See \cite[Lemma C.2]{DGKZ2022}.
\end{proof}

The following Lemma is derived from \cite{B2011}.
\begin{lemma}\label{lemma bound nabla (Sf-f)}
Let $\Xi\subset \R^N$ be an open bounded set of class $C^\infty$ and consider $f\in C^\infty_0(\Xi)$. Then,
$$|\nabla S(t)f-\nabla f|\leq Ct \text{ for all } t\in [0,1],$$
and where $S$ is the heat semigroup.
\end{lemma}
\begin{proof}
We introduce
$$v_L=S(t)f,$$
which is a solution of the following linear equation:
\begin{align}
&\partial _tv_L=\Delta v_L,\\
&v_L(0)=f.
\end{align}
We recall the following Lemma by Brezis:
\begin{lemma}[Brezis]\label{lemma Brezis}
Let $\Xi\subset \R^N$ be an open bounded set of class $C^\infty$ and consider the following problem:
\begin{equation}
\left\{
\begin{aligned}
&\partial_t v=\Delta v \text{ in } \Xi\times (0,+\infty),\\
&v=0 \text{ on }  \partial\Xi\times (0,+\infty),\\
&v(0)=v_0\text{ on }  \Xi.\\
\end{aligned}
\right.
\end{equation}
If $v_0\in H^k(\Xi)$, for every positive integer $k$ and satisfies the following condition:
$$v_0=\Delta v_0=\dots=\Delta^j v_0=0 \text{ on } \partial\Xi,$$
for every positive integer $j$, then $v\in C^\infty(\overline{\Xi}\times [0,\infty))$.
\end{lemma}
\begin{proof}
See \cite[Theorem 10.2]{B2011}.
\end{proof}
Then, with Lemma \ref{lemma Brezis}, we have that
$$v_L\in C^\infty([0,\infty);\bar\Xi).$$
In particular, 
$$|\partial_t\nabla S(t)f|\leq C \text{ for all  } (x,t)\in \overline{\Xi}\times [0,1].$$
Therefore,
\begin{align}
|\nabla S(t)f-\nabla f|\leq \int_0^t |\partial_t\nabla S(\tau)f(x)|d \tau\leq C t.
\end{align}
This concludes the proof of Lemma \ref{lemma bound nabla (Sf-f)}.
\end{proof}

\section{Proof of Proposition \ref{Initial data in S(0)}}\label{proof of initial data in S(0)}
In this section, we provide the proof of Proposition \ref{Initial data in S(0)}, using straightforward computations. 
\begin{proof}[proof of Proposition \ref{Initial data in S(0)}]

We start with the proof of item \textit{(I)}.
\begin{itemize}[label=\textbullet]
 \item \textbf{Estimate in similarity variables:} With \eqref{theta(0)=theta_0}, we apply the transformations \eqref{def U cut-off}, \eqref{similarity variables} and \eqref{def q} as $\baru_{d_0, d_1,\theta_0}(0)\rightarrow U_{d_0, d_1}(x)\rightarrow W_{d_0, d_1}(y_0,s_0)\rightarrow q_{d_0, d_1}(y_0,s_0)$ and use \eqref{theta(0)=theta_0} and the fact that $\chi_0 \left( \frac{32 |x|}{K_0\sqrt{T|\log T|}} \right)\chi_1(x,0)=\chi_0 \left( \frac{32 |x|}{K_0\sqrt{T|\log T|}} \right)$ to obtain 
\begin{align*}
U_{d_0,d_1}(x) &= \theta^{\frac{1}{p-1}}(0) \chi_1(x,0) u_{d_0,d_1}(x) \\
& = T^{-\frac{1}{p-1}} \varphi \left( \frac{x}{\sqrt{T}}, |\log T |\right)\chi_1^2(x,0)\\
&\quad +T^{-\frac{1}{p-1}}\left(d_0 \frac{A^3}{|\log T|^{\frac{3}{2}}} + d_1\frac{A}{|\log T|^2}  \cdot \frac{x}{\sqrt{T}}\right) \chi_0 \left( \frac{32 |x|}{K_0\sqrt{T|\log T|}} \right)\\
&\quad + \theta^{\frac{1}{p-1}}_0H^*(x)(1 - \chi_1(x,0)) \chi_1(x,0).
\end{align*}
    
Then, we derive the initial data for $q$, after linearization in similarity variables as the following:

\begin{align}\label{q initial}
q_{d_0, d_1}(y_0, s_0) &= \left(d_0 \frac{A^3}{s_0^{\frac{3}{2}}} + d_1\frac{A}{s_0^2}  \cdot y_0 \right)\chi_0 \left( \frac{32 z_0}{K_0} \right) + (\chi_1^2(y_0e^{-\frac{s_0}{2}},0)-1)\varphi \left( y_0, s_0 \right) \\
&\quad +T^{\frac{1}{p-1}} \theta^{\frac{1}{p-1}}_0 H^*(y_0e^{-\frac{s_0}{2}})(1 - \chi_1(y_0e^{-\frac{s_0}{2}},0)) \chi_1(y_0e^{-\frac{s_0}{2}},0)\\
&=S_1+S_2+S_3.
\end{align}
where $y_0 = \frac{x}{\sqrt{T}}$, $s_0=-\log T$. From similar to the computations used in \cite[proof of Lemma 2.4]{MZ1997Nonlinearity}, we have that there exists $\mathcal{D}_A\subset [-2,2]^{N+1}$ such that
\begin{equation}\label{S1 in shrinking set}
(S_1)_0 \leq \frac{A^3}{s_0^{\frac{3}{2}}},\ (S_1)_1 \leq \frac{A}{s_0^{2}}, (S_0)_\perp=(S_0)_e=0.
\end{equation}
Moreover, one may see from $p>2$ that
\begin{equation}\label{supp S_2,S_3}
supp\ S_2, supp\ S_3\subset [K_0s_0^{\frac{p+1}{4}},+\infty)\subset [K_0\sqrt{s_0},+\infty),
\end{equation}
and
\begin{equation}\label{Bound S_2,S_3}
|S_2|+|S_3|\leq \frac{C}{\sqrt{s_0}},
\end{equation}
where we used Lemma \ref{existence theta0}, \eqref{def H*} and the fact that $|\log|x|^2|\leq C |\log T|$. 
We use the following lemma to conclude:
\begin{lemma}\label{lemma Q =Q_e}
For all $K_0>0$, we assume that there exist $C_0(K_0)>0,s_0>0$ such that for all $r\in L^\infty(\R^N)$, we have
$$supp\ r \subset [2K_0\sqrt{s_0},+\infty) \text{ and }\|r\|_{L^\infty(\R^N)}\leq \frac{C_0}{\sqrt{s_0}}.$$
Then, there exists $\bar C(C_0)>0$ such that for all $C>\bar C$, we have
$$r_i=0 \text{ for all } i =1,2,... \text{ and }|r_e|\leq \frac{C}{\sqrt{s_0}}.$$
\end{lemma}
\begin{proof}
The proof is straightforward from Definition of $r_i=0$ for all $i =1,2,...$ and $r_e$ in \eqref{projection}. One just need to notice that $|r_e|\leq |r|$.
\end{proof}
Using Lemma \ref{lemma Q =Q_e} on \eqref{supp S_2,S_3} and \eqref{Bound S_2,S_3}, it follows that there exits $C_0>0$, for all $C\geq C_0$, we have
$$(S_j)_i=0 \text{ for all } i =1,2,... \text{ and }|(S_j)_e|\leq \frac{C_0}{\sqrt{s_0}},$$
for $j=2,3$. Together with \eqref{S1 in shrinking set}, we obtain the estimates in similarity variables of Proposition \ref{Initial data in S(0)}.

\item Estimate in $P_2(0)$: Let us consider $|x| \in \left[ \frac{K_0}{4} \sqrt{T |\log T|}, \epsilon_0 \right]$, we define
$$
\tau_0(x) = \frac{- t(x)}{\varrho(x)} \in [0, 1],
$$
where $t(x)$ and $\varrho(x)$ are given in \eqref{def rho(x), t(x)}. From \eqref{def tilde theta} and \eqref{theta(0)=theta_0} we have for $|\xi| \leq 2\alpha_0\sqrt{|\log(\varrho(x))|}$, the following
\begin{equation}\label{theta (t(x))=(tau0)=(0)}
\tilde{\theta}(\tau_0(x))=\theta(0)=\theta_0.
\end{equation}
Therefore, applying Definition \eqref{def mathcal U} on \eqref{initial data}, we have
$$ \mathcal{U}(x, \xi, \tau_0(x)) =\varrho(x)^{\frac{1}{p-1}} \theta(t(x))^{\frac{1}{p-1}}  u_{d_0,d_1,\theta_0}(x+\xi\sqrt{\varrho(x)}).$$
We claim the following Lemma:
\begin{lemma}\label{Lemma bound mathcal V-hat V}
For all $\delta_0,\alpha_0,\hat C>0$, $|\xi|\leq 2\alpha_0\sqrt{\varrho(x)}$ and $\tau \in [0,1)$, if
\begin{equation}\label{bound hat V, U, hat U}
|\hat{\mathcal{V}}(\tau)|\leq \hat C \text{ and }|\mathcal{U}(\tau)-\hat{\mathcal{U}}(\xi,\tau)|\leq \frac{\delta_0}{\hat C(\hat C+\delta_0)},    
\end{equation}
then
\begin{equation}\label{bound V -hat V}
\left|\mathcal{V}(\tau)-\hat{\mathcal{V}}(\xi,\tau)\right|\leq \delta_0,
\end{equation}
where $\varrho(x)$, $\mathcal{V}$ and $\hat{\mathcal{V}}$ are defined in \eqref{def rho(x), t(x)}, \eqref{def mathcal V} and \eqref{def V hat}, and $\hat{\mathcal{U}}$ is given from $\hat{\mathcal{V}}$ by \eqref{def mathcal V}, i.e.
{\mathtoolsset{showonlyrefs=false}
\begin{align}\label{def hat mathcal U}
\hat{\mathcal{U}}=\frac{1}{\hat{\mathcal{V}}}-\theta(t(x))^{{\frac{1}{p-1}}}\varrho(x)^{\frac{1}{p-1}}.
\end{align}
}
\end{lemma}
\begin{proof}
We have from \eqref{def mathcal V} and \eqref{def V hat} that
\begin{align}\label{equality U, hat U, V, hat V}
    \left|\mathcal{V}-\hat{\mathcal{V}}\right|&=\left|\mathcal{U}-\hat{\mathcal{U}}\right| \Big|\mathcal{V}\Big|\left|\hat{\mathcal{V}}\right|\\
    &\leq \left|\mathcal{U}-\hat{\mathcal{U}}\right| \Big|\mathcal{V}-\hat{\mathcal{V}}\Big|\left|\hat{\mathcal{V}}\right|+\left|\mathcal{U}-\hat{\mathcal{U}}\right|\left|\hat{\mathcal{V}}\right|^2.
\end{align}
Applying \eqref{bound hat V, U, hat U}, we obtain \eqref{bound V -hat V}. This concludes the proof of Lemma \ref{Lemma bound mathcal V-hat V}
\end{proof}
It remains to prove \eqref{bound hat V, U, hat U} for $\tau=\tau_0$. Let's prove the first inequality of \eqref{bound hat V, U, hat U}. We recall from \eqref{def V hat} and \eqref{def theta}  that
\begin{equation}\label{computation bound mathcal hat V}
\begin{split}
\hat{\mathcal{V}}(\tau_0)&=\left[(p-1)\left(1-\frac{1}{(T-t(x))\theta(t(x))}\int^{0}_{t(x)}\theta(\sigma)d\sigma \right)  + b \frac{K_0^2}{16}\right]^{\frac{1}{p-1}}\\
&=\left[(p-1)\left(1-\frac{\theta_0|\log T|^\beta}{(T-t(x))\theta(t(x))}\int^{0}_{t(x)}\frac{1}{|\log (T-\sigma)|^\beta}d\sigma \right)  + b \frac{K_0^2}{16}\right]^{\frac{1}{p-1}}\\
&= \left[(p-1)\left(1-\frac{|\log (T-t(x))|^\beta}{(T-t(x))}\int^{0}_{t(x)}\frac{1}{|\log (T-\sigma)|^\beta}d\sigma \right)  + b \frac{K_0^2}{16}\right]^{\frac{1}{p-1}}.
\end{split}
\end{equation}
From straightforward computations, we have
\begin{align}
    \int^{0}_{t(x)}\frac{d\sigma}{|\log (T-\sigma)|^\beta} &=\int^{T-t(x)}_{T}\frac{d\sigma}{|\log \sigma|^\beta} \\
    &\underset{\epsilon_0\to 0}{\sim} \frac{T-t(x)}{|\log (T-t(x))|^\beta}-\frac{T}{|\log T|^\beta}
    \underset{T\to 0}{\sim} \frac{T-t(x)}{|\log (T-t(x))|^\beta}.
\end{align}
Then, with \eqref{computation bound mathcal hat V} and taking $\epsilon_0$ then $T$ small, we get
\begin{equation}\label{hat mathcal V tau0}
\begin{split}
\hat{\mathcal{V}}(\tau_0)\sim \left[b \frac{K_0^2}{16}\right]^{\frac{1}{p-1}}.
\end{split}
\end{equation}

To prove the second inequality, 
we denote by 
$$\bar{\mathcal{U}}(\tau_0):=\hat{\mathcal{V}}(\tau_0)^{-1}.$$
Accordingly, from \eqref{def hat mathcal U}, the proof of the second inequality in \eqref{bound hat V, U, hat U} reduces, for $x$ sufficiently small, to establishing the following:
\begin{equation}\label{bound U, bar U}
|\mathcal{U}(\tau_0)-\bar{\mathcal{U}}(\xi,\tau_0)|\leq \frac{\delta_0}{2\hat C(\hat C+\delta_0)},
\end{equation}
thanks to \eqref{asym rho}. From \eqref{def mathcal U}, \eqref{initial data}, it sufficient to prove, for $X=x+\xi\sqrt{\varrho(x)}$, the following two estimates:\\
- For all $|x| \in \left[ \frac{K_0}{4}\sqrt{T|\log T|}, 2K_0\sqrt{T}|\log T|^\frac{p+1}{4} \right]$ and $|\xi| \leq 2\alpha_0\sqrt{|\log(\varrho(x))|}$, we have

\begin{equation}\label{bound phib- bar mathcal U}
\left|  \left(\frac{\varrho(x)}{T}\right)^{\frac{1}{p-1}} \left(\frac{\theta(t(x))}{\theta_0}\right)^\frac{1}{p-1}  \varphi_b \left( X \right)- \bar{\mathcal{U}}(x, \tau_0(x)) \right| \leq \frac{\delta_0}{4\hat C(\hat C+\delta_0)}.
\qquad 
\end{equation}

- For all $|x| \in \left[K_0\sqrt{T}|\log T|^\frac{p+1}{4} , \epsilon_0  \right]$ and $|\xi| \leq 2\alpha_0|\sqrt{\log(\varrho(x))|}$, we have

\begin{equation}\label{bound H*- bar mathcal U}
\left|  \varrho(x)^{\frac{1}{p-1}}\theta(t(x))^\frac{1}{p-1} H^*\left( X \right)- \bar{\mathcal{U}}(x, \tau_0(x)) \right| \leq \frac{\delta_0}{4\hat C(\hat C+\delta_0)}.
\end{equation}

For \eqref{bound phib- bar mathcal U}, notice that since  $|x| \in \left[ \frac{K_0}{4}\sqrt{T|\log T|}, 2K_0\sqrt{T}|\log T|^\frac{p+1}{4} \right]$, then
$$\log \varrho(x)\sim \log T,$$
thanks to Lemma \ref{lemma equiv varrho}. Thus, from \eqref{def theta}, we have
$$\frac{\theta(t(x))}{\theta_0}\sim 1, \text{ as } \epsilon_0,T\to0.$$
Then, with \eqref{hat mathcal V tau0}, it is enough to show that
\begin{equation}
\left|  \left(\frac{\varrho(x)}{T}\right)^{\frac{1}{p-1}}  \varphi_b \left( X \right)-  \left[b \frac{K_0^2}{16}\right]^{\frac{1}{p-1}} \right| \leq \frac{\delta_0}{8\hat C(\hat C+\delta_0)}.
\qquad 
\end{equation}

We mention that, our functions $\left(\frac{\varrho(x)}{T}\right)^{\frac{1}{p-1}}  \varphi_b \left( X \right)$  is similar to the one in \cite[proof of (A.4) and (A.5)]{DZ2019}. Therefore, we apply the same computations. For \eqref{bound H*- bar mathcal U}, we write from \eqref{def H*}, the following:
\begin{align}\label{detail H*-bat mathcal U}
    &\left|  \varrho(x)^{\frac{1}{p-1}}\theta(t(x))^\frac{1}{p-1} H^*\left( X \right)- \bar{\mathcal{U}}(x, \tau_0(x)) \right| \\
    &=\left|  \left[ b\frac{|X|^2}{\varrho(x)2|\log |X||}\right]^{-\frac{1}{p-1}}  \left[\frac{\theta(t(x)) (2|\log|X||)^\beta}{\theta_\infty}\right]^{\frac{1}{p-1}}- \left[\frac{T}{\varrho(x)}(p-1)+b\frac{K_0^2}{16}\right]^{-\frac{1}{p-1}} \right|
\end{align}

One may see from the definition of $X$ that
\begin{equation}\label{bounds X by x 2}
    |x|(1-8\frac{\alpha_0}{K_0})\leq|X|\leq |x|(1+8\frac{\alpha_0}{K_0}).
\end{equation}
Then, when $\alpha_0$ is small enough in terms of $K_0$, we get
\begin{align}\label{bounds X by x 3}
    &|X|\sim |x|,\ |\log|X||\sim |\log|x||.
\end{align}
Together with Lemma \ref{existence theta0}, \eqref{def theta} and Lemma \ref{lemma equiv varrho}, we have that
\begin{align}\label{bound X2/rho logX}
     \frac{|X|^2}{\varrho(x)2|\log |X||}\sim \frac{K_0^2}{16} \text{ and } \frac{\theta(t(x)) (2|\log|X||)^\beta}{\theta_\infty}\sim 1\text{ as } x \to 0.
\end{align}
Moverover, we recall that $\varrho(x)$ is increasing in $|x|$. Thus
\begin{equation}
    0\leq \frac{T}{\varrho(x)}\leq \frac{T}{\varrho(K_0\sqrt{T}|\log T|^{\frac{p+1}{4}})}\leq \frac{C(K_0)}{|\log T|^{\frac{p-1}{2}}}\rightarrow 0 \text{ as } T\rightarrow0.
\end{equation}
Combining this with \eqref{detail H*-bat mathcal U} and \eqref{bound X2/rho logX}, we obtain \eqref{bound H*- bar mathcal U}. Thus, the second inequality of \eqref{bound hat V, U, hat U} follows. Hence, we obtain \eqref{bound V -hat V}. Furthermore, the method presented in \cite[proof of (ii) of Lemma A.1]{DZ2019} combined with Lemma \ref{existence theta0} and \eqref{log rho < log T} can be utilized to demonstrate the following:
$$|\nabla_\xi \mathcal{U}|\leq \frac{C}{|\log \varrho (x)|},$$
which concludes the proof of (I) in Proposition \ref{Initial data in S(0)}.
\end{itemize}

Now, for (II), one may see from \eqref{q initial} that
\begin{align}
(q_{d_0, d_1})_b = \left(d_0 \frac{A^3}{s_0^{\frac{3}{2}}} + d_1\frac{A}{s_0^2}  \cdot y_0 \right)\chi_0 \left( \frac{32 z_0}{K_0} \right),
\end{align}
for $T$ sufficiently small. Thus, with similar computations as for \cite{TZ2019}[Proposition 4.5], we obtain (II) in Proposition \ref{Initial data in S(0)}, which concludes the proof of Proposition \ref{Initial data in S(0)}.
\end{proof}

\section{Some bounds on terms of equation (\ref{eq q})}
Here, we give some estimates of terms of equation (\ref{eq q}). We have first the following result on the potential.
\begin{lemma}[Estimate on potential $V$]\label{lemma estimate V}
Consider $V$ as defined in \eqref{eq q}. Then, for all $K_0>0$, we have
    \begin{align}\label{bound V}
    V(y, s)=&-\frac{bp}{(p-1)^2s}(|y|^2-2N)+\frac{1}{s}\left(\frac{ap}{\kappa}-\frac{2Nbp}{(p-1)^2}\right)\\
    &-abp(p-2)\kappa^{2p-3}\frac{|y|^2}{s^2}+O\left(\frac{1+|y|^4}{s^3}\right), \text{ for all } \frac{|y|}{\sqrt{s}}\leq 2K_0.
\end{align}
\end{lemma}
\begin{proof}
We assume $\frac{|y|}{\sqrt{s}}\leq 1$. Then, the proof follows directly from the definition of $V$ in \eqref{eq q}, Corollary \ref{cor theta'<0} and Taylor expansion. Now, for $1\leq\frac{|y|}{\sqrt{s}}\leq 2K_0$, it follows from the definition of $V$ in \eqref{eq q} the following
$$
\left|V(y, s)+c_0\frac{|y|^2}{s}+\frac{c_1}{s}+c_3\frac{|y|^2}{s^2}\right|\leq C(K_0) \leq C(K_0)\frac{|y|^4}{s^2},
$$
for all $c_0,c_1$ and $c_2\in \R$.
\end{proof}
Now, we give the following estimates on the gradient term.
\begin{lemma}[Estimate on radient term $J$]\label{lemma estimate J}
Under the same assumptions, quantifiers, and notation as in Corollary \ref{cor theta'<0}, there exists $T_6\leq T_5$, for all $T\in (0,T_6]$, such that the following hold: we consider $u$ solution of \eqref{eq u}, with $u(0)=u_{d_0,d_1,\theta_0}$ defined in \eqref{initial data} for some $(d_0, d_1) \in [-1,1]^{1+N}$ and $\theta_0$ given by Lemma \ref{existence theta0}. We assume $u \in \mathcal{S}(T, K_0, \epsilon_0, \alpha_0, A, \delta_0, C_0, \eta_0, t)$ for all $t \in [0,t_6]$ for some $t_6<T$. Then, we have
\begin{equation}\label{bound J interior}
\begin{split}
&|\chi J(q, \theta)| \leq C(K_0, A) \left( \frac{|q|}{s} + \frac{|\nabla q|}{\sqrt{s}} \right),\\
&\left|\chi\left( J(q,\theta) + 4 \frac{\nabla \bar \varphi}{ \bar \varphi} \cdot \nabla q\right)\right| 
\leq C(K_0, A)\left(\frac{|y|^2}{s^2}|q| + \frac{q^2}{s} + |\nabla q|^2\right).
\end{split}
\end{equation}
Moreover,
\begin{align}\label{bound J exterior}
\left|(1-\chi)J(q,\theta) \right| \leq \frac{C(K_0)}{s}.
\end{align}
In particular, we have
\begin{equation}\label{cor bound J}
\begin{split}
&|\chi J(q,\theta)|\leq \frac{C(K_0,A)}{s},\\
&\left|\chi J(q,\theta) \right|\leq \frac{C(K_0,A)(1+|y|^3)}{s^{\frac{5}{2}}},\\
\end{split}
\end{equation}
and 
\begin{equation}\label{bound J 0,1,2}
\begin{split}
&|J(q,\theta)_0|+|J(q,\theta)_1|\leq \frac{C(A)}{s^{\frac{5}{2}}},\\
&|J(q,\theta)_2-\frac{16b}{(p-1)^2s}q_2|\leq \frac{C(A)}{s^3},\\
\end{split}
\end{equation}
where $\chi$ and $J$ are defined in \eqref{def chi} and \eqref{eq q} respectively. 
\end{lemma}

\begin{proof}
\textbf{Proof of \eqref{bound J interior}:} It follows the same reasoning as in \cite[Lemma D.3]{DZ2019}, relying on the following estimates:
$$|\nabla \bar\varphi|\leq \frac{C}{\sqrt{s}}, \frac{|\nabla \bar\varphi|^2}{\bar\varphi^2}\leq C\frac{|y|^2}{s^2},$$
with \eqref{estimate q 1}, \eqref{estimate grad q 3}, Corollaries \ref{cor theta'<0} and \ref {Growth estimates 2} and Lemmas \ref{Growth estimates} and \ref{Growth estimates 2}.

\textbf{Proof of \eqref{bound J exterior}:} We have from \eqref{def phi} and \eqref{expression phi profile}, the following
$$\frac{|\nabla \bar\varphi|^2}{\bar\varphi}\leq \frac{|\nabla \varphi_b|^2}{\varphi_b}\leq \frac{C}{s},$$
on one hand. On the other hand, we have from \eqref{def q}, \eqref{similarity variables} and \eqref{def U cut-off} that
$$ \frac{|\nabla q+\nabla \bar \varphi|^2}{|q+\bar\varphi|}=(T-t)^{\frac{p}{p-1}}\theta(t)^{\frac{1}{p-1}}\frac{|\nabla(\chi_1u)|^2}{|\chi_1u|}.$$
Therefore, it is enough to show that for all $ |x|\in [K_0\sqrt{(T-t)\log(T - t)|},2K_0\sqrt{(T - t)} |\log(T - t)|^{\frac{p+1}{4}}]$
$$\frac{|\nabla(\chi_1u)|^2}{|\chi_1u|}\leq C(K_0)\theta(t)^{-\frac{1}{p-1}}(T-t)^{-\frac{p}{p-1}}|\log(T-t)|^{-1},$$
which follows from Corollary \ref{cor theta'<0}, and more precisely from \ref{estimate until in P2}, \eqref{expression theta' rigorous} and \eqref{t(x)<t}. This concludes the proof of \eqref{bound J exterior}.

\textbf{Proof of \eqref{cor bound J}:} From \eqref{bound J interior}, \eqref{estimate q 1}, \eqref{estimate q 2}, \eqref{estimate grad q 2}, \eqref{estimate grad q 3}, and Corollary \ref{Growth estimates 2}, we have that 
\begin{align}\label{J+ nabla phi/phi q}
\left|\chi\left( J(q,\theta) + 4 \frac{\nabla \bar \varphi}{ \bar \varphi} \cdot \nabla q\right)\right| 
\leq \frac{C(K_0,A)(1+|y|^3)}{s^\frac{5}{2}}
\end{align}
We have with Taylor expansion that 
\begin{equation}\label{grad q/q}
\frac{\nabla \bar \varphi}{\bar \varphi}=-\frac{2b}{(p-1)^2}\frac{y}{s}\left(1+O\left(\frac{1+|y|^2}{s}\right)\right), \text{ for all } |y|\leq 2K_0\sqrt{s}.
\end{equation}
Moreover, with \eqref{estimate grad q 2}, we have that
\begin{equation}\label{estimate grad q 4}
|(\nabla q)_b(y, s)| \leq \frac{C(K_0)A^6(1+|y|)}{s}.
\end{equation}
Thus, we obtain the second estimate of \eqref{cor bound J}. The first estimate follows directly from \eqref{bound J interior} and Lemma \ref{Growth estimates}.

\textbf{Proof of \eqref{bound J 0,1,2}:} The first estimate in \eqref{bound J 0,1,2} follows from \eqref{J+ nabla phi/phi q}, \eqref{grad q/q} and \eqref{estimate grad q 2} and Corollary \ref{Growth estimates 2}. For the second one, we combine \eqref{J+ nabla phi/phi q}, \eqref{grad q/q}, \eqref{decomp nabla q b} and the bound for $(\nabla q)_\perp$ in Definition \ref{def shrinking set S}. This concludes the proof of Lemma \ref{lemma estimate J}.

\end{proof}

Here, we have the following bounds on the quadratic term.
\begin{lemma}[Estimate on quadratic term $B$]\label{lemma estimate B}
Under the same assumptions, quantifiers, and notation as in Corollary \ref{cor theta'<0}, we consider $B(q,\theta)$ defined in \eqref{eq q}. Then, we have
\begin{equation}\label{bound B}
    |B(q,\theta)|\leq C(K_0) q^2.
\end{equation}
Moreover,
\begin{align}\label{cor bound B in s}
&|B(q)|\leq \frac{C(K_0,A)}{s},\\ 
&|B(q)|\leq \frac{C(K_0)A^{14}(1+|y|^3)}{s^{5/2}},
\end{align}

\end{lemma}
\begin{proof}
The proof of \eqref{bound B} follows directly from the definitions of $B$, \eqref{estimate q 1}, Corollary \ref{cor theta'<0} and a Taylor expansion.

The proof of the first estimate in \eqref{cor bound B in s} follows from \eqref{bound B} and \eqref{estimate q 1}. For the second estimate in \eqref{cor bound B in s}, we write 
$$q=Q_1+Q_2,$$
where $Q_1=\underset{i=0}{\overset{2}{\sum}}h_iq_i$ and $Q_2=q_-+q_e$. From \eqref{bound B} and Definition \ref{def shrinking set S}, we have
\begin{align}
|B(q)|&\leq C (Q_1^2+Q_2^2)=C (Q_1^\frac{1}{2}Q_1^\frac{3}{2}+Q_2 Q_2)\\
&\leq C\left(\frac{A^2}{s^{\frac{1}{4}}}\frac{A^6(1+|y|^3)}{s^{9/4}}+\frac{A^{7}}{\sqrt{s}}\frac{A^7(1+|y|^3)}{s^{2}}\right)\\
&\leq C\frac{A^{14}(1+|y|^3)}{s^{5/2}}.
\end{align}
This concludes the proof of Lemma \ref{lemma estimate B}.
\end{proof}
In the following, we give some estimates on the rest term.
\begin{lemma}[Estimate on rest term $R$]\label{lemma estimate R}
Under the same assumptions, quantifiers, and notation as in Corollary \ref{cor theta'<0}, there exists $T_7\leq T_5$, for all $T\in (0,T_7]$, such that the following hold: we consider $u$ solution of \eqref{eq u}, with $u(0)=u_{d_0,d_1,\theta_0}$ defined in \eqref{initial data} for some $(d_0, d_1) \in [-1,1]^{1+N}$ and $\theta_0$ given by Lemma \ref{existence theta0}. We assume $u \in \mathcal{S}(T, K_0, \epsilon_0, \alpha_0, A, \delta_0, C_0, \eta_0, t)$ for all $t \in [0,t_7]$ for some $t_7<T$. Then, for all $s\in [-\log T,-\log (T-t_7)]$, we have
\begin{align}\label{DL R}
    R=&\frac{1}{s}\left(a-\frac{2bN\kappa}{(p-1)^2}\right)+\frac{1}{s^2}(a+\frac{pa^2}{2\kappa})\\
    &+\frac{|y|^2}{s^2}\left[  \frac{bp}{(p-1)^2} \left( \frac{2bNk}{(p-1)^2} - a \right) + \frac{\kappa b}{(p-1)^2} \left( \frac{4(p-2)b}{(p-1)^2} - 1\right)\right]\\
    &+O\left(\frac{1+|y|^4}{s^3}\right), \text{ for all } \frac{|y|}{\sqrt{s}}\leq 2K_0.
\end{align}
Moreover,
\begin{align}\label{bound R}
    |R(y, s)|\leq \frac{C}{s},
\end{align}
\begin{align}\label{bound nabla R}
    |\nabla R(y, s)|\leq \frac{C(1+|y|^3)}{s^2}.
\end{align}
In particular,
\begin{align}\label{cor bound R}
&\left|R-\frac{\beta\kappa}{(p-1)s}\right|\leq \frac{C(K_0)(1+|y|^2)}{s^2}, \text{ for all } \frac{|y|}{\sqrt{s}}\leq 2K_0.
\end{align}
where $R$ is defined in \eqref{eq q}.
\end{lemma}
\begin{proof}
    We assume that $z=\frac{y}{\sqrt{s}}$. Then,
\begin{align}
R(y, s) =& -\frac{b |z|^2}{(p-1)s} \varphi_b^p(z) + \frac{a}{s^2} - \frac{2bN}{(p-1)s} \varphi_b^p(z) + \frac{4pb^2 |z|^2}{(p-1)^2 s} \varphi_b^{2p-1}(z)
- \frac{1}{2}z\nabla_z\varphi_b(z)\\
&- \frac{\varphi_b(z)}{p-1} - \frac{a}{(p-1)s} + \bar\varphi^p(z) - \frac{8b^2 |z|^2}{(p-1)^2 s} \frac{\varphi_b^{2p}(z)}{\bar\varphi(y,s)}.
\end{align}
We recall that $\varphi_b$ is the solution of the following equation

$$-\frac{1}{2}z\nabla_z\varphi_b(z)-\frac{\varphi_b(z)}{p-1}+\varphi_b(z)^p=0.$$
Then,

\begin{align}\label{R proof computation}
R(y, s) =& -\frac{b |z|^2}{(p-1)s} \varphi_b^p(z) + \frac{a}{s^2} - \frac{2bN}{(p-1)s} \varphi_b^p(z) + \frac{4pb^2 |z|^2}{(p-1)^2 s} \varphi_b^{2p-1}(z)
- \varphi_b^p(z)\\
&-  \frac{a}{(p-1)s} + \bar\varphi^p(z) - \frac{8b^2 |z|^2}{(p-1)^2 s} \frac{\varphi_b^{2p}(z)}{\bar\varphi(y,s)}.
\end{align}
We have from Corollary \ref{cor theta'<0} that
$$|\bar\varphi^p -\varphi^p_b |\leq \frac{C}{s}.$$
Together with $|z|^2\varphi_b^{p-1}+\varphi_b\leq C$ and 
\begin{equation}\label{phi2p/phi leq phi 2p-1}
    \frac{\varphi_b^{2p}}{\bar\varphi}\leq \varphi_b^{2p-1},
\end{equation}
we obtain \eqref{bound R}. Again from Corollary \ref{cor theta'<0} that
$$\left|\nabla_y\left(\bar\varphi^p(y,s) -\varphi^p_b\left(\frac{y}{\sqrt{s}}\right) \right)\right|\leq C\frac{|y|}{s^2}.$$
Then, with \eqref{R proof computation} and \eqref{phi2p/phi leq phi 2p-1}, we obtain \eqref{bound nabla R}. Using Taylor expansion, we have 
$$\bar \varphi^p(y,s) = \varphi_b^p(z) + \frac{pa}{s} \varphi_b^{p-1} (z)+ \frac{p(p-1)a^2}{2s^2} \varphi_b^{p-2}(z)+ O(s^{-3}) \text{ for all } |z|=\frac{|y|}{\sqrt{s}}<K_0,$$
where we used Corollary \ref{cor theta'<0}. 
Thus,
\begin{align}\label{dev R expression}
R(y, s) =& -\frac{b |z|^2}{(p-1)s} \varphi_b^p(z) + \frac{a}{s^2} - \frac{2bN}{(p-1)s} \varphi_b^p(z) + \frac{4pb^2 |z|^2}{(p-1)^2 s} \varphi_b^{2p-1}(z) - \frac{a}{(p-1)s}\\
& +\frac{pa}{s} \varphi_b^{p-1} (z)+ \frac{p(p-1)a^2}{2s^2} \varphi_b^{p-2}(z) - \frac{8b^2 |z|^2}{(p-1)^2 s} \frac{\varphi_b^{2p}(z)}{\bar\varphi(y,s)}+O(s^{-3}) .
\end{align}
We assume $|z|\leq 1$. Then, with Taylor Expansion and Corollary \ref{cor theta'<0}, we obtain
\begin{align}
&\varphi^p_b(z) = \kappa^p - \frac{pb\kappa}{(p-1)^3}|z|^2+ O(|z|^4),\\
&\varphi_b^{2p-1}(z) = \kappa^{2p-1}+O(|z|^2),  \\
& \varphi_b^{p-1}(z) = \frac{1}{p-1} - \frac{b}{(p-1)^2}|z|^2 + O( |z|^4),\\
&\varphi_b^{p-2}(z) = \kappa^{p-2}+O(|z|^2), \\\
&\text{and } \frac{\varphi_b^{2p}(z)}{\bar\varphi(y,s)} =\varphi_b^{2p-1}(z)+O\left(\frac{1}{s}\right).
\\   
\end{align}
Then, combining this later with \eqref{dev R expression}, we obtain \eqref{DL R} for $|\frac{y}{\sqrt{s}}|\leq 1$. We follow the same reasoning as for the proof of \eqref{bound V} for $1\leq |\frac{y}{\sqrt{s}}|\leq 2K_0$. 

For \eqref{cor bound R}, the proof follows from \eqref{DL R}, \eqref{def a} and \eqref{expression b}. This concludes the proof of Lemma \ref{lemma estimate R}.
\end{proof}
Now, we give estimates on the non local term.
\begin{lemma}[Estimate on the non local term $N$]\label{lemma estimate N}
Under the same assumptions, quantifiers, and notation as in Corollary \ref{cor theta'<0}, there exists $T_8\in (0,T_5]$, for all $T\in (0,T_8]$, such that the following hold: we consider $u$ solution of \eqref{eq u}, with $u(0)=u_{d_0,d_1,\theta_0}$ defined in \eqref{initial data} for some $(d_0, d_1) \in [-1,1]^{1+N}$ and $\theta_0$ given by Lemma \ref{existence theta0}. We assume $u \in \mathcal{S}(T, K_0, \epsilon_0, \alpha_0, A, \delta_0, C_0, \eta_0, t)$ for all $t \in [0,t_8]$ for some $t_8<T$. Then, for all $s\in [-\log T,-\log (T-t_8)]$, we have
\begin{equation}\label{DL N}
    N(q,\theta)=\frac{1}{p-1}\frac{\theta^{\prime}(s)}{\theta(s)}\left(\kappa+\frac{a}{s}-\frac{b\kappa}{(p-1)^2} \frac{|y|^2}{s}+q+O\left(\frac{|y|^4}{s^2}\right)\right), \text{ for all } \frac{|y|}{\sqrt{s}}\leq 2K_0.
\end{equation}
Moreover,
\begin{equation}\label{Bound N}
    |N(q,\theta)|\leq \frac{C(K_0)}{s},
\end{equation}
\begin{equation}\label{Bound nabla N}
    |\nabla N(q,\theta)|\leq C(K_0)\frac{(1+|y|^3)}{s^2},
\end{equation}
where $a$ is given in \eqref{def a}. In particular,
\begin{align}\label{estimate proj N}
&|\tilde{N}(q)_0| \leq \frac{C}{s}, \quad |\tilde{N}(q)_1| \leq \frac{C}{s^{2}},   |\tilde{N}(q)_2| \leq \frac{C }{s^2}, \\
&|\tilde{N}(q)_-(y)| \leq \frac{C(1 + |y|^3)}{s^{2} }, \|\tilde{N}(q)_{\text{e}}\|_{L^\infty(\mathbb{R})} \leq \frac{C}{s},\\
\end{align}
where $N$ given in \eqref{eq q} and $\tilde{N}(q)$ is given right after \eqref{bound R + N}.
\end{lemma}
\begin{proof}
The proof of \eqref{DL N} follows from \eqref{expression theta rigorous} and Taylor expansion for $|\frac{y}{\sqrt{s}}|\leq 1$ and the same reasoning as for the proof of \eqref{bound V} for $1\leq |\frac{y}{\sqrt{s}}|\leq 2K_0$.

For \eqref{Bound N}, it is direclty obtained using the boundedness of $\varphi_b$ (see \eqref{expression phi profile}), Corollary \ref{cor theta'<0} and \eqref{estimate q 1}.

The proof of \eqref{Bound nabla N} follows directly from Definition of $N$ in \eqref{eq q}, Corollary \ref{cor theta'<0}, \eqref{global estimate grad q} and \eqref{estimate grad q 2}. 
    
For \eqref{estimate proj N}, it follows directly from \eqref{DL N}, \eqref{Bound N}, Corollary  \ref{cor theta'<0}, Definition  \ref{def shrinking set S} and the following:
\begin{align}
&\tilde{N}(q)_0=N(q)_0+\frac{\beta\kappa}{(p-1)s}, \quad \tilde{N}(q)_1=N(q)_1,   \tilde{N}(q)_2=N(q)_2, \\
&\tilde{N}(q)_-=N(q)_-,\  \|\tilde{N}(q)_{\text{e}}\|_{L^\infty(\mathbb{R})} \leq \|N(q)\|_{L^\infty(\mathbb{R})} ,\\
\end{align}
which concludes the proof of Lemma \ref{lemma estimate N}.
\end{proof}
Finally, we give an estimate on the cut-off term.
\begin{lemma}[Estimate on cut-off term $\bar F$]\label{Estimate on F}
Under the same assumptions, quantifiers, and notation as in Corollary \ref{cor theta'<0}, there exists $T_{11}\in (0,T_5]$, for all $T\in (0,T_9]$, such that the following hold: we consider $u$ solution of \eqref{eq u}, with $u(0)=u_{d_0,d_1,\theta_0}$ defined in \eqref{initial data} for some $(d_0, d_1) \in [-1,1]^{1+N}$ and $\theta_0$ given by Lemma \ref{existence theta0}. We assume $u \in \mathcal{S}(T, K_0, \epsilon_0, \alpha_0, A, \delta_0, C_0, \eta_0, t)$ for all $t \in [0,t_9]$ for some $t_9<T$. Then, for all $s\in [-\log T,-\log (T-t_9)]$, we have
\begin{equation}\label{bound F}
    supp (\bar F)=\left[K_0 s^{\frac{p+1}{4}}, 2K_0 s^{\frac{p+1}{4}}\right],\ |\bar F|\leq \frac{C}{\sqrt{s}},
\end{equation}
where $\bar F$ is given in \eqref{def F}. 
\end{lemma}
\begin{proof}
One may from \eqref{def tilde F} that $F$ is not identically zero only when\\ $|x|\in \left[K_0 \sqrt{T - t} |\log(T - t)|^{\frac{p+1}{4}}, 2K_0 \sqrt{T - t} |\log(T - t)|^{\frac{p+1}{4}}\right]$. Then, with straightforward computations, we have the following bounds on $\chi_1$ defined in \eqref{def chi1}.
\begin{align}
    &|\partial_t \chi_1| \leq C(K_0) (T-t)^{-1}, \quad |\nabla \chi_1|\leq C(K_0)(T - t)^{-\frac{1}{2}} |\log(T - t)|^{-\frac{p+1}{4}},\\
    &|\Delta \chi_1|\leq C(K_0)(T - t)^{-1} |\log(T - t)|^{-\frac{p+1}{2}}.
\end{align}
Thus, using Corollary \ref{cor theta'<0}, we obtain the following
\begin{align}
    &|\theta^\frac{1}{p-1}\partial_t \chi_1\baru|\leq C(K_0) (T-t)^{-\frac{p}{p-1}}|\log(T - t)|^{-\frac{1}{2}},\\
    &|\theta^\frac{1}{p-1}\Delta \chi_1\baru|\leq C(K_0)(T - t)^{-\frac{p}{p-1}} |\log(T - t)|^{-\frac{1}{2}-\frac{p+1}{2}},\\
    &|\theta^\frac{1}{p-1}\nabla\chi_1\nabla \baru|\leq C(K_0)(T - t)^{-\frac{p}{p-1}}|\log(T - t)|^{-\frac{1}{2}-\frac{p+1}{2}},\\
    &\left|\theta^\frac{1}{p-1}\left(\frac{|\nabla (\chi_1 u)|^2}{\chi_1 u+1}+\chi_1\frac{|\nabla\baru|^2}{\baru+1}\right)\right|\leq C(K_0)(T - t)^{-\frac{p}{p-1}} |\log(T - t)|^{-\frac{1}{2}-\frac{p+1}{2}},\\
    &|\theta^\frac{p}{p-1}((1+\chi_1 u)^p+\chi_1(1+\baru)^p)|\leq C(K_0)(T - t)^{-\frac{p}{p-1}} |\log(T - t)|^{-\frac{p}{2}}.
\end{align}
Combining this with \eqref{def tilde F} and \eqref{def F}, we obtain \eqref{bound F}. This concludes the proof of Lemma \ref{Estimate on F}.
\end{proof}

\section{Estimates on linearized operators}
Here, we present the following result, drawing inspiration from Bricmont and Kupiainen \cite{BK1994} and Merle and Zaag \cite{MZ1997Nonlinearity}, which describes the dynamics associated with the linear operators  $K(s,\sigma)$ and $K_1(s,\sigma)$.
\begin{lemma}[Bricmont and Kupiainen]\label{lemma psi}
We denote by $K(s,\sigma)$ as the fundamental solution associated to the linear operator $\mathcal{L} + \mathcal{V}$. Then, we have the following result: For all $l^* > 0$, there exists $\sigma = \sigma(l^*)>0$ such that if $s \geq \sigma$ and $v(s)$ satisfies
$$
\sum_{m=0}^2 |v_m(s)| 
+ \left\| \frac{v_{-}(y,s)}{1+|y|^3} \right\|_{L^\infty} 
+ \|v_{e}(s)\|_{L^\infty} < +\infty,
$$
then, for all $s \in [\sigma, \sigma + l^*]$, the function $\psi(s) = K(s,\sigma)v(\sigma)$ satisfies

\begin{align}\label{lemma psi -}
\left\| \frac{\psi_-(y,s)}{1 + |y|^3} \right\|_{L^\infty} &\leq C \frac{e^{s - \sigma}((s - \sigma)^2 + 1)}{s} (|v_0(\sigma)| + |v_1(\sigma)| + \sqrt{s}|v_2(\sigma)|)\\
&+ C e^{- \frac{s - \sigma}{2}} \left\| \frac{v_-(\sigma)}{1 + |y|^3} \right\|_{L^\infty}
+ C \frac{e^{-(s - \sigma)^2}}{s^{\frac{3}{2}}} \|v_e(\sigma)\|_{L^\infty},
\end{align}

\begin{align}\label{lemma psi +}
\|\psi_e(s)\|_{L^\infty} \leq C e^{s - \sigma} \left( \sum_{l = 0}^{2} s^{l/2} |v_l(\sigma)| + s^{\frac{3}{2}} \left\| \frac{v_-(\sigma)}{1 + |y|^3} \right\|_{L^\infty} \right)
+ C e^{-\frac{s - \sigma}{p}} \|v_e(\sigma)\|_{L^\infty},
\end{align}

\begin{align}\label{lemma psi 2}
    \left|\psi_2(s)-\frac{\sigma^2}{s^2}v_2(\sigma)\right|\leq &C e^{s-\sigma} \frac{s - \sigma}{s}(|v_0(\sigma)|+|v_1(\sigma)|)+  Ce^{s-\sigma}\frac{(s-\sigma)(2+s-\sigma)}{s^2}|v_2(\sigma)|\\
    &+ C e^{s-\sigma}\frac{(s - \sigma)}{s}\left\|\frac{v_-(y,\sigma)}{1+|y|^3}\right\|_{L^\infty}+Ce^{s-\sigma}\frac{(s-\sigma)}{s^2}\|v_e (\sigma)\|_{L^\infty}.
\end{align}
\end{lemma}
\begin{rmk}
Observe that the estimate \eqref{lemma psi 2} follows naturally from the differential equation satisfied by $ \psi_2 $, namely
$$
\psi_2'(s) = -\frac{2}{s} \psi_2(s) + \text{perturbation terms},
$$
with initial condition $ \psi_2(\sigma) = v_2(\sigma) $. The solution to this equation can be expressed as
$$
\psi_2(s) = \frac{\sigma^2}{s^2} v_2(\sigma) + \text{rest terms}.
$$

\end{rmk}
\begin{proof}[Proof of Lemma \ref{lemma psi}]
The estimates \eqref{lemma psi -} and \eqref{lemma psi +} are given by \cite[Lemma 2.9]{VT2017}.

The proof of \eqref{lemma psi 2} is similar to \cite[proof of Lemma 3.5 (i)]{MZ1997Nonlinearity}, but the sake of clarity, we give here details on the differences for our case. We introduce the scalar product in $L^2_\rho(\R^N)$ as follows
$$(r_1,r_2)=\int_{\R^N} r_1(y)r_2(y)\rho(y)dy \text{ for all } r_1,r_2 \in L^2_\rho(\R^N),$$
where $\varrho$ is given right before \eqref{spec L}. By definition, we have

\begin{align}\label{bound psi 2 1}
\psi_2(s) &= ( k_2, \chi(\cdot, s) K(s, \sigma) v(\sigma) ) \\
&= \sigma^2 s^{-2} v_2(\sigma) + ( k_2, (\chi( \cdot, s) - \chi(\cdot, \sigma)) \sigma^2 s^{-2} v(\sigma)) \\
&\quad + ( k_2, \chi( \cdot, s)(K(s, \sigma) - \sigma^2 s^{-2}) v(\sigma)),
\end{align}
and we recall that $s\in [\sigma,2\sigma].$ Then, with straightforward computations, we have that
\begin{equation}\label{bound psi 2 2}
\begin{aligned}
|( k_2, (\chi(\cdot, s) - \chi(\cdot, \sigma)) \sigma^2 s^{-2} v(\sigma) )| &\leq C \frac{(s - \sigma)}{\sigma^{\frac{3}{2}}} \sigma^2 s^{-2}e^{-\frac{K_0^2}{4}\sigma}| v(\sigma)|\\
&\leq Ce^{-\frac{K_0^2}{4}\sigma} \frac{(s - \sigma)}{s^{\frac{3}{2}}}\left(\underset{i=0}{\overset{2}{\sum}}|v_i(\sigma)|+\left\|\frac{v_-(y,\sigma)}{1+|y|^3}\right\|+|v_e(y,\sigma)|\right)\\
&\leq C \frac{(s - \sigma)}{s^{2}}\left(\underset{i=0}{\overset{2}{\sum}}|v_i(\sigma)|+\left\|\frac{v_-(y,\sigma)}{1+|y|^3}\right\|+|v_e(y,\sigma)|\right).
\end{aligned}
\end{equation}
We decompose 
\begin{equation}\label{bound psi 2 3}
( k_2, \chi( \cdot, s)(K(s, \sigma) - \sigma^2 s^{-2}) v(\sigma))=\sum_{r=0}^2 b_r + b_+ + b_2,
\end{equation}
where
\begin{equation}
\begin{aligned}
&b_r(\sigma) = ( k_2, \chi( \cdot, s)(K(s, \sigma) - \sigma^2 s^{-2}) h_r ) v_r(\sigma),\\
&b_-(\sigma) = ( k_2, \chi( \cdot, s)(K(s, \sigma) - \sigma^2 s^{-2}) v_- (\sigma)),\\
&b_e(\sigma) = ( k_2, \chi(\cdot, s)(K(s, \sigma) - \sigma^2 s^{-2}) v_e (\sigma)).
\end{aligned}
\end{equation}
For $r = 0, 1$, we derive

\begin{equation}\label{b_r 0,1}
\begin{aligned}
|b_r(\sigma)| &\leq |( k_2, \chi(\cdot, s)(K(s, \sigma) - e^{(s-\sigma)\mathcal{L}} ) v_r(\sigma)|+|( k_2, \chi(\cdot, s)(e^{(s-\sigma)\mathcal{L}} - \sigma^2 s^{-2}) h_r ) v_r(\sigma)|\\
&\leq C e^{s-\sigma} \frac{(s - \sigma)}{s}|v_r(\sigma)| + C|v_r(\sigma)|\ |e^{(1-\frac{r}{2})(s-\sigma)} - \sigma^2 s^{-2} |\ |(k_2,\chi(.,s)h_r)| \\
&\leq C e^{s-\sigma} \frac{(s - \sigma)}{s}|v_r(\sigma)| + Ce^{s-\sigma}(s-\sigma)|v_r(\sigma)| e^{-\frac{K_0^2}{8}s}\\
&\leq  C e^{s-\sigma} \frac{(s - \sigma)}{s}|v_r(\sigma)|,
\end{aligned}
\end{equation}
and
\begin{equation}
\begin{aligned}\label{b_r 2}
|b_2(\sigma)| &\leq |( k_2,(K(s, \sigma) - \sigma^2 s^{-2}) h_2 ) v_2(\sigma)| +| ( k_2, (\chi( \cdot, \sigma)-1) (K(s, \sigma) - e^{(s-\sigma)\mathcal{L}}) h_2 ) v_2(\sigma)|\\
&\ +|( k_2, (\chi( \cdot, \sigma)-1) (e^{(s-\sigma)\mathcal{L}} - \sigma^2 s^{-2}) h_2 ) v_2(\sigma)|\\
&\leq  C(s-\sigma)\frac{(1+s-\sigma)}{s^2}|v_2(\sigma)|++ Ce^{s-\sigma} e^{-\frac{K_0^2}{8}s}|v_2(\sigma)|\frac{s-\sigma}{s}\\
&\ + e^{-\frac{K_0^2}{4}s}(1 - \sigma^2 s^{-2})|v_2(\sigma)|\\
&\leq  C(s-\sigma)\frac{(1+s-\sigma)}{s^2}|v_2(\sigma)|+ Ce^{s-\sigma} e^{-\frac{K_0^2}{8}s}|v_2(\sigma)|\frac{s-\sigma}{s}\\
&\leq  C(s-\sigma)e^{s-\sigma} \frac{(1+s-\sigma)}{s^2}|v_2(\sigma)|,
\end{aligned}
\end{equation}
for $s_0$ large enough and where we used \cite[Lemma 5]{BK1994} for first inequality of the first estimate and the second estimated, also the fact that $e^{(s-\sigma)\mathcal{L}} h_r = e^{(1-\frac{r}{2})(s-\sigma)} h_r$ and $( k_2, h_r ) = 0$. From similar computations as for \eqref{b_r 0,1}, we have that
\begin{align}
|b_-(\sigma)| \leq  C e^{s-\sigma} \frac{(s - \sigma)}{s}\left\|\frac{v_-(y,\sigma)}{1+|y|^3}\right\|.
\end{align}
From similar computations as in \cite[proof of Lemma 3.5 (i) for $\alpha_2$]{MZ1997DMJ} 
\begin{equation}\label{b_e}
\begin{aligned}
    |b_e(\sigma)| \leq  &|( k_2, \chi(\cdot, s)(K(s, \sigma)-e^{(s-\sigma)\mathcal{L}}) v_e (\sigma))|+ |( k_2, \chi(\cdot, s)(e^{(s-\sigma)\mathcal{L}} - 1) v_e (\sigma))|\\
    &+ |( k_2, \chi(\cdot, s)(1 - \sigma^2 s^{-2}) v_e (\sigma))|\\
    \leq &C e^{s-\sigma}e^{-\frac{K_0^2}{4}s}\frac{(s-\sigma)}{s}\|v_e (\sigma)\|+C e^{s-\sigma}e^{-\frac{K_0^2}{4}s}(s-\sigma)\|v_e (\sigma)\|\\
    &+|( k_2, (\chi(\cdot, s)-1)(1 - \sigma^2 s^{-2}) v_e (\sigma))|\\
    \leq &C e^{s-\sigma}e^{-\frac{K_0^2}{4}s}\frac{(s-\sigma)}{s}\|v_e (\sigma)\|+C e^{s-\sigma}e^{-\frac{K_0^2}{4}s}(s-\sigma)\|v_e (\sigma)\|+C e^{-\frac{K_0^2}{4}s}\frac{(s-\sigma)}{s}\|v_e (\sigma)\|\\
    \leq &Ce^{s-\sigma}\frac{(s-\sigma)}{s^2}\|v_e (\sigma)\|,
\end{aligned}
\end{equation}
for $s_0$ large enough. Thus, we combine \eqref{bound psi 2 1}, \eqref{bound psi 2 2}, \eqref{bound psi 2 3}, \eqref{b_r 0,1}, \eqref{b_r 2} and \eqref{b_e} to obtain \eqref{lemma psi 2}. This concludes the proof of Lemma \ref{lemma psi}.
\end{proof}

This yields the following result concerning the fundamental solutions of the gradient equation associated.
\begin{cor}\label{cor psi perp}
We define $K_1(s,\sigma)$ as the fundamental solution associated to the linear operator $\mathcal{L} + \mathcal{V}-\frac{1}{2}$. Then, we have the following result: For all $l^* > 0$, there exists $\sigma = \sigma(l^*)>0$ such that if $s \geq \sigma$ and $v(s)$ satisfies
$$
\sum_{m=0}^2 |v_m(s)| 
+ \left\| \frac{\nabla v_{-}(y,s)}{1+|y|^3} \right\|_{L^\infty} 
+ \|\nabla v_{e}(s)\|_{L^\infty} < +\infty,
$$
then, for all $s \in [\sigma, \sigma + l^*]$, the function $\Psi(s)=  K_1(s,\sigma)\nabla v(\sigma)$. Then, we have
\begin{align}\label{bound psi perp}
\left\|\frac{\Psi_\perp(y,s)}{1+|y|^3}\right\| &\leq C \frac{e^{\frac{s - \sigma}{2}}((s - \sigma)(1+s-\sigma) + 1)}{s} (|v_1(\sigma)| + |v_2(\sigma)|+ O(e^{-\frac{K_0^2}{8}\sigma}||v(\sigma)||_{V_A(\sigma)}))\\
&+ C \left(\frac{e^{\frac{s - \sigma}{2}}((s - \sigma)(1+s-\sigma) + 1)}{\sqrt{s}}+e^{-\frac{s - \sigma}{2}}\frac{\sigma^2}{s^2}+e^{- (s - \sigma)}\right) \left\| \frac{(\nabla v)_\perp(\sigma)}{1 + |y|^3} \right\|_{L^\infty}\\
&+ C \frac{e^{\frac{s - \sigma}{2}}(s-\sigma)}{s^2} \|(\nabla v)_e(\sigma)\|_{L^\infty}\\
\end{align}
where $\|v(\sigma)\|_{V_A(\sigma)}=\overset{2}{\underset{i=0}{\sum}}|r_i(\sigma)|+\left\|\frac{r_-(y,\sigma)}{1+|y|^3}\right\|_{L^\infty}+\|v_e(\sigma)\|$.
\end{cor}
\begin{proof}
We denote by $r=\nabla v$. Since we have \eqref{K1=exp K}, then we have 
$$\Psi(s)=e^{-\frac{s-\sigma}{2}}\psi(s),$$
where $\Psi(s)=  K_1(s,\sigma)r(\sigma)$ and $\psi(s)=K(s,\sigma)r(\sigma)$. From straightforward computation, we get
\begin{equation}\label{bound psi perp with -,2}
\left\|\frac{\Psi_\perp(y,s)}{1+|y|^3}\right\| \leq Ce^{-\frac{s-\sigma}{2}}\left(\left\|\frac{\psi_-(y,s)}{1+|y|^3}\right\|_{L^\infty}+|\psi_2(s)|\right). 
\end{equation}
For $n=0,1$, we have
\begin{align}\label{nabla v 0,1}
    (\nabla v)_n(\sigma)&=-\int v(\sigma)\nabla\chi(\sigma) k_n\varrho -\int v(\sigma) \chi(\sigma) \nabla k_n\varrho +\frac{1}{2} \int v(\sigma) \chi(\sigma) y k_n\varrho \\
    &=O(e^{-\frac{K_0^2}{8}\sigma}\|v(\sigma)\|_{V_A(\sigma)}) +(n+1) v_{n+1}.\\
\end{align}
Moreover, with straightforward computations, we have
\begin{equation}\label{bound r2 with perp}
|r_2(\sigma)|\leq C \left\|\frac{r_\perp(y,\sigma)}{1+|y|^3}\right\|,
\end{equation}
and
\begin{equation}\label{bound r- with perp}
\left\|\frac{r_-(y,\sigma)}{1+|y|^3}\right\|\leq C|r_2(\sigma)|+C \left\|\frac{r_\perp(y,\sigma)}{1+|y|^3}\right\|\leq C \left\|\frac{r_\perp(y,\sigma)}{1+|y|^3}\right\|.
\end{equation}
Then, we combine \eqref{bound psi perp with -,2} with Lemma \ref{lemma psi}, \eqref{nabla v 0,1}, \eqref{bound r2 with perp} ans \eqref{bound r- with perp} and get
\begin{align}
    e^{\frac{s-\sigma}{2}}\left\|\frac{\Psi_\perp(y,s)}{1+|y|^3}\right\| &\leq C \frac{e^{s - \sigma}((s - \sigma)(1+s-\sigma) + 1)}{s} (|v_1(\sigma)| + |v_2(\sigma)|+ O(e^{-\frac{K_0^2}{8}\sigma}||v(\sigma)||))\\
&+ C \left(\frac{e^{s - \sigma}((s - \sigma)(1+s-\sigma) + 1)}{\sqrt{s}}+\frac{\sigma^2}{s^2}+e^{- \frac{s - \sigma}{2}}\right) \left\| \frac{(\nabla v)_\perp(\sigma)}{1 + |y|^3} \right\|_{L^\infty}\\
&+ C \frac{e^{s - \sigma}(s-\sigma)}{s^2} \|(\nabla v)_e(\sigma)\|_{L^\infty},\\
\end{align}
giving us \eqref{bound psi perp}. This concludes the proof of Corollary \ref{cor psi perp}.
\end{proof}

\printbibliography
\end{document}